\documentclass[10pt]{amsart}

\usepackage{amssymb, amsmath, amsthm, amsfonts, mathtools}
\usepackage{mathrsfs}
\usepackage{graphicx}
\usepackage{rotating}
\usepackage{float}
\usepackage{caption}
\usepackage{pdflscape}
\usepackage{hyperref} 
\usepackage{url}
\usepackage[all,arc,2cell]{xy}
\UseAllTwocells
\usepackage{enumerate}
 \usepackage{lineno}
\usepackage{stix}
\usepackage{accents}
\usepackage{makecell}
\usepackage{tikz}

\newcommand{\res}{\mathrm{res}}
\newcommand{\tr}{\mathrm{tr}}

\usepackage{array}  
\newcolumntype{C}{>{$}c<{$}} 
\usepackage{wasysym}

\usepackage{tabularx}

\newcommand{\Mackey}[7]{
\xymatrix
{
#1 \ar@/_/[d]_(.5){#7}\\
#2 \ar@/_/[u]_(.5){#6} \ar@/_/[d]_(.5){#5} \\
#3 \ar@/_/[u]_(.5){#4} 
}
}

\makeatletter
\newcommand{\superimpose}[2]{%
  {\ooalign{$#1\@firstoftwo#2$\cr\hfil$#1\@secondoftwo#2$\hfil\cr}}}
\makeatother

\newcommand{\mackeycirc}{\mdwhtcircle}
\newcommand{\mackeybullet}{\mdblkcircle}
\newcommand{\mackeyblacktriangle}{\blacktriangle}
\newcommand{\mackeyblacktriangledown}{\blacktriangledown}
\newcommand{\mackeybullethat}{\widehat{\mdblkcircle}}
\newcommand{\mackeybulletoverline}{\overline{\mdblkcircle}}

\newcommand{\mackeybox}{\mdlgwhtsquare}
\newcommand{\mackeyboxoverline}{\overline{\mdlgwhtsquare}}
\newcommand{\mackeyboxhat}{\widehat{\mdlgwhtsquare}}
\newcommand{\mackeyboxhatoverline}{\widehat{\overline{\mdlgwhtsquare}}}
\newcommand{\mackeyunderlinedbullethat}{\underline{\widehat{\mdblkcircle}}}
\newcommand{\mackeyunderlinedbullettilde}{\underline{\widetilde{\mdblkcircle}}}
\newcommand{\mackeyunderlinedodothat}{\underline{\widehat{{\odot}}}}
\newcommand{\mackeyunderlinedboxhatoverline}{\underline{\widehat{\overline{\mdlgwhtsquare}}}}
\newcommand{\mackeyunderlinedboxhat}{\underline{\widehat{{\mdlgwhtsquare}}}}
\newcommand{\mackeyunderlinedboxdothat}{\widehat{\underline{\boxdot}}}
\newcommand{\mackeyunderlinedboxdottilde}{\underline{\widetilde{\boxdot}}}
\newcommand{\mackeyunderlinedboxleftslashhat}{\underline{\widehat{\squarellblack}}}
\newcommand{\mackeyunderlinedboxleftslashtilde}{\underline{\widetilde{\squarellblack}}}
\newcommand{\mackeyboxleftslashoverline}{\overline{\squarellblack}}
\newcommand{\mackeyunderlinedboxblackboxhat}{\underline{\widehat{\blackinwhitesquare}}}
\newcommand{\mackeyunderlinedboxboxhat}{\underline{\widehat{\boxbox}}}

\newcommand{\mackeyboxleftslash}{\squarellblack}

\newcommand{\mackeyboxleftslashhat}{\widehat{\squarellblack}}

\newcommand{\mackeyunderlinedhalftopboxhat}{\underline{\widehat{\squaretopblack}}}
\newcommand{\mackeyhalftopbox}{\squaretopblack}

\newcommand{\mackeyunderlinedblacktrianglehat}{\underline{\widehat{\blacktriangle}}}

\newcommand{\mackeyunderlinedboxboxtilde}{\underline{\widetilde{\boxbox}}}

\newcommand{\mackeyunderlinedcirccirchat}{\underline{\widehat{\circledwhitebullet}}}

\newcommand{\mackeyunderlinedhalfbottomboxhat}{\underline{\widehat{\squarebotblack}}}

\newcommand{\mackeyeinffourethy}{\underline{  \widecheck{\squarellblack} } }
\newcommand{\mackeyeinften}{\widehat{\dot{\mdlgwhtsquare}}}
\newcommand{\mackeyeinftenethy}{\underline{\widehat{\dot{\mdlgwhtsquare}}}}

\newcommand{\mackeyeinfeighteenethy}{\underline{\widetilde{\dot{\mdlgwhtsquare}}}}

\newcommand{\mackeyeinftwentyethy}{\underline{\widehat{\mathring{\squarellblack}}}}

\newcommand{\mackeyeinftwoethy}{\underline{  \widecheck{\boxcircle} } }

\newcommand{\mackeyunderlinedboxdotcheck}{\underline{  \widecheck{\boxdot} } }

\newcommand{\mackeyeinfsix}{\blackhourglass}
\newcommand{\mackeyeinfsixethy}{\underline{\widehat{\overline{\blackhourglass}}}}

\newcommand{\mackeyeinftwentytwoethy}{\underline{\widehat{\overline{\boxcircle}}}}

\newcommand{\mackeyeinftwentyeightethy}{\underline{\widetilde{\dot{\squarellblack}}}}

\newcommand{\mackeysigmaminusoneethy}{\underline{\widetilde{\dot{\boxdot}}}}

\newcommand{\mackeysigmaminusnineethy}{\underline{\widetilde{\dot{\boxbox}}}}

\newcommand{\mackeyunderlinedblackboxwhitecirchatoverline}{\underline{\widehat{\overline{\inversebullet}}}}

\newcommand{\mackeyeinffoursigmaethy}{\breve{\overline{\squarellblack} }}

\newcommand{\mackeyunderlinedhalftopboxcheck}{\underline{\widecheck{\squaretopblack}}}

\newcommand{\mackeypiccirc}{{\color{gray}\mackeycirc}}
\newcommand{\mackeypicalgfixpoints}{\boxtimes}

\newcommand{\mackeypictriangledown}{{\color{gray}\mackeyblacktriangledown}}
\newcommand{\mackeypictriangle}{{\color{gray}\mackeyblacktriangle}}
\newcommand{\mackeypicbullet}{{\color{gray}\mackeybullet}}
\newcommand{\mackeypicbulletoverline}{{\color{gray}\mackeybulletoverline}}
\newcommand{\mackeypicbullethat}{{\color{gray}\mackeybullethat}}
\newcommand{\mackeypiceinften}{{\color{gray}\mackeyeinften}}
\newcommand{\mackeypicboxhatoverline}{{\color{gray}\mackeyboxhatoverline}}

\newcommand{\MUR}{MU_{\R}}
\newcommand{\MUC}{MU^{(\!(C_4)\!)}}
\newcommand{\wmu}{A}
\newcommand{\wmum}{\wmu_-}
\newcommand{\wmuP}{\wmu(+)}
\newcommand{\wmuM}{\wmu(-)}

\newcommand{\mymu}{\mu}

\usetikzlibrary{decorations.pathmorphing,shapes}
\newcounter{sarrow}

\usepackage{hyperref}
  \definecolor{dark-red}{rgb}{0.6,0.15,0.15}
   \definecolor{dark-blue}{rgb}{0.15,0.15,0.6}
   \definecolor{medium-blue}{rgb}{0,0,0.5}

\hypersetup{
    colorlinks, 
    linkcolor=dark-red,
    citecolor=dark-blue, urlcolor=medium-blue
}

\usepackage[nameinlink,capitalise,noabbrev]{cleveref}

\newtheorem{thm}{Theorem}[section]
\newtheorem{cor}{Corollary}[section]
\newtheorem{prop}{Proposition}[section]
\newtheorem{lem}{Lemma}[section]
\theoremstyle{definition}
\newtheorem{defn}{Definition}[section]

\newtheorem{rem}{Remark}[section]

\newtheorem{notation}{Notation}[section]

\makeatletter
\let\c@lem=\c@thm
\let\c@cor=\c@thm
\let\c@prop=\c@thm
\let\c@lem=\c@thm
\let\c@defn=\c@thm
\let\c@exmps=\c@thm
\let\c@rem=\c@thm
\let\c@warn=\c@thm
\let\c@claim=\c@thm
\let\c@quest=\c@thm
\let\c@notation=\c@thm
\let\c@note=\c@thm
\makeatother

\numberwithin{equation}{section}
\numberwithin{figure}{section}
\numberwithin{table}{section}

\DeclareMathOperator{\Pic}{Pic}
\newcommand{\mPic}{\underline{\Pic}}

\newcommand{\PicRdashG}[2]{\Pic_{#2}(#1)}
\newcommand{\PichG}[2]{\Pic(#1^{h#2})}
\newcommand{\PicRG}[2]{\PichG{#1}{#2}}
\newcommand{\PicalgRG}[2]{\Pic_{#2}(#1) }

\newcommand{\m}[1]{\underline{#1}}

\newcommand{\Z}{\mathbb{Z}}

\newcommand{\R}{\mathbb{R}}

\newcommand{\F}{\mathbb{F}}

\newcommand{\G}{\mathbb{G}}

\newcommand{\W}{\mathbb{W}}

\newcommand{\subpic}{\varspadesuit}
\DeclareRobustCommand\spi{\underaccent{\bar}{\pi}}

\DeclareSymbolFontAlphabet{\scr}{rsfs}

\newcommand{\smsh}{\wedge}

\newcommand{\xra}{\xrightarrow}

\def\quickop#1{\expandafter\newcommand\csname #1\endcsname{\operatorname{#1}}}
\quickop{Hom} \quickop{End} \quickop{Aut} \quickop{Tel} \quickop{Mic} 
\quickop{Ext} \quickop{Tor} \quickop{Cotor} \quickop{Id} \quickop{Coker} \quickop{Ker}
\quickop{Lim} \quickop{Colim} \quickop{Holim} \quickop{Hocolim}
\quickop{id} \quickop{tel} \quickop{mic} \quickop{coker} 
\quickop{colim} \quickop{holim} \quickop{hocolim} \quickop{im}
\DeclareMathOperator{\Gal}{Gal}
\DeclareMathOperator{\Mod}{Mod}

\DeclareMathOperator{\Ind}{Ind}

\DeclareMathOperator{\Sym}{Sym}

\DeclareMathOperator{\RO}{RO}

\newcommand{\pic}{\mathrm{pic}}

\newcommand{\ZZ}{\mathbb{Z}}

\newcommand{\dfrak}{\bar{\mathfrak{d}}_1}

\DeclareMathOperator{\Sp}{Sp}

\newcommand{\JEG}{J_{\scriptscriptstyle E}^{\scriptscriptstyle G}}
\newcommand{\JRG}{J_{\scriptscriptstyle R}^{\scriptscriptstyle G}}
\newcommand{\JEC}[1]{J_{\scriptscriptstyle E}^{\scriptscriptstyle C_{#1}}}

\newcommand{\fk}{\mathbb{k}}

\author[Beaudry]{Agn\`es Beaudry}
\address{Department of Mathematics, University of Colorado Boulder, Campus Box 395, Boulder, CO, 80309-0395}
\author[Bobkova]{Irina Bobkova}
\address{Department of Mathematics, Texas A\&M University, Mailstop 3368, College Station, TX 77843}
\author[Hill]{Michael Hill}
\address{Department of Mathematics, University of California, Los Angeles,
Box 951555, Los Angeles, CA, 90095-1555}
\author[Stojanoska]{Vesna Stojanoska}
\address{Department of Mathematics, University of Illinois, Urbana-Champaign, 273 Altgeld Hall
1409 W. Green Street, 
Urbana, IL 61801}

\subjclass[2010]{55M05, 55P42, 20J06, 55Q91, 55Q51, 55P60}

\title{Invertible $K(2)$-Local $E$-Modules in $C_4$-Spectra}
\date{\today}

\begin{document}

\begin{abstract}
We compute the Picard group of the category of $K(2)$-local module spectra over the ring spectrum $E^{hC_4}$, where $E$ is a height 2 Morava $E$-theory and $C_4$ is a subgroup of the associated Morava stabilizer group. This group can be identified with the Picard group of $K(2)$-local $E$-modules in genuine $C_4$-spectra. We show that in addition to a cyclic subgroup of order 32 generated by $ E\wedge S^1$ the Picard group contains a subgroup of order 2 generated by $E\wedge S^{7+\sigma}$, where $\sigma$ is the sign representation of the group $C_4$. In the process, we completely compute the $RO(C_4)$-graded Mackey functor homotopy fixed point spectral sequence for the $C_4$-spectrum $E$.
\end{abstract}
\maketitle
\setcounter{tocdepth}{1}
\tableofcontents


\section{Introduction}

Starting with the computations by Hopkins and Mahowald of the homotopy groups of the spectrum of topological modular forms, there has been a proliferation of computations related to various invariants associated to the homotopy fixed points for the action of finite subgroups of the Morava stabilizer group $\G=\G_n$ on Morava $E$-theory $E=E_n$. 
In theory, one would like to understand $E^{h\G}$ itself as this is a model for the $K(n)$-local sphere. 
At heights $n> 2$, there has been little progress in this direction, and the case $n=p=2$ remains the focus of current research.

If one restricts to a finite subgroup $G$ of $\G$, questions about $E^{hG}$ and its module category become more tractable using techniques introduced by Hopkins and Miller. Among other things, the methods of Hopkins--Miller allow one to understand $E_*$ as a $G$-module.
These ideas first appeared in print in Nave \cite{nave2}. They were generalized by Hill, Hopkins and Ravenel in unpublished work and also in \cite{HHR}.
The whole program has recently been enhanced by Hahn and Shi \cite{hahnshi}. 

Roughly, one fixes a real orientation
\[\MUR \to E\]
where $\MUR$ is the real bordism spectrum. This map is required to be $C_2$-equivariant, where $C_2$ has the usual action on $\MUR$ coming from complex conjugation, while on $E$ it acts through the central $\{ \pm 1\} \subseteq \G$.
Then for any subgroup $G$ of $\G$ containing $C_2=\{\pm 1\}$, one can use the norm to construct a map
\[N_{C_2}^GMU_{\R} \to E,\]
which provides information on the structure of $E$ as a $G$-spectrum.

An important application of these ideas is the computation of the homotopy groups $\pi_*E^{hG}$, which in turn, provides the information needed to study the Picard group of the category of $K(n)$-local $E^{hG}$-module spectra, $\PichG{E}{G}$. Recall that, for a symmetric monoidal category, the Picard group consists of isomorphism classes of invertible objects with respect to the symmetric monoidal product, whenever this forms a set. The problem of computing $\PichG{E}{G}$ appears in various forms throughout the literature. 
Classical examples are the folklore results of Hopkins which state that $\Pic(KU) \cong \Z/2$ and $\Pic(KO) \cong \Z/8$ (see also Gepner--Lawson \cite{GepnerLawson}) and the fact that $\Pic(E_n)\cong \Z/2$ for all $n$, a result of Baker--Richter \cite{BakerRichter}. The problem has been extensively revisited by Heard--Mathew--Stojanoska \cite{MathewStojanoska,HeMaSt}. For example, they compute $\PichG{E_n}{G}$ for the finite subgroups $G\subseteq \G$ at chromatic heights $n=p-1$ for $p$ odd. Most recently, Heard--Li--Shi \cite{heardlishi} have also computed $\Pic(E_n^{hC_2})$ at all heights $n$ when $p=2$.

The Picard group of a ring spectrum always contains a cyclic subgroup generated by the suspension $\Sigma E^{hG}$. In all of the examples of $\PichG{E}{G}$ studied so far, it was found that $\PichG{E}{G}$ is exactly this cyclic group. For instance, if $n=p-1$ and $G$ is a maximal finite subgroup of $\G$ containing the $p$-torsion, $\PichG{E_n}{G}$ is cyclic of order $2n^2p^2$, which is the periodicity of $E_n^{hG}$ \cite{HeMaSt}.
When $p=2$, then $\PichG{E_n}{C_2}$ is cyclic of order $2^{n+2}$ \cite{heardlishi}, which is again the periodicity of $E_n^{hC_2}$.

In this paper, we compute an example of $\PichG{E}{G}$ which is not a cyclic group. The example is the following. We work at the prime $p=2$  and chromatic height $n=2$. 
Any formal group law $\Gamma$ of height $2$ has an automorphism $\gamma$ of order $4$ over the algebraic closure of $\F_2$ and the subgroup $C_4$ this automorphism generates is unique up to conjugation in the associated Morava stabilizer group. 
The spectrum $E^{hC_4}$ has already received much attention in the literature. It is closely related to 
the spectrum $\mathrm{TMF}_0(5)$ of topological modular forms with $\Gamma_0(5)$ level structure.  The latter was studied extensively by Behrens--Ormsby in \cite{BO}. 
Further, $E$ as a $C_4$-module spectrum is closely related to the spectrum $K_{[2]}$ studied in Hill--Hopkins--Ravenel \cite{HHRC4}. It also play a key role in Bobkova--Goerss \cite{BobkovaGoerss} and in Henn \cite{henn_centr}. In fact, the computation of the homotopy fixed point spectral sequence for the spectrum $E^{hC_4}$ can be mined from these references and its homotopy groups are now well understood.
The spectrum $E^{hC_4}$ is $32$-periodic, so $\PichG{E}{C_4}$ necessarily contains a cyclic group of that order. However, in this case, it turns out that the Picard group also contains elements which are not suspensions of $E^{hC_4}$.

 After replacing $E$ by an equivalent cofree $G$-spectrum, the fact that $E^{hG} \to E$ is a faithful $K(n)$-local Galois extension of $G$-spectra implies that the homotopy category of $K(n)$-local $E^{hG}$-module spectra is equivalent to the homotopy category of $K(n)$-local $E$-modules in genuine \(G\)-spectra. The latter is the category of genuine $G$-spectra with a compatible $E$-module structure, whose Picard group we denote by $\PicRdashG{E}{G}$. This allows us to identify the Picard groups  
\[
\PichG{E}{G}\cong \PicRdashG{E}{G}.
\]
This translates our problem into that of computing the Picard group of the $C_4$-equivariant ring spectrum \(E\). 

In general, if \(R\) is a \(G\)-equivariant commutative ring spectrum, then the groups 
\[
\PicRdashG{R}{H}=\PicRdashG{i_{H}^{\ast}R}{H},
\]
as $H$ runs through the subgroups of $G$, assemble into a Mackey functor which we will denote by \(\mPic(R)\). The restriction maps come from the ordinary restriction functors in genuine \(G\)-spectra. Since these are strong symmetric monoidal functors, they induce homomorphisms on \(\Pic\). The transfer maps are given by the norm maps in the category of \(R\)-modules. These are also strong symmetric monoidal functors, so they induce homomorphisms on \(\Pic\). The Mackey compatibility is inherited from the corresponding statements in the homotopy category of \(R\)-modules. 

Our main result is then the following theorem.
\begin{thm}\label{thm:fullpicmackey}
Let $\Gamma$ be a formal group law of height $2$ over $\F_2$ and let $\fk$ be the algebraic closure of $\F_2$. Let $E=E(\fk, \Gamma)$ be the associated Morava $E$-theory and $\G=\G(\fk, \Gamma)$ the associated Morava stabilizer group. Let $C_4\subseteq \G$ be a cyclic subgroup of order $4$, which necessarily contains $C_2 =\{\pm 1\}$.
Then there are isomorphisms
\begin{align*}
\PicRdashG{E}{C_4}\cong \Z/32\{ E \smsh S^1\} \oplus \Z/2\{ E \smsh S^{7+\sigma}\}
\end{align*}
and 
\begin{align*}
\PicRdashG{E}{C_2}\cong \Z/16 \{E \smsh S^1\}.
\end{align*}
As a Mackey functor, this assembles into
\[\Mackey{  \mPic(E)(C_4/C_4) \cong \Z/32 \oplus \Z/2}{  \mPic(E)(C_4/C_2) \cong \Z/16}{  \mPic(E) (C_4/\{e\}) \cong \Z/2 .}{0}{1}{\left[ \begin{matrix} 26 \\ 1 \end{matrix}\right]}{\left[ \begin{matrix}1 & 8 \end{matrix}\right]}\]
\end{thm}

The result (and its proof) for $E^{hC_2}$ is a special case of  \cite{heardlishi}. We include the computations here since they are necessary for our analysis of the Mackey functor $\mPic(E)$. Together with \cite[Proposition 3.10]{HeMaSt}, \Cref{thm:fullpicmackey} has the following immediate consequence:
\begin{cor}
Let $C_6\subseteq \G$ be the subgroup generated by $-1$ and a third root of unity.
There is an isomorphism
$\PichG{E}{C_6}\cong \Z/48$.
\end{cor}

Our approach is to study the map $J_R^{G} \colon RO(G) \to \PicRdashG{R}{G}$, for $R$ an \(E_{\infty}\)-ring $G$-spectrum, given by
\[
J_{R}^{G}(V)=R\wedge S^{V}.
\]
When \(R\) is a \(G\)-equivariant commutative ring spectrum, then these homomorphisms assemble into a map of Mackey functors
\[
\m{J}_{R}\colon\m{\RO}\to\mPic(R),
\]
where here \(\m{\RO}\) is the representation ring Mackey functor. We determine the image of \(\m{J}_{E}\), which gives us lower bounds for the Picard groups.

To prove that these are also upper bounds, we use the Picard homotopy fixed point spectral sequence. In fact, we need the structure of this spectral sequence as a spectral sequence of Mackey functors for a crucial step in the argument. On the way, as an input to the $E_2$-term of this spectral sequence, we also compute the \emph{algebraic Picard groups} $\PicalgRG{E_0}{C_4}$ and $\PicalgRG{E_0}{C_2}$. These also naturally assemble into a coefficient system, where the value at \(G/H\) is the Picard group of the category of \(i_{H}^{\ast}E_{0}\)-modules in \(H\)-modules. This coefficient system can be computed via the isomorphism
\[
\mPic(E_{0})(G/H) \cong H^1(H;i_{H}^{\ast}E_0^{\times}).
\]
The restriction maps are determined by naturality for group cohomology, while the transfer maps are given on representing modules by tensor induction in the category of \(E_{0}\)-modules.

We find that
$\PicalgRG{E_0}{G}$ are cyclic groups of order $4$ when $G=C_4$ and $2$ when $G=C_2$. Note that in both cases, $2|\PicalgRG{E_0}{G}|$ is the periodicity of the cohomology $H^*(G, E_t)$ in $t$, which in turn is the size of $\PicalgRG{E_*}{G}$.

A key input to our computations is the knowledge of the homotopy fixed point spectral sequence of Mackey functors computing $\spi_{\star}E$. The ingredients for such a computation appear in various places in the literature. In particular, many of the pieces necessary to do this computation appear in \cite{HHRC4}. In \Cref{sec:cohomology}, we describe this computation. Many results and much notation from \Cref{sec:cohomology} are used in proofs throughout the paper, but we have attempted to keep the narrative as free of this dependence as possible. Note in passing that these computations together with our result on the Picard group give a complete description of the $RO(C_4)$-graded Mackey functor $\spi_{\star}E$ (\Cref{rem:allcompsintwo}).

\subsection*{Organization} 
This is a brief outline of the paper. In \Cref{sec:Z2}, we discuss the map $J_R^G$. In \Cref{sec:Etheorybig} compute its image in the cases of interest. This gives a lower bound on the order of $\PichG{E}{C_4}$ and, in particular, proves that this group is not cyclic. In \Cref{sec:computations}, we review the computations of the homotopy fixed point spectral sequences needed for the rest of the paper.
In \Cref{sec:alg}, we compute the Mackey functor $\underline{H}^1(C_4, E_0^{\times})$. In \Cref{sec:pic}, we discuss some equivariant properties of the Picard spectral sequence and then use this spectral sequence to give the upper bounds on $\PichG{E}{C_4}$ and $\PichG{E}{C_2}$. 

\subsection*{Acknowledgements}
We would like to thank Mike Hopkins, Lennart Meier, Doug Ravenel, XiaoLin Danny Shi and Mingcong Zeng for useful conversations. We would like to thank the anonymous referee for helpful comments. This material is based upon work supported by the National Science Foundation under grants No.~DMS--1725563, DMS--1638352, DMS--1811189, DMS--2005627 and DMS--1812122.
The authors would like to thank the Isaac Newton Institute for Mathematical Sciences for support and hospitality during the program Homotopy Harnessing Higher Structures when a large part of writing this paper was undertaken. This work was supported by EPSRC Grant Number EP/R014604/1.

\section{Preliminaries}\label{sec:prelim}
\subsection{Equivariant homotopy theory}\label{sec:eqhomotopytheory}
We review some notation and results from equivariant homotopy theory that will be used throughout the paper. For a finite group $G$, we will be working in the category of genuine $G$-spectra as in \cite{HHR}.

As usual, $RO(G)$ denotes the ring of real orthogonal virtual representations of the group $G$. For  a genuine $G$-spectrum $X$, its equivariant homotopy groups assemble into an $RO(G)$-graded Mackey functor ${\spi}_{\bigstar}X$ given by
\[
{\spi}_{V}(X)(G/H) = \pi_{V}^H(X) =[S^V, X]^H.
\]
Here $V \in RO(G)$ and $[S^V, X]^H$ denotes the genuine $H$-equivariant homotopy classes of maps. We simply write 
\[\pi_V X = [S^V, X]^e \]
in the case of the trivial group $H=e$. The conjugation action of $G$ on homotopy classes of maps $[S^V, X]$ induces an action of $G$ on $\pi_V X$. For $d\in \Z$, we may also use the notation $\pi_d i_e^*X = [S^d, i_e^*X]$ when we want to stress the fact that we are considering the homotopy groups of the underlying spectrum $i_e^*X$.

Spectra like the Morava $E$-theory spectra arise naturally not as genuine \(G\)-spectra but rather as \(G\)-objects in the category of spectra. We have a homotopically meaningful way to lift these \(G\)-objects in spectra to genuine \(G\)-spectra, however.  The \emph{cofree localization}
\[
R\mapsto F(EG_{+},R)
\]
takes equivariant maps which are underlying weak equivalences to genuine equivariant equivalences. As such, we can view it as a functor 
\[
F(EG_{+},-)\colon \Sp^{BG}\to \Sp^{G}
\]
from \(G\)-objects in spectra to genuine \(G\)-spectra \cite[\S 2.2.1]{hill_meier}. Moreover, this functor is lax symmetric monoidal, and hence takes \(E_{\infty}\)-ring spectra on which \(G\) acts via \(E_{\infty}\)-ring maps to \(E_{\infty}\)-ring objects in genuine \(G\)-spectra. In fact, if \(R\) is an \(E_{\infty}\)-ring spectrum in \(\Sp^{BG}\), then the cofree spectrum \(F(EG_{+},R)\) inherits an action of a \(G\)-\(E_{\infty}\)-operad, and hence we have all norms \cite[Theorem 2.4]{hill_meier}. Since \(G\)-\(E_{\infty}\)-ring spectra are equivalent to commutative ring objects in genuine \(G\)-spectra, we can therefore view any of these as equivariant commutative ring spectra.

\begin{rem} \label{rem:promotingE}
If $\Gamma$ is a formal group law of height $n$ over a perfect field $k$ of characteristic $p$ such that $G \subseteq \Aut_{k}(\Gamma)$ and \(E(k, \Gamma)\) is the associated Morava $E$-theory spectrum, then by the Goerss--Hopkins--Miller theorem, $E(k, \Gamma)$ is an $E_{\infty}$-ring and the action of $G$ is by $E_{\infty}$-ring maps. So this allows us to view \(E(k, \Gamma)\) as a commutative ring object in genuine \(G\)-equivariant spectra.
\end{rem}

\begin{notation}
If \(X\) is a spectrum with \(G\)-action, then let
\(X^h=F(EG_+,X)\).
For a subgroup $K$ of $G$, the homotopy fixed point spectrum $X^{hK}$ is just the \(K\)-fixed points of \(X^h\). 
\end{notation}

\subsection{The Mackey functor homotopy fixed point spectral sequences}\label{sec:hfpssmackey}

Next, we recall the setup for working with a full Mackey functor of homotopy fixed point spectral sequences (Mackey HFPSS). We work in the same context as in \cite{hill_meier}.

For any $G$-module $M$, we let $\underline{H}^*(G, M)$ be the Mackey functor determined by
\[
\underline{H}^*(G, M)(G/K) = {H}^*(K, i_K^\ast M) 
\]
with the standard group cohomology restrictions and transfers.

As in the construction of the classical homotopy fixed point spectral sequence, the filtration on $F(EG_+,X)$ induced from the skeletal filtration of $EG \simeq \varinjlim EG^{(\bullet)}$ gives rise to
an $RO(G)$-graded 
spectral sequence of Mackey functors \cite[Proposition 2.8]{hill_meier} with
\[
 \underline{E}_2^{s,V} =\underline{H}^s(G,\pi_V X)    \Longrightarrow {\spi}_{V-s}X^{h}.\]
Here, $s \in \Z_{\geq 0}$ and $V \in RO(G)$. The differentials 
\[d_r = d_{r}^{s,V} \colon \underline{E}_r^{s,V} \to \underline{E}_r^{s+r,V+r-1}\] 
commute with the restrictions and transfers.
In particular, if we restrict $V$ to the trivial representations, evaluating the Mackey HFPSS at $G/H$ recovers the standard homotopy fixed point spectral sequence computing $\pi_*(X^{hH})$. 
If $X$ is a ring spectrum, then the spectral sequence is multiplicative. 

\begin{rem}\label{rem:slicess}
The homotopy fixed point spectral sequence can also be constructed using the filtration of $F(EG_+, X)$ induced by the slice filtration $X \simeq \varprojlim P^{\bullet}X$. Ullman shows that there is a natural map
\[ 
P^{\bullet}X\to F(EG_+, P^{\bullet}X),
\]
and this induces a map of spectral sequences from the slice spectral sequence of $X$ 
\[ 
E_2^{s,V} \cong \spi_{V-s} P_{d}^{d} X  \Longrightarrow {\spi}_{V-s} X 
\]
(where $d=\dim(V)$) to the homotopy fixed point spectral sequence of $X$. If $X$ is a ring spectrum, this is a map of multiplicative spectral sequences \cite[Theorem I.9.1]{ullman}.
\end{rem}

\subsection{Real Representations of $C_2$ and $C_4$}\label{sec:realreps}

We recall the structure of the representation rings $RO(G)$, when $G$ is $C_4$ or $C_2$, establishing some notation for the rest of the paper.
Let $\gamma=\gamma_4$ be a generator of $C_4$ and $\gamma_2$ a generator of $C_2$. These are the real representations of interest:
\begin{itemize}
\item The trivial representation $1$;
\item The sign representation $\sigma$ of $C_4 $ (on which $\gamma$ acts as $-1$);
\item The sign representation $\sigma_2$ of $C_2 $ (on which $\gamma_2$ acts as $-1$);
\item The two-dimensional irreducible $C_4$-representation $\lambda$; this is $\R^2$ on which $\gamma$ acts by rotation by $\pi/2$;
\item The regular representation of $C_2$, $\rho_2=1+\sigma_2$; and
\item The regular representation of $C_4$, $\rho = \rho_4=1+\sigma + \lambda$.
\end{itemize}
Then we have 
\begin{align*}
RO(C_2) &\cong \Z[\sigma_2]/(\sigma_2^2-1), \text{ and}\\
RO(C_4) &\cong \Z[\sigma, \lambda]/(\sigma^2-1, \sigma\lambda -\lambda, \lambda^2 -2-2\sigma).
\end{align*}

We also use the following standard notation throughout. See for example Definition 3.4 of \cite{HHRC4}.
\begin{notation}\label{notn:aV}
For $V \in RO(G)$, the inclusion $S^0 =\{0,\infty\} \to S^V $ is an equivariant map denoted by $a_V \in \pi_{-V}^{G}S^0$. If $R$ is a ring, the image of $a_V$ under the unit in $\pi_{-V}^{G}R$ is also denoted by $a_V$.
\end{notation}


\section{The homomorphism $\JRG$}\label{sec:Z2}

Let \(\mathcal O\) be an \(E_{\infty}\)-operad, and suppose that $R$ is an \(\mathcal O\)-algebra in genuine equivariant spectra. 
There is a good symmetric monoidal category of \(R\)-modules in genuine \(G\)-spectra \cite{BHGSym}, and base-change along the unit map
\[
S^{0}\to R
\] 
gives us a group homomorphism
\[
\JRG \colon RO(G) \to \Pic_{G}(R),
\]
where \(\Pic_{G}(R)\) is the Picard group of the category of \(R\)-modules in genuine \(G\)-spectra. Explicitly, this homomorphism is given by
\[
\JRG(V) = R\smsh S^{V},
\]
and since the target depends only on the equivariant equivalence class of the \(V\)-sphere, this factors through the \(JO\)-equivalence classes of representations (i.e., equivalences of associated spherical bundles). This gives us the name.

To apply this for spectra like Morava $E$-theory, we must promote this naive equivariant commutative ring spectrum to a genuine equivariant ring spectrum.
The formal procedure to do this was described in \Cref{sec:eqhomotopytheory} and has nice properties for \(\Pic\).

For $R$ a ring spectrum, let $\Mod(R)$ be the homotopy category of $R$-modules in spectra and let $\Mod_G(R)$ be the homotopy category of $R$-modules in the category of $G$-spectra. The proof of the following result is the discussion immediately following the statement of Theorem 6.4 in \cite{hill_meier}.

\begin{prop}\label{prop:equivalencemodules}
Let $R^{hG} \to R$ be a faithful $G$-Galois extension, where $R$ is a cofree $G$ ring spectrum. Then $\Mod(R^{hG})$ and $\Mod_G(R)$ are equivalent categories. In particular, $\Pic(R^{hG}) \cong \Pic_G(R)$.
\end{prop}
\begin{cor}\label{cor:MN}
If $M$ and $N$ are $R$-modules for $R$ as in \Cref{prop:equivalencemodules}, then $M^{hG} \simeq N^{hG}$ if and only if $M \simeq N$ as $R$-modules in $G$-spectra. In particular, $M^{hG}$ and $N^{hG}$ represent the same elements in $\Pic(R^{hG})$ if and only if there is a $G$-equivariant equivalence of $R$-modules $M \simeq N$.
\end{cor}

We will use the following notation. If $X$ is a spectrum and $u\colon X \to M$ is a map where $M$ is an $R$-module spectrum, we let $u^R$ be the composite
\[ 
R \smsh X \xra{R \smsh u}  R \smsh M \xra{ \mu_M}  M
\]
where $\mu_M \colon  R \smsh M \to M$ is the module structure map. In particular, we can apply this construction for any map $X \to R$ using the $R$-module structure on $R$ given by multiplication $\mu \colon R \smsh R \to R$.

The group $\Pic_{G}(R)$ contains a cyclic subgroup generated by $\Sigma R$, or equivalently the image $\JRG(\Z\{1\})$ of $\Z\{1\} \subseteq RO(G)$, where $1$ denotes the trivial one-dimensional representation. We let $d\in \Pic_{G}(R)$ denote the element  $\Sigma^d R$.

\begin{defn}\label{defn:Rorientable}
Let $V$ be a $G$ representation of dimension $d$. An \emph{$R$-orientation $u_V$ for $V$} is a $G$-equivariant map 
\[
u_V \colon S^d \to  R\smsh S^V
\]
such that $u_V^R \colon R\smsh S^d \to R\smsh S^V$ is a $G$-equivalence. 

A representation $V$ is $R$-orientable if there exists an $R$-orientation for $V$.
\end{defn}

\begin{rem}\label{rem:uVR}
Given a $G$-equivalence $u_V^R \colon R\smsh S^d \to R\smsh S^V$ which is also a map of $R$-modules, we can precompose $u_V^R$ with $1\smsh S^d \colon S^0 \smsh S^d \to R\smsh S^d$ where $1$ is the unit of $R$ to obtain an orientation $u_V$.
 \end{rem}
 
By construction, weak equivalences between cofree spectra are detected on the underlying, non-equivariant homotopy. This gives us a way to detect \(R\)-orientability.

\begin{prop}
If $R$ is cofree and there is an element $u_V \in \pi_d^{G}( R \smsh S^V)$ such that $u_V^R$ induces an underlying equivalence, then $V$ is $R$-orientable.
\end{prop}

\begin{prop}\label{lem:JGRorientable}
Let $V$ be a $G$ representation of dimension $d$, and assume that $R$ is cofree. 
Then the representation $V$ is $R$-orientable if and only if  $\JRG(V)=d$ in $\Pic(R)$.
\end{prop}

Therefore, to compute the image of the map $\JRG$, it is useful to have a criterion for recognizing $R$-orientable representations. 

Now, let $V$ be a $d$-dimensional real representation of $G$. Recall that $V$ is \emph{orientable} in the classical sense if $G \to O(d) \cong \Aut(V)$ factors through $SO(d)$. Since $SO(d)$ is path connected, for any orientable representation $V$ of dimension $d$, the action of $g\in G$ is homotopic to the identity on $S^V$. It follows that there is an equivariant isomorphism of \(\pi_{\ast}i_e^*R\)-modules
\begin{equation}\label{eq:isohomotopy}
\pi_* (i_e^*R \smsh S^d ) \cong \pi_*i_e^*(R \smsh S^V) .
\end{equation}
Since the source is a free \(\pi_{\ast}i_e^*R\)-module on a fixed class in dimension \(d\) (corresponding to the element \(1\) in \(\pi_{0}i_e^*R\)), the isomorphism is equivalent to an element 
\[
u_{V}\in \big(\pi_{d}i_e^*(R\smsh S^{V})\big)^{G}.
\]
However, this is not the same as having an equivariant map
\[
S^{d}\to R\smsh S^{V},
\]
since the homotopy groups we are computing are just the underlying homotopy classes of maps. 
Even if $R$ is cofree, this does not imply that $V$ is $R$-orientable. 
\begin{defn}\label{defn:uVnonequivariant}
Let $V$ be a (classically) orientable representation of dimension $d$. A \emph{pseudo $R$-orientation} is a map of underlying spectra
\[
u_V \in\big( \pi_{d}i_e^*(R \smsh S^V) \big)^{G}
\]
such that $u_V^{i_e^*R} \colon i_e^* R \smsh S^d \to i_e^*(R \smsh S^V)$ induces an isomorphism
\[
\pi_*(i_e^*R \smsh S^d) \to \pi_* i_e^*(R \smsh S^V)
\]
of $G$-$\pi_*i_e^*R$-modules.  
\end{defn}

\begin{rem}\label{rem:pseudocofree}
In fact, if $R$ is cofree, a pseudo $R$-orientation 
\[
u_V \colon S^{d}\to R\smsh S^{V}.
\] 
is an $R$-orientation if and only if $u_V$ underlies a $G$-equivariant map of $G$-spectra.
\end{rem}

\begin{rem}\label{rem:units}
By construction, the pseudo $R$-orientations are units in the ring $(\pi_{\star}i_e^*R)^G$.
\end{rem}

\begin{prop}
Let $V$ be a classically orientable $G$ representation of dimension $d$, and let $R$ be cofree.
Then $V$ is $R$-orientable if and only if there exists a pseudo $R$-orientation 
\[
u_V \in \big( \pi_{d}i_e^*(R\smsh S^{V})\big)^{G} \cong H^0(G, \pi_{d}(R\smsh S^{V}))
\] 
which is a permanent cycle in the homotopy fixed point spectral sequence
\[ 
H^s(G, \pi_{t}(R\smsh S^V)) \Longrightarrow \pi_{t-s}(R\smsh S^V)^{hG}.
\]
\end{prop}
\begin{proof}
If $V$ is $R$-orientable, then the map 
$u_V^R \colon R \smsh S^d \to R \smsh S^V$ 
induces an isomorphism of spectral sequences which maps $ 1 \in H^0(G, \pi_0R) \cong H^0(G, \pi_d(R\smsh S^d))$ to $ u_V \in H^0(G, \pi_d (R\smsh S^V))$.  Since $1$ is a permanent cycle, so is $u_V$. 

Conversely, assume that there is a class ${u}_V \in \left(\pi_d i_e^*( R \smsh S^V)\right)^G $ as in \Cref{defn:uVnonequivariant} which is a permanent cycle. This means that it represents a class 
\[
\widetilde{u}_{V}\in \pi_{d} (R\smsh S^{V})^{hG}\cong [S^{d},R\smsh S^{V}]^{G},
\]
since \(R\) is cofree and \(S^{V}\) is a finite \(G\)-CW complex. In other words, \(\widetilde{u}_{V}\) is an equivariant map
\[
S^{d}\to R\smsh S^{V},
\]
and by base change, we get an equivariant map
\[
 \widetilde{u}_{V}^R  \colon R\wedge S^{d}\to R\smsh S^{V}.
\]
By assumption, this map induces the same map as 
 \({u}_{V}^{i_e^*R}\) on $\pi_*i_e^*(-)$, and so is an underlying equivalence. Since \(R\) is cofree, the map \(\widetilde{u}_{V}^R\) is then an equivariant equivalence, and thus \(\widetilde{u}_{V}\) is an orientation.
\end{proof}

We close this section by connecting the orientability for representations of subgroups to the orientability of the induced representations. This uses in an essential way the full equivariant commutative ring structure on spectra like \(E\), rather than the results so far which have just used an \(E_{\infty}\)-ring structure.

\begin{prop}[{\cite[Theorem 10]{BHGSym}}]
If \(R\) is an equivariant commutative ring spectrum, then the symmetric monoidal category of \(R\)-modules sits in a symmetric monoidal coefficient system: for any map of orbits \(G/H\to G/K\), we have a natural map
\[
\res_{H}^{K}\colon \Mod_{K}(i_{K}^{\ast}R)\to \Mod_{H}(i_{H}^{\ast} R).
\]

There are homotopically meaningful norm maps
\[
{}^{R}N_{H}^{G}\colon \Mod_H(i_{H}^{\ast}R) \to \Mod_G(R)
\]
given by
\[
{}^{R}N_{H}^{G}(M)= R\smsh_{N_{H}^{G}i_{H}^{\ast}R} N_{H}^{G}(M),
\]
where \(N_{H}^{G}(M)\) is the ordinary norm from \(H\)-spectra to \(G\)-spectra, and where \(R\) is an \(N_{H}^{G}i_{H}^{\ast}R\)-module via the counit of the norm forget adjunction, and these satisfy the double coset formula, up to isomorphism.
\end{prop}
Here, ``homotopically meaningful'' means that there is a model structure for which these maps are left Quillen functors. This is essentially the hard part of \cite{BHGSym}.

\begin{cor}
The coefficient system
\[
G/H\mapsto \Pic_H(i_{H}^{\ast}R)
\]
extends to a Mackey functor, where the transfer maps in \(\Pic\) are given by the norm maps.
\end{cor}
\begin{proof}
The functor \(\Pic\) is natural for symmetric monoidal functors, and therefore the symmetric monoidal functors \(i_{H}^{\ast}\) and \({}^{R}N_{H}^{G}\) induce maps on \(\Pic\). We need only check that the Mackey double-coset formula is satisfied on \(\Pic\). However, this is exactly the condition that the symmetric monoidal coefficient system 
\[
G/H\mapsto \Mod_{H}(i_{H}^{\ast}R)
\]
with its norm maps forms a symmetric monoidal Mackey functor \cite[Theorem 5.10]{BHGSym}.
\end{proof}

\begin{prop}
The homomorphisms \(J_{R}^{H}\) as \(H\) varies over the subgroups of $G$ give a map of Mackey functors
\[
\underline{RO}\to\underline{\Pic}(R).
\] 
\end{prop}
\begin{proof}
The maps \(J_{R}^{H}\) visibly commute with the restriction maps. We need only check that they also commute with the transfer maps. The transfers in \(\underline{RO}\) are given by induction, while in \(\underline{\Pic}(R)\), we have the norms. The result then follows from the computation
\[
{}^{R}N_{H}^{G}(i_{H}^{\ast}R\smsh S^{W})=R\smsh_{N_{H}^{G}i_{H}^{\ast}R}\big(N_{H}^{G}(i_{H}^{\ast}R\smsh S^{W})\big)\cong
R\smsh S^{\Ind_{H}^{G}W},
\]
where the last isomorphism is the usual cancellation formula and the computation of the norm of a representation sphere.
\end{proof}

\begin{cor}\label{cor:induction}
Let \(R\) be an equivariant commutative ring \(G\)-spectrum. Suppose that $H$ is a subgroup of $G$ and $W$ is a virtual $H$-representation of dimension $0$. Let $V = \Ind_H^G(W)$. If $W$ is $i_{H}^{\ast}R$-orientable, then $V$ is $R$-orientable.
\end{cor}


\section{The $C_4$-spectrum $E$ and the non-cyclicity of its Picard group}\label{sec:Etheorybig}

The goal of this section is to compute the image of the map $\JEG$ for the groups $G=C_2$ and $C_4$ acting on the Morava $E$-theory spectrum.

\subsection{A convenient choice of $E$-theory}\label{sec:mikescomment}
We begin by fixing a choice of $E$-theory that will be convenient for computations. The methods of this section are due to Hill--Hopkins--Ravenel \cite{HHR, HHRC4} and were the motivation for the work of Hahn--Shi \cite{hahnshi}. In fact, an analogous construction of genuine equivariant Morava $E$-theory at $p=2$ can be done at all heights and will soon appear in work of Shi. It is also used at higher heights in computations of Hill--Shi--Wang--Xu \cite{c4e4}.

We write $MU_{\R}$ to denote the $2$-localization of the real bordism spectrum. From \cite[\S 5.4.2]{HHR} there are classes $\bar{r}_i^{C_2} \in \pi_{i\rho_2}^{C_2}MU_{\R}$ whose underlying homotopy classes, which we denote by $r_i^{C_2} \in \pi_{2i}i_e^*MU_{\R}$, give a set of polynomial generators 
 \[ \pi_{*}i_e^*MU_\R \cong \Z_{(2)}[r_1,r_2,\ldots].\]
Let $N_2^4 :=N_{C_2}^{C_4}$ be the norm functor from the category of  $C_2$-spectra to the category of $C_4$-spectra. The spectrum $MU^{(\!(C_4)\!)}$ is defined to be the $C_4$-spectrum
\[MU^{(\!(C_4)\!)}:=N_2^4MU_{\R} ={MU_{\R}  \overset{\curvearrowrightminus}{\wedge} MU_\R}  \]
with action of $\gamma$ given by a cyclic permutation of the factors twisted by the conjugation action on $MU_\R$ analogous to the action of the form $\gamma(x,y) =(\bar{y},x) $ in algebra. 
Again, in \cite[Cor.5.49]{HHR}, it is shown that there are classes
\begin{align}\label{eq:r1is}
\bar{r}_{i,0}&=\bar{r}_i^{C_4},   & \bar{r}_{i,1}&=\gamma\bar{r}_i^{C_4}, 
\end{align}
in $\pi_{i\rho_2}^{C_2}i^*_{C_2}MU^{(\!(C_4)\!)}$ with the property that for $r_{i, \epsilon}\in \pi_{2i}i_e^*MU^{(\!(C_4)\!)}$ corresponding to $\bar{r}_{i,\epsilon}$ we have 
\[\pi_*i_e^*MU^{(\!(C_4)\!)}  \cong \Z_{(2)}[r_{1,0}, r_{1,1}, r_{2,0}, r_{2,1}, \ldots].  \]
Using the commutativity of $C_4$, 
\[\pi_{i\rho_2}^{C_2}i^*_{C_2}MU^{(\!(C_4)\!)}=[S^{i\rho_2}, i^*_{C_2}MU^{(\!(C_4)\!)}]^{C_2}\]
inherits an action of $C_4$ from the action on $ i^*_{C_2}MU^{(\!(C_4)\!)}$: In terms of morphisms, this is the action by post-composition rather than the conjugation action used to define equivariance. This action has the following effect on the classes $\bar{r}_{i,0}$:
\[
\gamma (\bar{r}_{i,0}) =\bar{ r}_{i,1}, \qquad  \gamma(\bar{r}_{i,1})=(-1)^i \bar{r}_{i,0}.
\]
Similarly, the action of $\gamma$ on $\pi_*i_e^*MU^{(\!(C_4)\!)}$ is given by 
\[
\gamma (r_{i,0}) = r_{i,1}, \qquad  \gamma(r_{i,1})=(-1)^i r_{i,0}.
\]
The spectrum $MU^{(\!(C_4)\!)}$ has two real orientations coming from the right and left unit of $MU_{\R} \wedge MU_{\R}$, both of which are equivariant maps. We fix the orientation coming from the left unit. We also denote by $\bar{r}_i^{C_2}$ the image of $\bar{r}_i^{C_2}\in \pi_{i\rho_2}^{C_2}MU_{\R}$ under the left unit.

 In \cite{HHRC4}, the authors define a $C_4$-spectrum $k_{[2]}$, obtained from $MU^{(\!(C_4)\!)}$ by \emph{equivariantly killing} the generators $\bar{r}_{i,\epsilon}$ for $i\geq 2$ and $\epsilon=0,1$. That is
\[ k_{[2]}:=MU^{(\!( C_4 )\!)} \smsh_A S^0, \]
where $A$ is the associative ring spectrum
\[A=S^0[C_4 \cdot \bar{r}_2, C_4 \cdot \bar{r}_3, C_4 \cdot \bar{r}_4, \ldots];\]
see also \cite[Section 2.4]{HHR}.

In particular,
\[\pi_*i_e^*k_{[2]} \cong \Z_{(2)}[r_{1,0}, r_{1,1}].\]
See, for example, \cite[Theorem 12.2]{HHRC4}. It inherits a real orientation from $\MUC$ and, in \cite[\S 7]{HHRC4}, the authors compute the image of $\bar{r}_i^{C_2}$ for $i=1,2,3$ in $\pi_{\star}^{C_2}i_{C_2}^*k_{[2]}$:
\begin{align}\label{eq:defining-r-s}
\bar{r}_1^{C_2}&= \bar{r}_{1,0}+\bar{r}_{1,1}, \nonumber \\
  \bar{r}_2^{C_2}&= 3\bar{r}_{1,0}\bar{r}_{1,1}+\bar{r}_{1,1}^2, \\
   \bar{r}_3^{C_2}&=\bar{r}_{1,1}(5\bar{r}_{1,0}^2 + 5\bar{r}_{1,0}\bar{r}_{1,1}+\bar{r}_{1,1}^2). \nonumber 
\end{align}

The spectrum $K_{[2]}$ is obtained from $k_{[2]}$ by inverting an element $D \in \pi_{4\rho_4}^{C_4}k_{[2]}$ whose restriction in
$\pi_{8\rho_2}^{C_2} i_{C_2}^*k_{[2]}$ is 
$\bar{r}_{1,0}\bar{r}_{1,1}\bar{r}_{3}^{C_2}\gamma\bar{r}_{3}^{C_2}$. 
See \cite[(7.4)]{HHRC4}. Since inverting $D$ inverts its restrictions, it also inverts the classes $\bar{r}_{1,\epsilon}$ in $\pi_{\rho_2}^{C_2}i_{C_2}^*k_{[2]}$. 
So, the homotopy classes 
\begin{align}\label{eq:mus}
 \mu_0 &= \frac{\bar{r}_{1,0}-\bar{r}_{1,1}}{\bar{r}_{1,0}} = \frac{\bar{r}_1^{C_2}-2\bar{r}_{1,1}} {\bar{r}_{1,0}}, &
\mu_1   = \frac{\bar{r}_{1,0}+\bar{r}_{1,1} }{\bar{r}_{1,1}}  =\gamma(\mu_0) =\frac{\bar{r}_1^{C_2}}{\bar{r}_{1,1}}  . \end{align}
are defined in $ \pi_{0}^{C_2}i_{C_2}^*K_{[2]}$. We let
\begin{equation}\label{eq:muafterrev}\mu = \tr_{2}^{4}(\mu_0) =  \tr_{2}^{4}(\mu_1) \in  \pi_{0}^{C_4}K_{[2]} ,\end{equation}
noting that $\res_{2}^{4}(\mu) = \mu_0+\mu_1$. We also denote by $\mu_0$ and $\mu_1$ their restrictions to $\pi_0i_e^*K_{[2]}$ and, similarly, by $\mu$ its restrictions to  $ \pi_{0}^{C_2}i_{C_2}^*K_{[2]}$ and $\pi_0i_e^*K_{[2]}$. Note that in $\pi_*i_e^*K_{[2]}$ we have
\[r_{1,1} =r_{1,0}(1-\mu_0) \]
and that there is an isomorphism 
\begin{align*} 
 \left(\pi_*i_e^*K_{[2]}\right)^{\wedge}_{(2,r_1^{C_2})} &\cong  \left(\Z_{(2)}[r_{1,0}, r_{1,1}][(\res_1^4D)^{\pm 1}]\right)_{(2, \mu_0)}^{\wedge} \\
& \cong \left( \Z_{(2)}[r_{1,0}^{\pm 1}, \mu_0][(1-\mu_0)^{-1}, (11-7\mu_0+\mu_0^2)^{-1},(1 - 5\mu_0 + 5\mu_0^2)^{-1} ] \right)_{(2, \mu_0)}^{\wedge} \\
&  \cong \Z_{2}[\![\mu_0]\!][r_{1,0}^{\pm 1}] ,
\end{align*}
where the second isomorphism is obtained by substituting the above expression for $r_{1,1}$ in $\res_1^4D$.

The spectrum $K_{[2]}$ also inherits a real orientation and it follows from \eqref{eq:defining-r-s} that the formal group law $F_{[2]}$ of $K_{[2]}$ reduces to a formal group law $\Gamma_{[2]}$ of height $2$ defined over 
\[\pi_*i_e^*K_{[2]}/(2, {r}_1^{C_2}) \cong \F_2[{r}_{1,0}^{\pm 1}] , \]
whose $[2]$-series satisfies
\[[2]_{\Gamma_{[2]}}(x) ={r}_{1,0}^3x^4 + \ldots \] 
In fact, the $[2]$-series of $F_{[2]}$ satisfies
\begin{align}\label{eq:coefffgl}
[2]_{F_{[2]}}(x) &\equiv {r}_1^{C_2}x^2 
\equiv \mu_0 r_{1,0} x^2\mod (2, x^3), \\ \label{eq:coefffgl2}
[2]_{F_{[2]}}(x) &\equiv    {r}_3^{C_2} x^4 
\equiv 
r_{1,0}^3x^4 \mod (2, \mu_0, x^5).
   \end{align}

Let $\bar{F}_{[2]}(x,y) = r_{1,0}F_{[2]}(r_{1,0}^{-1}x, r_{1,0}^{-1}y)$ over $\Z_{2}[\![ \mu_0 ]\!]$ and $\bar{\Gamma}_{[2]} =r_{1,0}\Gamma_{[2]}(r_{1,0}^{-1}x, r_{1,0}^{-1}y)$ over $\F_2$. 
Let $\fk$ be the algebraic closure of $\mathbb{F}_2$
and
      \begin{equation}\label{eq:Wdefn}
   \W :=W(\fk)
   \end{equation} 
   be the ring of Witt vectors over $\fk$, so that $\fk=\W/2$. We consider $\bar{\Gamma}_{[2]} $ as a formal group law defined over $\fk$ and $\bar{F}_{[2]} $ as a formal group law defined over $ \W[\![\mu_0]\!]$. Then $(\W[\![\mu_0]\!], \bar{F}_{[2]})$ is a universal deformation of $(\fk, \bar{\Gamma}_{[2]})$. 
   This follows from \eqref{eq:coefffgl}. Indeed, any formal group law $F$ over a power series ring $\W[\![t]\!]$ such that $F$ reduces to $\bar{\Gamma}_2$ modulo $(2,t)$ and such that 
   \begin{align*}
   [p]_F(x)  &\equiv t x^2 \mod  (2, x^3) 
      \end{align*}
      is a universal deformation of $\bar{\Gamma}_2$.

  We let $E_{[2]}$ be the associated $E$-theory spectrum with distinguished unit $r_{1,0} \in \pi_2E_{[2]}$ and let $\varphi^u  \colon MU \to E_{[2]}$ be a complex orientation chosen so that $\pi_*\varphi^u $ classifies $F_{[2]}$. By Hahn--Shi \cite{hahnshi}, $E_{[2]}$ is Real oriented, and 
   $\varphi^u $ has a unique equivariant refinement $\varphi \colon MU_{\R} \to E_{[2]}$. That is, $\varphi$ is a $C_2$-equivariant map and $\pi_*i_e^*\varphi = \pi_*\varphi^u$.
   
Let  $C_4 \subseteq \Aut_k(\bar{\Gamma}_{[2]})$ be the subgroup induced by the action of $\gamma$ on the formal group law of $MU^{(\!(C_4)\!)}$. This is a cyclic group of order four since $\gamma^2$ is the non-trivial automorphism $[-1]_{\bar{\Gamma}_{[2]}}(x)$. By \cite[Theorem 1.3]{hahnshi}, norming up gives a $C_4$-equivariant map 
\[MU^{(\!( C_4 )\!)}  \to  E_{[2]}.\] 
We will call this the Hahn-Shi orientation.

By construction, we have a diagram of $C_4$-modules
\[ \xymatrix{\pi_*i^*_e MU^{(\!( C_4 )\!)} \ar[d] \ar[r] &\pi_*i_e^*E_{[2]} \\
\pi_*i_e^*K_{[2]} \ar[ur] & }\]
and hence a map from the $E_2$-term of the homotopy fixed point spectral sequence (HFPSS) of $K_{[2]} $ to that of $E_{[2]} $.  This gives a direct way to discuss classes in $\text{HFPSS}(E_{[2]})$ in terms of classes in $\text{HFPSS}(K_{[2]})$.

From the maps $MU^{(\!( C_4 )\!)}  \to  E_{[2]}$ and $MU^{(\!( C_4 )\!)} \to  K_{[2]}$, 
we get a zig-zag of spectral sequences, where the dotted arrow is only known to be a map of $E_2$-pages, not of spectral sequences:
\begin{equation}\label{eq:bigdiagram}\xymatrix{ & \text{SliceSS}(MU^{(\!( C_4 )\!)}) \ar[dl] \ar[d] \ar[r] &\text{SliceSS}(E_{[2]}) \ar[r] &\text{HFPSS}(E_{[2]}) \\
  \text{SliceSS}(k_{[2]}) \ar[r] \ar[d] & \text{SliceSS}(K_{[2]}) \ar[d]  &  &  \\
  \text{HFPSS}(k_{[2]})  \ar[r] & \text{HFPSS}(K_{[2]}) \ar@{..>}[rruu]_-{\text{ \ \  \ \ \ \  map of $E_2$-pages}}& &  }\end{equation}
The key differentials in $\text{SliceSS}(K_{[2]})$ (and $\text{SliceSS}(k_{[2]})$) we use below in our proofs all lift to differentials in $ \text{SliceSS}(MU^{(\!( C_4 )\!)}) $, which we can then push to differentials in $\text{SliceSS}(E_{[2]}) $ and $\text{HFPSS}(E_{[2]}) $. This kind of lifting of differentials is also the strategy employed in \cite{c4e4}. Further, it follows from \cite[Theorem 9.4]{ullman} that the slice spectral sequence and homotopy fixed point spectral sequence for $E_{[2]}$ are isomorphic in, for example, the range $s +5 \leq t-s$, and similarly for the slice and homotopy fixed point spectral sequences of $K_{[2]}$.

With these observations in hand, we can use the computations of \cite{HHRC4} for the slice spectral sequences of $k_{[2]}$ and $K_{[2]}$ to compute the homotopy fixed point spectral sequence of $E_{[2]}$.

\begin{rem}\label{rem:periodicityD14}
In \Cref{sec:computations}, we will show that the $C_4$ equivariant spectrum $E_{[2]}$ is $32$ periodic, with periodicity generator an element of $\pi_{32}^{C_4} E_{[2]}$ detected by a permanent cycle we will denote by
$\Delta_1^4$. The restriction of the periodicity generator in $\pi_{32} i_e^* E_{[2]}$ is detected by $(r_{1,0}^2r_{1,1}^2)^4$. See \Cref{tab:elementsE}. One can obtain a better range from Ullman's results, but in practice, the periodicity of the homotopy fixed point spectral sequence for $E_{[2]}$ implies that one only needs to look in a range where $t-s$ is large enough with respect to $s$ to make the comparison. 
\end{rem}

\begin{notation}
We use the convention 
\[E := E_{[2]} = E(\fk, \bar{\Gamma}_{[2]})\]
and write
\[E_* = \pi_* i_e^*E.\]
We implicitly replace $E$ with $E^h = F({EC_4}_+, E)$ so that everywhere, we are working with a cofree spectrum.
\end{notation}

\subsection{The homomorphisms $\JEC{4}$ and $\JEC{2}$}\label{sec:JE}

The structure of $RO(C_2)$ and $RO(C_4)$ is reviewed in \Cref{sec:realreps}. The ring $RO(C_4)$ is of rank $3$ as a $\Z$-module, generated by the trivial representation $1$, the sign representation $\sigma$, and the $2$-dimensional irreducible representation $\lambda$. The representation $\lambda$ is (classically) orientable as it is modeled by a rotation by $\pi/2$ on $\R^2$. The representation $\sigma$ is not an orientable representation of $C_4$, but $2\sigma$ is orientable. Similarly, $RO(C_2)$ is of rank $2$ over $\Z$ generated by $1$ and $\sigma_2$. Again, the sign representation $\sigma_2$ is not orientable, but $2 \sigma_2$ is.

At this point, we relate our definition of pseudo $E$-orientation to the elements $u_V$ that appear in the slice spectral sequence computations of \cite{HHRC4}. These classes are, for example, defined in \cite[Definition 3.12]{HHR}. If $V\in RO(G)$ of dimension $d$ is classically orientable when restricted to a subgroup $G' \subseteq G$, then one can fix a choice of generator
\[u_V \in \pi_{d-V}^{G'} H\underline{\Z} \cong H_d^{G'}(S^V, \underline{\Z}) \xrightarrow{\cong} H_d(S^V, \Z)\cong \Z.\]
These classes satisfy the relation $u_V u_{W} = u_{V+W}$.

When $G=C_4$, this gives classes
\begin{align*}
u_{\lambda}&\in \spi_{2-\lambda} H\underline{\Z}(C_4/C_4), & u_{2\sigma}& \in\spi_{2-2\sigma} H\underline{\Z}(C_4/C_4), & u_{\sigma}& \in\spi_{1-\sigma} H\underline{\Z}(C_4/C_2)
\end{align*}
together with their restrictions and products.
Similarly, when $G=C_2$, we get classes
\begin{align*}
u_{2\sigma_2}&\in \spi_{2-2\sigma_2} H\underline{\Z}(C_2/C_2), & u_{\sigma_2} &\in\spi_{1-\sigma_2} H\underline{\Z}(C_2/e).
\end{align*}

Let $V$ be classically orientable for the subgroup $G' \subseteq C_4$.
Since the slice $P_0^0MU^{(\!(C_4)\!)} $ is  $H\underline{\Z}$, the map (as in \Cref{rem:slicess}) from the slice spectral sequence of $MU^{(\!(C_4)\!)}$ to the homotopy fixed point spectral sequence of $E$ gives a commutative diagram
\[\xymatrix{ \pi_{0}i_e^* H\underline{\Z}  \ar[d]_-{\cong}^{\iota^*} & \pi_{0}^{G'} H\underline{\Z} \ar[r]^-{\subseteq} \ar[l]^-{=} &  {H}^0(G', \pi_{0}E)  \ar[d]^-{\cong}_{\iota^*}  \\
\pi_{d-V} i_e^* H\underline{\Z}  & \pi_{d-V}^{G'} H\underline{\Z} \ar[r]  \ar[l]^-{=}  &  {H}^0(G', \pi_{d-V}E)
  } \]
  where the vertical isomorphisms are induced by precomposition with a chosen underlying equivalence $\iota \colon S^{d-V} \to S^0$. The top inclusion is the natural ring map. It follows that
the class $u_V \in \pi_{d-V}^{G'} H\underline{\Z}$ maps to a pseudo-orientation in ${H}^0(G', \pi_{d-V}E)$ which we denote by the same name.

\begin{rem}\label{rem:uvsforE}
By the discussion above, there are pseudo-orientations
\begin{align*}
u_{\lambda}&\in  \underline{E}^{0,2-\lambda }_2(C_4/C_4), & u_{2\sigma}& \in \underline{E}^{0,2-2\sigma}_2 (C_4/C_4), & u_{\sigma}& \in \underline{E}^{0,1-\sigma}_2 (C_4/C_2).
\end{align*}
As was noted in \Cref{rem:units}, the classes $u_V$ are all invertible in $\underline{E}^{0,\star} =\underline{H}^0(G, \pi_{\star}E)$.
\end{rem}

Let  $\bar{r}_{1,0} \in \pi_{\rho_2}^{C_2}i_{C_2}^*E$ be the image of the same-named class of $\pi_{\rho_2}^{C_2}i_{C_2}^*\MUC$ induced by the Hahn--Shi Real orientation $MU_\R \to E$. By our construction of \(E\), this class is a unit and hence induces an equivalence 
\[
i_{C_2}^{\ast} E \smsh S^{\rho_2}  \simeq i_{C_2}^\ast E.
\] 
In particular, this means that $\JEC2(\rho_2)=0$. By \Cref{cor:induction}, we have $\JEC4(\rho_4)=0$.
\begin{prop}
In $\Pic_{C_4}(E)$, $\JEC4(\rho_4)=0$, where $\rho_4= 1+\sigma+\lambda$ is the regular representation of $C_4$.
\end{prop}

Before turning to the more technical computations of the next sections, we state two results which are proved in \Cref{sec:computations}. We use them here to compute the image of $\JEC{4}$ and $\JEC{2}$.
The first result we state follows from the computations of \cite{HHRC4} and is proved in \Cref{prop:pc} below:
\begin{prop}\label{prop:orientableus}
The classes $u_{\lambda}^8, u_{2\sigma}^2$ and $u_{2\sigma}u_{\lambda}^4$ are permanent cycles in the 
 homotopy fixed point spectral sequence for $E_2^{hC_4}$. Therefore, the representations $8\lambda$, $4\sigma$ and $2\sigma+4\lambda$ are $E$-orientable for $C_4$.
In particular, the images under $\JEC4$ of
 \begin{itemize}
 \item $16-8\lambda$,
 \item $4-4\sigma$, and 
 \item $10-2\sigma-4\lambda$ 
 \end{itemize}
are zero in $\Pic_{C_4}(E)$. 

The class $u_{8\sigma_2}$ is a permanent cycle  in the homotopy fixed point spectral sequence for $E_2^{hC_2}$. So, the image under $\JEC2$ of
 \begin{itemize}
 \item $8-8\sigma_2$
 \end{itemize}
is zero in $\Pic_{C_2}(E)$. 
   \end{prop}

The second result we state is proved in \Cref{cor:notintegershit} and \Cref{table:finalfortrivialsigma}:
\begin{prop}\label{prop:oneminussigma}
There is no integer $d$ such that $E \smsh S^{\sigma-1} \simeq E \smsh S^d$ as $C_4$ equivariant $E$-modules. In particular, $\JEC4(\sigma-1)$ is not in the cyclic subgroup generated by $\Sigma E^{hC_4}$. 
\end{prop}

\begin{prop}\label{prop:image}
The image of the map $\JEC4 \colon RO(C_4) \to \Pic_{C_4}(E)$ is isomorphic to 
\[\ZZ/{32} \{ 1 \} \oplus \ZZ/2\{7+\sigma\}.\]
The image of the map $\JEC2 \colon RO(C_2) \to \Pic_{C_2}(E)$ is isomorphic to 
\[\ZZ/{16}\{1\}.\]
\end{prop}
\begin{proof}
As an abelian group, $RO(C_4)$ is isomorphic to $\Z\{1, \sigma, \lambda\}$. 
The map $\JEC4$ factors through the quotient
\begin{equation}\label{eq:start}
\Z\{1, \sigma, \lambda\}/(16-8\lambda, 4-4\sigma, 10-2\sigma-4\lambda, 1+\sigma+\lambda).\end{equation}
Simplifying the relations, we have that \eqref{eq:start} is isomorphic to
\begin{align*}
 \Z\{1, \sigma\}/(24+8\sigma, 4-4\sigma, 14+2\sigma ) & \cong  
   \Z/32\{1\} \oplus \Z/2\{7+\sigma\}.
\end{align*}
Since the periodicity of $E^{hC_4}$ is $32$, the image of $\JEC4$ contains the cyclic subgroup $\Z/32\{1\}$. Further, since 
$\JEC4(\sigma -1) = \JEC4(-8+(7+\sigma)) $ is not in this cyclic subgroup, then neither is $7+\sigma$. Therefore, $\JEC4$ does not factor through a smaller quotient.

For $C_2$, a similar computation shows that $\JEC4$ factors through $\Z/16$. Since the periodicity of $E^{hC_2}$ is $16$, there can be no further relations.
\end{proof}

\begin{rem}
\Cref{prop:image} provides a lower bound for the order of $\Pic_{C_4}(E)$. That is, this group has order at least $64$.
\end{rem}

\begin{rem}\label{rem:allcompsintwo}
\Cref{prop:orientableus} and \Cref{prop:image} together immediately imply that for any virtual representation $V \in RO(C_4)$, 
the homotopy fixed point spectral sequence computing $\spi_{*-V}E \cong \spi_* S^{V}\smsh E$ is a shift of either that computing $\spi_*E$ or that computing $\spi_{*+1-\sigma}E$. We explain this in two examples. 
Since 
\[2\sigma-2  \equiv 16 \ \ \text{in} \ \  \ZZ/{32} \{ 1 \} \oplus \ZZ/2\{7+\sigma\},\] 
there is an isomorphism of Mackey functor homotopy fixed point spectral sequences 
\[ \underline{E}_*^{*,*+2-2\sigma} \cong \underline{E}_*^{*,*+16}.\] 
Similarly, since 
\[ 15+\sigma \equiv 1-\sigma \ \ \text{in} \ \  \ZZ/{32} \{ 1 \} \oplus \ZZ/2\{7+\sigma\},\] 
there is an isomorphism of spectral sequences 
\[ \underline{E}_*^{*,*+15+\sigma} \cong \underline{E}_*^{*,*+1-\sigma}  .\] 
\end{rem}


\section{The $C_4$ Homotopy Fixed Point Spectral Sequence for $E$ Theory}\label{sec:computations}

The goal of this section is to compute the homotopy fixed point spectral sequence
\begin{equation}\label{eq:ssV}
\underline{E}_2^{s,\star+t} = \underline{H}^s(C_4, \pi_{\star+t}E) \Longrightarrow {\spi}_{\star+t-s}E^{h}
\end{equation}
with differentials $d_r \colon \underline{E}_r^{s,\star+t} \to \underline{E}_r^{s+r, \star+t+r-1}$.
As is noted in \Cref{rem:allcompsintwo}, the computation for any $ \star \in RO(C_4)$ is determined by the computation for $\star=0$ or for $\star = 1-\sigma$. So, throughout this section, we detail the computation of $\underline{E}_{*}^{*,*}$ and of $\underline{E}_{*}^{*,1-\sigma+*}$.

We first describe the structure of $\underline{H}^*(C_4, E_*)$ as a Mackey functor:
\begin{equation}\label{eq:mackey}
\xymatrix{H^*(C_4, E_*) \ar@/_/[d]_-{\res_2^{4}} \\ 
H^*(C_2, E_*) \ar@/_/[u]_-{\tr_2^4} \ar@/_/[d]_-{\res_1^{2}}  \\
H^*(\{e\}, E_*) \ar@/_/[u]_-{\tr_1^2} }
\end{equation}
This then determines the $\underline{E}_2$-term of \eqref{eq:ssV} when $\star=0$. The computation of the $\underline{E}_2$-term of \eqref{eq:ssV} when $\star=1-\sigma$, will easily follow from this computation. This is explained in \Cref{sec:Emodulesign}.

Next, we discuss
the differentials and extensions. They are imported from the slice spectral sequence for $K_{[2]}$ to the homotopy fixed point spectral sequence for $E^{hC_4}$ using the diagram \eqref{eq:bigdiagram}. Although all differentials are a direct consequence of those in the slice spectral sequence, there is a significant difference between this computation and that of \cite{HHRC4} in the range near the zero line. The $d_3$-differentials need particular attention and we give more details in this part of the computation. From the $\underline{E}_{4}$-term onwards, the computation in the range $0\leq s\leq 2$ is also different, but for $s\geq 3$, the computation is almost identical to that of the slice spectral sequence. 

Finally, we describe the $\underline{E}_{\infty}$-terms of $\spi_*E^h$ and $\spi_{(1-\sigma)+*}E^h$.

\subsection{Some important $C_4$-modules}\label{notn:A}\label{notn:reps}
Before giving a more detailed description of the structure of $E_*$ as a $C_4$-module, we need to set up some notation for a selection of $C_4$-modules that will be making multiple appearances in the computations below.

Denoting by $\gamma$ the generator of $C_4$, and by $e$ the trivial group element, we have the modules
\begin{align*}
\Z&= \Z[C_4/C_4] = \Z[C_4]/(\gamma-e)\, &&\text{corresponding to } 1 \text{ in } RO(C_4)\\
\Z_{-}&= \Z[C_4]/(\gamma+e)\, &&\text{corresponding to } \sigma \text{ in } RO(C_4)\\
\Z[{C_4/C_2}] &= \Z[C_4]/(\gamma^2-e)\, &&\text{corresponding to } 1+\sigma \text{ in } RO(C_4) \\
\Z[{C_4/C_2}]_- &= \Z[C_4]/(\gamma^2+e) \, &&\text{corresponding to } \lambda \text{ in } RO(C_4)
\end{align*}
where the correspondence with elements of $RO(C_4)$ occurs after base change to $\R$.
Moreover, we will write 
\begin{align*}
\Z[{C_4/C_2}_+] &= \Z[C_4/C_2]\oplus \Z &    \Z[{C_4/C_2}_-] &=\Z[C_4/C_2]\oplus \Z_{-} \ .
\end{align*}
Note that 
\begin{align*} 
\Z_{-} \otimes \Z_{-} &\cong \Z, &  \Z_{-} \otimes \Z &\cong \Z_{-}, \\
\Z_{-} \otimes \Z[{C_4/C_2}]  &\cong  \Z[{C_4/C_2}] ,  & \Z_{-} \otimes \Z[{C_4/C_2}]_-  &\cong  \Z[{C_4/C_2}]_- .
\end{align*}

Since the coefficients of $E$ are modules over the Witt vectors $\W = W(\fk)$, we will have to base change all of the above modules to $\W$; that amounts to just writing $\W$ instead of $\Z$.
 
For $\mu$ as in \eqref{eq:muafterrev}, we will let 
\begin{align*}
\wmu &= \W[\![ \mu ]\!] [C_4/C_2] =  \W[\![\mu ]\!] \otimes_{\W}\W[C_4/C_2] 
\end{align*}
and
\begin{align*}
\wmu_- &= \W[\![ \mu ]\!] [C_4/C_2]_- =  \W[\![\mu ]\!] \otimes_{\W} \W[C_4/C_2]_{-} \ .
\end{align*}
Let 
\begin{align*}
\Delta &= {e}+{\gamma} \in \W[C_4], & 
 \bar{\Delta}& ={e} -{\gamma}\in \W[C_4] \ .
\end{align*}
 We denote by the same name the images of $\Delta$ and $\bar{\Delta}$ in $ \W[{C_4/C_2}]$, in $\W[{C_4/C_2}]_-$, and in $ \W[{C_4/C_2}_{\pm}] $ via the inclusion $\W[C_4/C_2] \subseteq \W[{C_4/C_2}_{\pm}] $.

Let $\mu$ be as in \eqref{eq:muafterrev}.
For
\[\W[\![\mu ]\!][{C_4/C_2}_{\pm}]  = \W[\![\mu ]\!] \otimes_{\W} \W[{C_4/C_2}_{\pm}] ,\] 
and $\ast$ corresponding to a generator of the $C_4$-submodule $\W \subseteq \W[{C_4/C_2}_{+}]$, we define
\begin{align*}
\wmuP&= \frac{\W[\![ \mu ]\!][{C_4/C_2}_+]}{\mu \cdot \ast =\Delta}.
\end{align*}
Similarly, for $\ast$ a generator of $\W_{-}\subseteq \W[{C_4/C_2}_{-}]$, we define
\begin{align*}
 \wmuM&=  \frac{\W[\![\mu] \!][C_4/{C_2}_-]}{\mu\cdot \ast = \bar{\Delta}+2 \cdot \ast}.
\end{align*}

There are exact sequences
\begin{equation}\label{eq:exactA}
\xymatrix@R=0.7pc{
0 \ar[r] & A \ar[r] & A(+) \ar[r] & \W \ar[r] & 0\\
0 \ar[r] & A \ar[r] & A(-) \ar[r] & \W_{-} \ar[r] & 0.
}
\end{equation}
Further, note that
\begin{align*}
\wmuP/2 &\cong \wmuM/2 \cong \frac{\fk[\![ \mu ]\!][{C_4/C_2}_+]}{\mu \cdot \ast =\Delta}
\end{align*}
and that there is an exact sequence
\[ \xymatrix@R=0.7pc{
0 \ar[r] & A/2 \ar[r] & A(+)/2 \ar[r] & \fk \ar[r] & 0.}\]

\subsection{The Mackey functor cohomology of $C_4$ with coefficients in $E_*$}\label{sec:act}\label{rem:act}\label{sec:cohomology}
Recall that there is isomorphism $E_*  \cong \W[\![ \mu_0 ]\!][r_{1,0}^{\pm 1}]$
of $\W[C_4]$-modules, with $|r_{1,0}|=2$ and the action of $\gamma$ given by
\begin{align*}
\gamma(r_{1,0})&= r_{1,1} =r_{1,0}(1-\mu_0),  &  \gamma(\mu_0)  = \mu_1 = \frac{2-\mu_0}{1-\mu_0}.
\end{align*}
In \Cref{tab:elementsE}, we have named some of the elements in $E_*$, $E_*^{C_2}$ and $E_*^{C_4}$. We will use the notation and the information contained in \Cref{tab:elementsE} throughout the rest of the paper.

{
\begin{table}[H]
\begin{center}
\begin{tabular}{|C|C|C|C|}
\hline
 \text{Elements} & \text{Degree} & \text{Action of $\gamma$}  \\
\hline
\hline
\makecell{r_{1,0} \\
 r_{1,1} =  r_{1,0}(1-\mu_0) } & E_2 &  \makecell{\gamma(r_{1,0}) = r_{1,1}  \\   \gamma(r_{1,1}) = -r_{1,0}}  \\
 \hline
 \makecell{\mu_0 =\frac{ r_{1,0}- r_{1,1}}{r_{1,0}},  \\  \mu_1=\frac{r_{1,0}+r_{1,1}}{r_{1,1}} =\frac{2-\mu_0}{1-\mu_0} } &  E_0^{C_2} &\makecell{ \gamma(\mu_0) = \mu_1  \\  \gamma(\mu_1) = \mu_0  }  \\
 \hline
\makecell{\Sigma_{2,0} =r_{1,0}^2,   \\   \Sigma_{2,1} =-r_{1,1}^2 } &  E_4^{C_2} &\makecell{ \gamma( \Sigma_{2,0}) = -\Sigma_{2,1}  \\  \gamma( \Sigma_{2,1}) = -\Sigma_{2,0}  }  \\
\hline
 \delta_1= r_{1,0}r_{1,1}   &  E_4^{C_2}  & \gamma(\delta_1) = -\delta_1  \\
\hline
\mu =\mu_0 + \mu_1  =\frac{2-\mu_0^2}{1-\mu_0}
  & E_0^{C_4} &  \gamma(\mu)= \mu   \\
  \hline
T_2=r_{1,0}^2+r_{1,1}^2 & E_4^{C_4} & \gamma(T_2)= T_2    \\
\hline
\makecell{ \Delta_1 = \delta_1^2   }  & E_8^{C_4}   & \gamma(\Delta_1) = \Delta_1    \\
\hline
 \makecell{T_4=\Delta_1( \mu -2) }& E_8^{C_4} & \gamma(T_4)= T_4     \\
\hline
\end{tabular}
\end{center}
\caption{Some elements of $E_*$ as a $C_4$-module and their images under the action of $\gamma$. Names are chosen to reflect the notation of \cite[Table 3]{HHRC4}.}
\label{tab:elementsE}
\label{notn:elements}
\end{table}}

\begin{rem}
There are isomorphisms
\begin{align*}
\W\{r_{1,0},r_{1,1}\} &\cong \W[C_4/C_2]_{-} , &  \W\{\mu_0, \mu_1\} &\cong \W[C_4/C_2] ,\\
 \W\{\Sigma_{2,0},  \Sigma_{2,1}\} &\cong \W[C_4/C_2]  , & \W\{\delta_1\} &\cong \W_{-} .
\end{align*} 
\end{rem}

\begin{rem}
The norm element
\[\Delta_1= r_{1,0} \gamma(r_{1,0}) \gamma^2(r_{1,0})\gamma^3(r_{1,0})\]
is a periodicity generator of $E_*$ as a $C_4$-module. However, we will see that $\Delta_1^4$ is a permanent cycle and is a periodicity generator of $E^h$ as a $C_4$-spectrum. See \Cref{rem:periodicityD14}.
\end{rem}

\begin{rem}
The relations
\begin{align*}
\mu_0^2 &= \mu  \mu_0 -\mu  +2,\\
r_{1,0}^4&=\Delta_1 (1-\mu_0)^{-2}
\end{align*} imply that 
$\W [\![\mu_0]\!] $ is isomorphic to $\W[\![\mu ]\!] \{1,\mu_0\}$ as a $\W[\![\mu ]\!]$-module and, in fact, that
\[ E_* \cong \W[\![\mu ]\!][\Delta_1^{\pm 1}]\{r_{1,0}^k, \mu_0 r_{1,0}^k : 0\leq k \leq 3\} \]
as $\W[\![\mu]\!][\Delta_1^{\pm 1}]$-modules.
\end{rem}

\begin{lem}\label{lem:reps}
As a $\W[\![\mu]\!][C_4]$-module, $E_*$ is $8$ periodic, with periodicity generator $\Delta_1 \in E_8$. There are isomorphisms
\begin{align*}
E_{0} &\cong  \W[\![\mu]\!]\{1,\mu_0,\mu_1\}/(\mu\cdot 1=\mu_0+\mu_1), \\
E_{2} &\cong \W[\![\mu]\!]\{ r_{1,0}, r_{1,1} \}, \\ 
E_{4}& \cong  \W[\![\mu]\!]\{\delta_1,\Sigma_{2,0},\Sigma_{2,1}\}/(\mu \delta_1 =\Sigma_{2,0}+\Sigma_{2,1}+2\delta_1),  \\
E_{6} &\cong \W[\![\mu]\!] \{ \delta_1 r_{1,0},  \delta_1 r_{1,1} \}  .
\end{align*}
in particular,
\begin{align*}
E_{0} &\cong   \wmuP, &
E_{2} &\cong E_6 \cong \wmum, &
E_{4} &\cong    \wmuM.
\end{align*}
\end{lem}
\begin{proof}
A straightforward computation gives the first set of isomorphisms. The second set of isomorphisms is given by the $C_4$-linear maps determined by:
\begin{align*}
E_{0} &\to\wmuP & E_{2} &\to\wmum &  E_{4} &\to\wmum & E_{6} &\to\wmuM \\
1 &\mapsto \ast  & r_{1,0} &\mapsto e  & \delta_1 &\mapsto \ast & \delta_1 r_{1,0} &\mapsto e \\
\mu_0 &\mapsto  e & & &     \Sigma_{2,0}&\mapsto e & & &\qedhere
\end{align*}
\end{proof}

\begin{rem}
\Cref{lem:reps} implies that
\begin{align*}
 E_*^{C_2} &\cong \W[\![\mu_0]\!][\Sigma_{2,0}^{\pm 1}] \\ E_*^{C_4} &\cong \W[\![\mu]\!][T_2,\Delta_1^{\pm 1}]/(T_2^2-\Delta_1((\mu-2)^2+4)).
\end{align*}
\end{rem}

From the facts that $\Delta_1 = (r_{1,0}r_{1,1})^2$ and that the ideal $(2,\mu)$ is equal to $(2,\mu_0^2)$, it follows that
\[ E_* \cong (\W[r_{1,0}, r_{1,1}] [\Delta_1^{-1}])^{\wedge}_{(2,\mu)} \]
as $\W[C_4]$-modules. An argument similar to that of Theorem 6 in \cite{GHM_V1} gives an isomorphism
\[ H^*(C_4, E_*) \cong \left(\W \otimes H^*(C_4, \Z[r_{1,0}, r_{1,1}])[\Delta_1^{-1}]\right)^{\wedge}_{(2,\mu)}.\]
Hence, in order to compute the cohomology of $C_4$ with coefficients in $E_*$, we first compute $H^*(C_4,\Z[r_{1,0}, r_{1,1}])$, noting that $\Z[r_{1,0}, r_{1,1}]$ is the symmetric algebra of the induced sign representation
\[\Z[r_{1,0}, r_{1,1}] \cong \Sym(\Z[C_4/C_2]_-).\]
Then we base change to $\W$,  invert $\Delta_1$ and complete the answer at the ideal $(2,\mu)$. The computation of the cohomology $H^*(C_4,\Z[r_{1,0}, r_{1,1}])$ is essentially the same as that performed in \cite{HHRC4} to compute the $E_2$-term of the slice spectral sequence of $k_{[2]}$. 
This approach to the computation also appears in work in progress of Henn.
 Finally, the answer is described in  \cite[Section 2.2]{BobkovaGoerss} and can be deduced from \cite{HHRC4}. Following these references, we describe the answer in \Cref{coh-c4}. We do not repeat the computation of the group cohomology, but rather focus our attention on describing the $\underline{E}_2^{*,V+*}$-terms as Mackey functors for $V=0$ and $V=1-\sigma$. 
 Since we will be deducing most of our results from \cite{HHRC4}, we have opted to choose notation that does not clash with theirs. We need the description of the $E_2$-page of the \(RO(C_2)\)-graded homotopy fixed point spectral sequence for Morava \(E\)-theory used in \cite{hahnshi}, which we restate using the notation of this paper:

\begin{prop} \label{coh-c2}
There is an isomorphism 
\[
H^{*}(C_{2},E_{\star})\cong \mathbb{W}[\![\mu_{0}]\!][\bar{r}_{1,0}^{\pm 1}, a_{\sigma}, u_{2\sigma}^{\pm 1}]/2a_{\sigma},
\]
where the $( *, \star)$-degree of the elements is given by \(|\mu_0|=(0,0)\), \(|\bar{r}_{1,0}|=(0,\rho_{2})\), \(|a_{\sigma}|=(1,-\sigma)\), and \(|u_{2\sigma}|=(0,2-2\sigma)\).
\end{prop}

\begin{notation}
For \(\bar{r}_{1,0}\) and \(\bar{r}_{1,1}\) as in \eqref{eq:r1is}, it follows from \eqref{eq:mus} that 
\[\bar{r}_{1,1}=\bar{r}_{1,0}(1-\mu_{0}).\] 
For \(a_{\sigma}\) as in \Cref{notn:aV}, let \(\eta_{i}=a_{\sigma}\bar{r}_{1,i}\), and note that $\eta_i \in H^1(C_2, \Z\{r_{1,i}\}) $ is a generator for $i=0,1$. Further, 
\[
\bar{r}_{1,0}\eta_1 =   \bar{r}_{1,1}\eta_0.
\]
\end{notation}

We also single out a particular \(RO(C_{4})\)-graded homotopy class.
\begin{defn}
Let 
\[
\dfrak = N_{2}^{4}(\bar{r}_{1,0}) \in \pi_{\rho_4}^{C_4}E.
\]
This is an actual homotopy class (and hence a permanent cycle in the homotopy fixed point spectral sequence).
\end{defn}

The class \(\dfrak\) is a unit. In the homotopy fixed point spectral sequence, the classes \(u_{V}\) for orientable \(V\) (introduced in \Cref{rem:uvsforE}) are also units. These give us several ways to rewrite classes. For example,
\[
\Delta_{1}=\dfrak^{2}u_{2\sigma}u_{\lambda}^{2}.
\]

The unit \(\dfrak\) also greatly simplifies the homotopy fixed point \(E_{2}\)-term analysis. The following result is immediate since $\dfrak$ and $u_{\lambda}$ are units.
\begin{prop}\label{prop:pMults}
If \(\mathfrak p_{1}=(\dfrak u_{\lambda})^{-1}\), then multiplication by \(\mathfrak p_{1}\) induces an isomorphism
\[
E_{2}^{\ast,\star}\to E_{2}^{\ast,\star-3-\sigma},
\]
and similarly for Mackey functors.
\end{prop}

This means we need only determine the integrally graded stems to have completely determined the \(RO(C_{4})\)-graded stems.

\begin{prop}\label{coh-c4}
There are classes
\[
\eta \in H^1(C_4,E_2)\qquad \nu \in H^1(C_4,E_4)
\]
and 
\[
\varsigma \in H^1(C_4,E_6)\qquad \varpi \in H^2(C_4,E_8)
\]
such that
\[
H^\ast (C_4,E_\star) \cong \mathbb{W}[\![\mu]\!][T_2,\Delta_1^{\pm 1},\eta,\nu,\varsigma, \varpi][\mathfrak p_{1}^{\pm 1}, u_{\lambda}^{\pm 1}]/\sim
\]
where $\sim$ is the ideal of relations given by
\begin{align*}
2\eta &= 2\nu= 2\varsigma = 4\varpi = 0 ;&  T_2^2 & =  \Delta_1((\mu-2)^2 + 4);  \\
\Delta_1 \eta^2  &= T_2\varpi = \varsigma^2;  & T_2 \varsigma  &= \mu\Delta_1\eta; \\
 \qquad T_2\eta &= \mu\varsigma ; & \varsigma\eta &= \mu\varpi; \\
\nu^2 &= 2\varpi; &   \mu\nu&=\eta\nu = T_2\nu=\varsigma\nu= 0.
\end{align*}
\end{prop}
\begin{rem}
Note that $\eta^3$ is $\mu$ divisible since $   \eta^3=\eta\varsigma^2\Delta_1^{-1}=\mu\varsigma\varpi\Delta_1^{-1}$.
\end{rem}

\begin{rem}\label{rem:HHRtranslation}
The classes $a_V$ and $u_V$ were introduced in \Cref{notn:aV} and \Cref{rem:uvsforE}. There is also a class $\eta' \in \pi_{2-\sigma}^{C_4}E$
 detected in $H^1(C_4, \pi_{3-\sigma}E)$. See \cite[Table 3]{HHRC4}. It satisfies $\res_2^4(\eta') = u_{\sigma}(\eta_0+\eta_1)$.

A dictionary with \cite[Table 3]{HHRC4}, and some useful relations, are then given by
\begin{align*}
\varsigma &= \eta' u_{\lambda} \dfrak & \varpi &= a_{\lambda} u_{\lambda} u_{2 \sigma} \dfrak^2 \\
\Delta_1&= u_{2\sigma} u_{\lambda}^2 \dfrak^2 & \nu &= a_{\sigma} u_{\lambda} \dfrak.
\end{align*}
Note further that $\varpi\Delta_1^{-1} = a_{\lambda} u_{\lambda}^{-1}$.

Under some choice of isomorphism, we can identify our classes with those in \cite[Perspective 2]{BO}. Again, the elements $\eta$ and $\nu$ have the same name there as here. Further,
\begin{align*}
\Delta_1 & = \delta & T_2 &= b_2 \\
\varsigma &= \widetilde{\gamma} &  T_4 &= b_4 \\
 \varpi &= \xi &  \varpi\Delta_1^{-1}  &= \beta .
 \end{align*}
In \cite[Proposition 2.10]{BobkovaGoerss}, $\varsigma$ corresponds to $\gamma$ and our $\mu$ is congruent to their $z$ modulo $2$. 
\end{rem}

\begin{table}[H]
\begin{center}
\centering
\begin{tabular}{ |C|C|C|C|C|C| }
\hline
   \multicolumn{4}{|C|}{\makecell{{\nabla}({e})={\nabla}({\gamma})={\nabla}({\gamma^2})={\nabla}({\gamma^3})=1, \\ \ {\Delta}(1)=\Delta, \ {\Delta}(e)=e+\gamma^2, \  {\Delta}(\gamma)=\gamma+\gamma^3, }}  \\
 \hline
 \hline
  \mackeybox & \mackeyboxhat & \mackeyboxoverline & \mackeycirc     \\ 
 \hline
  \Mackey{\W}{\W}{\W}{2}{1}{2}{1} 
  & \Mackey{\W}{\W[C_4/{C_2}]}{\W[C_4/{C_2}]}{2}{1}{\nabla}{\Delta}
   &\Mackey{0}{\W_{-}}{\W_{-}}{2}{1}{}{}
  & \Mackey{\mathbb{W}/4}{\fk}{0}{}{}{2}{1} 
 \\
 \hline
   \mackeybullet & \mackeybulletoverline  & \mackeyblacktriangledown &  \mackeyblacktriangle   \\
 \hline
  \Mackey{\fk}{0}{0}{}{}{}{} 
 &   \Mackey{0}{\fk}{0}{}{}{}{}
 &  \Mackey{\fk}{\fk}{0}{}{}{1}{0} 
 &\Mackey{\fk}{\fk}{0}{}{}{0}{1} 
 \\
 \hline
\mackeyeinfsix  &  \mackeyboxleftslash & \mackeyboxhatoverline  &  \mackeyeinften     \\
 \hline
  \Mackey{\fk}{\fk}{\W[C_4/{C_2}]_-}{\nabla}{0}{0}{1}
  & \Mackey{\W}{\W}{\W}{2}{1}{1}{2} 
  &    \Mackey{0}{0}{\W[C_4/{C_2}]_{-}}{}{}{}{}
&  \Mackey{\fk}{\fk[C_4/{C_2}]}{\W[C_4/{C_2}]_-}{1}{0}{\nabla}{\Delta}
      \\
  \hline
   \mackeyboxleftslashoverline
  & \mackeyboxleftslashhat
  &\mackeyhalftopbox
& \mackeybullethat
   \\ 
  \hline
  \Mackey{0}{\W_{-}}{\W_{-}}{1}{2}{}{} 
 &   \Mackey{\W}{\W[C_4/{C_2}]}{\W[C_4/{C_2}]}{1}{2}{\nabla}{\Delta}
 &  \Mackey{\W}{\W}{\W}{1}{2}{2}{1}  
 &  \Mackey{\fk}{\fk[C_4/{C_2}]}{0}{}{}{\nabla}{\Delta}
 \\
 \hline
 \end{tabular}
\caption{Mackey functors similar to those of \cite[Table 2]{HHRC4}.}
\label{table:Mackey}
\end{center}
\end{table}

 \begin{rem}
Multiplication by the element $\Delta_1$ induces an isomorphism $\Sigma^8 E_{t} \to E_{t+8}$ and multiplication by $\varpi$ induces an isomorphism 
\[\underline{H}^s(C_4, E_t) \to \underline{H}^{s+2}(C_4, E_{t+8})\] 
for all $s\geq 1$, giving $\underline{H}^*(C_4, E_*)$ two periodicities.
\end{rem}
   
\Cref{coh-c2} and \Cref{coh-c4} identify the layers of \eqref{eq:mackey}. It still remains to identify the transfers and restrictions and this is the goal of the remainder of this section. To do this, we use \Cref{lem:reps}. The first step is the computation of the Mackey functors 
\begin{align*}
&\underline{H}^*(C_4, \W)  &  &\underline{H}^*(C_4, \W[C_4/C_2]) \\
&\underline{H}^*(C_4, \W_{-}) & & \underline{H}^*(C_4, \W[C_4/C_2]_{-}),
\end{align*}
which is straightforward from the definitions. The Mackey functors we need are listed in \Cref{table:Mackey}.

Whenever possible, we use notation reminiscent of that used in \cite{HHRC4}. 
\begin{lem}\label{prop:cohmackey}
For the Mackey functors listed in \Cref{table:Mackey}, there are isomorphisms
\begin{align*}
\underline{H}^*(C_4, \W)&\cong
\begin{cases}
\mackeybox & \text{if} *=0,\\
\mackeycirc & \text{if}*=2n,\\
0 & \text{if} *=2n-1,
\end{cases} 
\\
 \underline{H}^*(C_4, \W[C_4/C_2])&\cong
\begin{cases}
\mackeyboxhat & \text{if} *=0,\\
\mackeybullethat & \text{if} *=2n,\\
0 & \text{if} *=2n-1,
\end{cases} 
\\
 \underline{H}^*(C_4, \W_{-})&\cong
\begin{cases}
\mackeyboxoverline & \text{if} *=0,\\
\mackeybulletoverline & \text{if} *=2n,\\
\mackeybullet & \text{if} *=2n-1,
\end{cases} 
\\
\underline{H}^*(C_4, \W[C_4/C_2]_{-})&\cong
\begin{cases}
\mackeyboxhatoverline  & \text{if} *=0,\\
0 & \text{if} *=2n,\\
\mackeybullethat & \text{if} *=2n-1.
\end{cases}
\end{align*}
where $n>0$ in the above formulas. 
\end{lem}

\begin{notation}
 In general, if $\sun$ is a Mackey functor in \Cref{table:Mackey}, then $\underline{\sun}$ is the Mackey functor defined as 
 \[
 \underline{\sun}(G/H)= \W [\![ \mu ]\!] \otimes_{\W}\sun(G/H),
 \]
  with restrictions and transfers extended to be $\W [\![ \mu ]\!]$-linear. 
\end{notation}

\begin{prop}\label{prop:cohmackeycomplete}
For the Mackey functors listed in \Cref{table:Mackey} and \Cref{table:Mackey-two}, there are isomorphisms
\begin{align*}
\underline{H}^*(C_4, A(+))&\cong
\begin{cases}
\mackeyunderlinedboxdothat & *=0,\\
\mackeyunderlinedodothat &*=2n,\\
0 & *=2n-1,
\end{cases} & \underline{H}^*(C_4, A)&\cong
\begin{cases}
\mackeyunderlinedboxhat & *=0,\\
\mackeyunderlinedbullethat & *=2n,\\
0 &*=2n-1,
\end{cases} 
\\
 \underline{H}^*(C_4, A(-))&\cong
\begin{cases}
\mackeyunderlinedboxdottilde & *=0,\\
\mackeyunderlinedbullettilde & *=2n,\\
\mackeybullet & *=2n-1,
\end{cases} 
&
\underline{H}^*(C_4, A_-)&\cong
\begin{cases}
\mackeyunderlinedboxhat  & *=0,\\
0 & *=2n,\\
\mackeyunderlinedbullethat & *=2n-1.
\end{cases}
\end{align*}
\end{prop}
\begin{proof}
Since $\mu$ is fixed by the action of $C_4$, there are isomorphisms
\begin{align*}
\underline{H}^*(C_4, A) &\cong \W[\![\mu]\!] \otimes_{\W} \underline{H}^*(C_4, \W[C_4/C_2]) \\
 \underline{H}^*(C_4, A_-) &\cong \W[\![\mu]\!] \otimes_{\W} \underline{H}^*(C_4, \W[C_4/C_2]_-) .
\end{align*}
The exact sequences \eqref{eq:exactA} give rise to exact sequences of Mackey functors in cohomology:
  \begin{align*}\xymatrix@R=0.6pc{0 \ar[r] & \mackeyunderlinedboxhat  \ar[r] &  \mackeyunderlinedboxdothat  \ar[r] &    \mackeybox  \ar[r] &   0 \\
   0  \ar[r] & \mackeyunderlinedboxhat  \ar[r] &  \mackeyunderlinedboxdottilde  \ar[r] &  \mackeyboxoverline   \ar[r] & 0 \\
    0  \ar[r] & \mackeyunderlinedbullethat  \ar[r] &  \mackeyunderlinedodothat  \ar[r] & \mackeycirc  \ar[r] & 0\\
  0  \ar[r] & \mackeyunderlinedbullethat  \ar[r] &  \mackeyunderlinedbullettilde  \ar[r] & \mackeybulletoverline    \ar[r] & 0 .}\end{align*}
  The definitions of the middle terms in these short exact sequences are given in \Cref{table:Mackey-two} and use the ring structure of $A(-)$ and $A(+)$.
  Note that all the boundary maps in the long exact sequences in cohomology are trivial due to the structure of the Mackey functors involved. 
\end{proof}

\begin{rem}
Since $E$ is a ring spectrum, we have the Frobenius identity
\[\tr_K^H(\res_K^H(a)x) = a \tr_K^H(x)\]
for $a\in \underline{\pi}_*E(G/H)$, $x\in \underline{\pi}_*E(G/K)$ and $K\subseteq H \subseteq G$.
\end{rem}

Together, \Cref{lem:reps} and \Cref{prop:cohmackeycomplete} give the following transfers and restrictions. Note that all other transfers follow from these using the Frobenius identity and the multiplicative structure of the restriction.
\begin{prop}\label{prop:resandtr}
In the Mackey functor \eqref{eq:mackey}, there are restrictions
\begin{align*}
\res_2^4(\mu )&= \mu_0+\mu_1
  &    \res_2^4(\eta)&= \eta_0+\eta_1
  , \\
\res_2^4(\Delta_1)&=\delta_1^2
,
 & \res_2^4(\varpi)&= \Sigma_{2,0}\eta_1^2=-\Sigma_{2,1}\eta_0^2
 ,\\ 
 \res_2^4(T_2)&= \Sigma_{2,0}-\Sigma_{2,1}
 ,  & \res_2^4(\varsigma) &= (\eta_0+\eta_1)\delta_1
 , 
\end{align*}
and $\res_2^4(\nu)=0$. In particular, $\res_2^4(\Delta_1^{-1}\varpi^2) = (\eta_0 \eta_1)^2$.

There are transfers
\begin{align*}
    \tr_2^4(\mu_0)&= \mu ,    &   \tr_2^4(\Sigma_{2,0})&= T_2,&    \tr_2^4(\eta_1\Sigma_{2,0})&= \varsigma ,  \\
 \tr_2^4(\eta_0)& = \eta,      &  \tr_2^4(\eta_0\eta_1)&=0,  &  \tr_2^4(\eta_0^2\eta_1) &= \varsigma\varpi \Delta_1^{-1}   
\end{align*}
together with $\tr_2^4(1)=2$.
\end{prop}

\begin{rem}
As already noted, \Cref{prop:resandtr} allows one to compute various other transfers. For example,
$\tr_2^4(\eta_0^2)$ is obtained from the above formulas by
\begin{align*}
\tr_2^4(\eta_0^2)&=\tr_2^4( \eta_0^2 +\eta_0\eta_1)  
= \tr_2^4( \res_2^4(\eta) \eta_0) 
= \eta \tr_2^4(\eta_0) = \eta^2.
\end{align*}
\end{rem}

\begin{table}[H]
\begin{center}
 \begin{tabular}{ |C|C|C|C|C|C|C|C| }
 \hline
 \multicolumn{2}{|C|}{\wmu = \W[\![\mu]\!][C_4/C_2] } &  \multicolumn{2}{|C|}{   \wmum= \W[\![\mu]\!][C_4/C_2]_- } \\
  \multicolumn{2}{|C|}{  \wmuP = \frac{\W[\![ \mu ]\!][{C_4/C_2}_+]}{\mu \cdot \ast =\Delta} } &  \multicolumn{2}{|C|}{  \wmuM=  \frac{\W[\![\mu] \!][C_4/{C_2}_-]}{\mu\cdot \ast = \bar{\Delta}+2 \cdot \ast} }\\
   \hline
      \hline
 \multicolumn{4}{|C|}{\widehat{\nabla}(\ast)=2 \ \ \ \widehat{\nabla}({e})=\widehat{\nabla}({\gamma})=\mu \ \ \ \widehat{\Delta}(1)=\ast} \\
  \multicolumn{4}{|C|}{\widetilde{\nabla}(\ast)=0 \ \ \ \widetilde{\nabla}({e})=\widetilde{\nabla}({\gamma})=1 \ \ \ \widetilde{\Delta}(1)=\Delta} \\
 \hline
   \hline
\mackeyunderlinedboxhat & \mackeyunderlinedboxhatoverline & \mackeyunderlinedboxdothat  & \mackeyunderlinedboxdottilde \\
 \hline
 \Mackey{\W[\![\mu ]\!]}{\wmu}{\wmu}{2}{1}{\nabla}{\Delta}
 &\Mackey{0}{0}{\wmum}{}{}{}{}
  &\Mackey{\W[\![\mu ]\!]}{\wmuP}{\wmuP}{2}{1}{\widehat{\nabla}}{\widehat{\Delta}}
  &\Mackey{\W[\![\mu ]\!]}{\wmuM}{\wmuM}{2}{1}{\widetilde{\nabla}}{\widetilde{\Delta}}
 \\
 \hline
\mackeyunderlinedboxleftslashtilde & \mackeyunderlinedbullethat &  \mackeyunderlinedodothat &\mackeyunderlinedbullettilde  \\ 
 \hline
   \Mackey{\W[\![\mu ]\!]}{\wmuM}{\wmuM}{1}{2}{\widetilde{\nabla}}{\widetilde{\Delta}}
&   \Mackey{\fk[\![\mu ]\!]}{\wmu/2}{0}{}{}{{\nabla}}{{\Delta}}
 & \Mackey{\frac{\mathbb{W}[\![\mu ]\!]}{(4,2\mu )}}{\wmuP/2}{0}{}{}{\widehat{\nabla}}{\widehat{\Delta}} 
 &\Mackey{\fk[\![\mu ]\!]}{\wmuM/2}{0}{}{}{\widetilde{\nabla}}{\widetilde{\Delta}}
 \\
  \hline
 \mackeyunderlinedboxleftslashhat & \mackeyunderlinedboxblackboxhat  &   \mackeyunderlinedboxboxhat & \mackeyunderlinedcirccirchat  \\ 
 \hline
  \Mackey{\W[\![\mu ]\!]}{A}{A}{1}{2}{\nabla}{\Delta} &
    \Mackey{(4,\mu)\W[\![\mu ]\!]}{(2{*}, e,\gamma)\wmuP}{\wmuP}{2}{1}{\widehat{\nabla}}{\widehat{\Delta}} 
 &     \Mackey{(2,\mu)\W[\![\mu ]\!]}{\wmuP}{\wmuP}{2}{1}{\widehat{\nabla}}{\widehat{\Delta}} 
 & \Mackey{(2,\mu)\frac{\mathbb{W}[\![\mu ]\!]}{(4,2\mu )}}{\wmuP/2}{0}{}{}{\widehat{\nabla}}{\widehat{\Delta}} 
  \\
 \hline
\mackeyunderlinedblacktrianglehat  &  
\mackeyunderlinedhalfbottomboxhat &
 \mackeyunderlinedhalftopboxhat   &  \mackeyunderlinedboxboxtilde \\ 
 \hline
  \Mackey{\fk[\![\mu ]\!]}{\wmuP/2 }{0}{}{}{\widehat{\nabla}}{\widehat{\Delta}}
&   \Mackey{(2,\mu)\W[\![\mu ]\!]}{\wmuP}{\wmuP}{1}{2}{\widehat{\nabla}}{\widehat{\Delta}} 
 &    \Mackey{\W[\![\mu ]\!]}{\wmuP}{\wmuP}{1}{2}{\widehat{\nabla}}{\widehat{\Delta}} 
 &   \Mackey{\W[\![\mu ]\!]}{(2{*}, e,\gamma)\wmuM}{\wmuM}{2}{1}{\widetilde{\nabla}}{\widetilde{\Delta}}
  \\
 \hline
 \end{tabular}
\caption{The $C_4$ Mackey functors in the category of $\W[\![\mu ]\!]$-modules appearing in $\underline{E}_2^{*,*}$.  In general, if $\sun$ is a Mackey functor in \Cref{table:Mackey}, then $\underline{\sun}$ is defined by $\underline{\sun}(G/H)= \W [\![ \mu ]\!] \otimes_{\W}\sun(G/H)$, with restriction and transfers extended to be $\W [\![ \mu ]\!]$-linear. See \Cref{notn:A} for more details on the notation.}
\label{table:Mackey-two}
\end{center}
\end{table}

\begin{table}[H]
\begin{center}
 \begin{tabular}{ |C|C|C|C|C|C|C|C| }
 \hline
   \multicolumn{4}{|C|}{ \widecheck{\widetilde{\nabla}}(\ast)= 0  \ \ \ \widecheck{\widetilde{\nabla}}({e})=\widecheck{\widetilde{\nabla}}({\gamma})=1 \ \ \  \widecheck{\widetilde{\nabla}}(\iota_0)=\iota, \widecheck{\widetilde{\Delta}}(1)=\Delta  \ \ \ \widecheck{\widetilde{\Delta}}(\iota) =0 }  
  \\
  \hline
   \multicolumn{4}{|C|}{\breve{\widehat{\nabla}}(\ast)= 2 \ \ \ \breve{\widehat{\nabla}}({e})=\breve{\widehat{\nabla}}({\gamma})=\mu \ \ \   \breve{\widehat{\nabla}}(\iota_0)= 0 \ \ \ \breve{\widehat{\Delta}}(1)= \ast }  
  \\
 \hline
  \multicolumn{4}{|C|}{\mathring{\widetilde{\nabla}}(\ast)= 2\iota \ \ \mathring{\widetilde{\nabla}}({e})=\mathring{\widetilde{\nabla}}({\gamma})=1 \ \ \ \mathring{\widetilde{\nabla}}(\iota_0)=2\iota \ \ \ \mathring{\widetilde{\Delta}}(1)=\Delta \ \ \ \mathring{\widetilde{\Delta}}(\iota)=\iota_0   } 
  \\
  \hline
  \multicolumn{4}{|C|}{\dot{\widetilde{\nabla}}(\ast)= \iota \ \ \ \dot{\widetilde{\nabla}}({e})=\dot{\widetilde{\nabla}}({\gamma})=1 \ \ \ \dot{\widetilde{\Delta}}(1)=\Delta \ \ \ \dot{\widetilde{\Delta}}(\iota)= 0  } 
     \\
 \hline
   \multicolumn{4}{|C|}{\widecheck{\widehat{\nabla}}(\ast)=2, \ \widecheck{\widehat{\nabla}}(e)=\widecheck{\widehat{\nabla}}(\gamma)=\mu  \ \ \ \widecheck{\widehat{\nabla}}(\iota_0)=\iota \ \ \ \widecheck{\widehat{\Delta}}(1)=\iota_0+\ast \ \ \ \widecheck{\widehat{\Delta}}(\iota)=0 }
   \\  
   \hline
     \multicolumn{4}{|C|}{\widebridgeabove{\widetilde{\nabla}}(2\ast)= \iota \ \ \ \widebridgeabove{\widetilde{\nabla}}({e})=\widebridgeabove{\widetilde{\nabla}}({\gamma})=1 \ \ \ \widebridgeabove{\widetilde{\Delta}}(1)=\Delta \ \ \  \widebridgeabove{\widetilde{\Delta}}(\iota)= 0  }  \\
   \hline
   \hline
\mackeyeinftwoethy  & \mackeyeinffourethy &  \mackeyeinfeighteenethy  &\mackeyeinfsixethy   \\
 \hline
  \Mackey{(2,\mu)\frac{\mathbb{W}[\![\mu ]\!]}{(4,2\mu )}}{\wmuP/2}{\wmum}{1}{0}{\widehat{\nabla}}{\widehat{\Delta}} 
  &\Mackey{\fk\{\iota\} \oplus \W[\![\mu ]\!] }{\fk\{\iota_0\} \oplus\wmuM}{\wmuM}{\left[ \begin{matrix}1 \\ 1 \end{matrix}\right]}{\left[ \begin{matrix}0 & 2 \end{matrix}\right]}{\widecheck{\widetilde{\nabla}}}{\widecheck{\widetilde{\Delta}}}
  &   \Mackey{\fk[\![\mu ]\!]}{\wmuM/2}{\wmum}{1}{0}{\widetilde{\nabla}}{\widetilde{\Delta}} 
  &   \Mackey{\fk}{\fk}{\wmum}{\nabla}{0}{0}{1}
 \\
 \hline
 \mackeyeinftwentytwoethy  & \mackeyeinftwentyeightethy  & \mackeysigmaminusoneethy   & \mackeyunderlinedboxdotcheck \\
 \hline
  \Mackey{\W/4}{\fk}{\wmum}{\nabla}{0}{2}{1}
& \Mackey{\fk\{\iota\} \oplus \W[\![\mu ]\!] }{ \wmuM}{\wmuM}{1}{2}{\dot{\widetilde{\nabla}}}{\dot{\widetilde{\Delta}}}
& \Mackey{\fk\{\iota\} \oplus \W[\![\mu ]\!] }{ \wmuM}{ \wmuM }{2}{1}{\dot{\widetilde{\nabla}}}{\dot{\widetilde{\Delta}}}
&  \Mackey{\frac{\mathbb{W}[\![\mu ]\!]}{(4,2\mu )}}{\wmuP/2}{\wmum}{1}{0}{\widehat{\nabla}}{\widehat{\Delta}} 
 \\
 \hline
  \mackeyunderlinedblackboxwhitecirchatoverline  & \mackeyeinffoursigmaethy  & \mackeysigmaminusnineethy   & \mackeyeinftenethy  \\
 \hline
  \Mackey{0}{\fk}{\wmum}{\nabla}{0}{}{}
& \Mackey{\W[\![\mu ]\!] }{\fk\{\iota_0\} \oplus \wmuP}{\wmuP }{\left[ \begin{matrix}1 \\ 1 \end{matrix}\right]}{\left[ \begin{matrix}0 & 2 \end{matrix}\right]}{\breve{\widehat{\nabla}}}{\breve{\widehat{\Delta}}}
& \Mackey{\fk\{\iota\} \oplus \W[\![\mu ]\!] }{(2{*}, e,\gamma)\wmuM }{\wmuM }{2}{1}{\widebridgeabove{\widetilde{\nabla}}}{\widebridgeabove{\widetilde{\Delta}}}

  &    \Mackey{\fk[\![\mu]\!]}{\wmu/2}{\wmum}{1}{0}{\nabla}{\Delta}
 
 \\
 \hline
   \multicolumn{2}{|C|}{   \mackeyeinftwentyethy}  &    \multicolumn{2}{|C|}{  \mackeyunderlinedhalftopboxcheck } \\
 \hline
   \multicolumn{2}{|C|}{   \Mackey{\W/4\{\iota\} \oplus \W[\![\mu ]\!] }{\fk\{\iota_0\} \oplus \wmuM}{\wmuM}{1}{2}{\mathring{\widetilde{\nabla}}}{\mathring{\widetilde{\Delta}}}}
&
   \multicolumn{2}{|C|}{ \Mackey{\fk\{\iota\} \oplus \W[\![\mu ]\!] }{\fk\{\iota_0\} \oplus \wmuP}{ \wmuP}{\left[ \begin{matrix}0 \\ 1 \end{matrix}\right]}{\left[ \begin{matrix}0 & 2 \end{matrix}\right]}{\widecheck{\widehat{\nabla}}}{\widecheck{\widehat{\Delta}}}}
 
 \\
 \hline
 \end{tabular}
\caption{$C_4$ Mackey functors in the category of $\W[\![\mu ]\!]$-modules appearing in $\underline{E}_2^{*,*}$. See \Cref{notn:A} for notation.}
\label{table:Mackey-three}
\end{center}
\end{table}

We represent the result of the computations of this section in the chart in \Cref{figure:additive-E2}. The relevant Mackey functors are depicted in  \Cref{table:Mackey} and \Cref{table:Mackey-two}.

\subsubsection{Generators for the $\underline{E}_2^{*,*}$-term}\label{sec:genE2}
We now list generators of $\underline{E}_2^{s,t} = \underline{H}^s(C_4, E_t)$
as a $\W[\![\mu ]\!]$-module. The class 
\[\varpi \Delta_1^{-1}\in H^2(C_4, E_0)\]
and its restriction to 
\[\eta_0^2 \Sigma_{2,0}^{-1} \in H^2(C_2, E_0)\]  
play a central role. 
\begin{enumerate}
   \item In  degrees $t=8l$ with $k\geq 1$ we have
   \begin{align*}
     \mackeyunderlinedboxdothat &\subseteq \underline{E}_2^{0, 8l} \hspace{0.7cm} & &\text{generated by }  \Delta_1^l \in  \underline{E}_2^{0, 8l} (C_4/C_4)  \\
   & &  & \text{and by }\delta_1^{2l}, \mu_0 \delta_1^{2l}   \in \underline{E}_2^{0,  8l}(C_4/C_2)\\
      & &  & \text{and by } (r_{1,0}r_{1,1})^{2l}, \mu_0 (r_{1,0}r_{1,1})^{2l}   \in \underline{E}_2^{0,  8l}(C_4/\{e\})\\
   \mackeyunderlinedodothat &\subseteq \underline{E}_2^{2k, 8l} & &\text{generated by }  (\varpi\Delta_1^{-1})^k\Delta_1^l \in  \underline{E}_2^{2k, 8l} (C_4/C_4) \\
    & &  &  \text{and by }  (\eta_0^2 \Sigma_{2,0}^{-1})^k\delta_1^{2l}, \mu_0  (\eta_0^2 \Sigma_{2,0}^{-1})^k \delta_1^{2l}   \in \underline{E}_2^{2k,  8l}(C_4/C_2).\hspace{1.5cm}
   \end{align*}

\item In degrees $t=2+8l$ with $k\geq 0$ we have
   \begin{align*}
   \mackeyunderlinedboxhatoverline &\subseteq \underline{E}_2^{0,  2+8l} &  & \text{generated by } r_{1,0} (r_{1,0}r_{1,1})^{2l} , \mu_0r_{1,0} (r_{1,0}r_{1,1})^{2l}  \in  \underline{E}_2^{0,  2+8l}(C_4/\{e\}) \\
   \mackeyunderlinedbullethat &\subseteq \underline{E}_2^{1+2k,  2+8l} &  & \text{generated by } \eta  (\varpi\Delta_1^{-1})^k\Delta_1^l  \in \underline{E}_2^{1+2k,  2+8l}(C_4/C_4)\\
& &  & \text{and by } \eta_0 (\eta_0^2 \Sigma_{2,0}^{-1})^k \delta_1^{2l}  ,  \mu_0\eta_0 (\eta_0^2 \Sigma_{2,0}^{-1})^k \delta_1^{2l}   \in \underline{E}_2^{1+2k,  2+8l}(C_4/C_2).
      \end{align*}

\item In degrees $t=4+8l$ with $k\geq 0$ we have
\begin{align*}
\mackeyunderlinedboxdottilde&\subseteq \underline{E}_2^{0, 4+8l} & &\text{generated by } T_2\Delta_1^{l}  \in \underline{E}_2^{0, 4+8l}(C_4/C_4) \\
& &  &    \text{and by }   \delta_1^{2l+1}, \Sigma_{2,0} \delta_1^{2l} \in \underline{E}_2^{0,  4+8l}(C_4/C_2) \\
& &  &    \text{and by } r_{1,0}^2 (r_{1,0}r_{1,1})^{2l}, \mu_0r_{1,0}^2 (r_{1,0}r_{1,1})^{2l} \in \underline{E}_2^{0,  4+8l}(C_4/\{e\}) \\
 \mackeybullet &\subseteq \underline{E}_2^{1+2k, 4+8l} & &\text{generated by } \nu  (\varpi\Delta_1^{-1})^k \Delta_1^{l} \in \underline{E}_2^{1+2k, 4+8l}(C_4/C_4) \\
\mackeyunderlinedbullettilde &\subseteq \underline{E}_2^{2(k+1), 4+8l} & &\text{generated by } \eta^2 (\varpi \Delta_1^{-1})^{k} \Delta_1^{l} \in \underline{E}_2^{2(k+1), 4+8l} (C_4/C_4)\\
& &  & \text{and by }   \eta_0^2(\eta_0^2 \Sigma_{2,0}^{-1})^{k}\delta_1^{2l},  \mu_0 \eta_0^2(\eta_0^2 \Sigma_{2,0}^{-1})^{k}\delta_1^{2l} \in \underline{E}_2^{2(k+1), 4+8l}(C_4/C_2)   .
\end{align*}

\item In degrees $t=6+8l$ with $k\geq 0$ we have 
\begin{align*}
\mackeyunderlinedboxhatoverline&\subseteq \underline{E}_2^{0,  6+8l} & & \text{generated by } r_{1,0}^3 (r_{1,0}r_{1,1})^{2l}, \mu_0r_{1,0}^3(r_{1,0}r_{1,1})^{2l}  \in \underline{E}_2^{0,  6+8l}(C_4/\{e\}) \\
 \mackeyunderlinedbullethat &\subseteq \underline{E}_2^{1+2k,  6+8l} & & \text{generated by } \varsigma (\varpi \Delta_1^{-1})^k\Delta_1^{l} \in \underline{E}_2^{1+2k,  6+8l}(C_4/C_4) \\
& &  & \text{and by } \eta_0 \delta_1 (\eta_0^2 \Sigma_{2,0}^{-1})^k \delta_1^{2l},  \mu_0\eta_0\delta_1 (\eta_0^2 \Sigma_{2,0}^{-1})^k \delta_1^{2l}  \in \underline{E}_2^{1+2k,  6+8l}(C_4/C_2).
\end{align*}
 \end{enumerate}
 
In \Cref{figure:additive-E2}, we use dashed lines to indicate that $\eta$ times the generator is divisible by $\mu $ in $\underline{E}_2(C_4/C_4)$.
The dashed line from $\mackeyunderlinedbullettilde$ to $ \mackeyunderlinedbullethat $ indicates that
\[\eta (\eta^2  \Delta_1^{l}(\varpi \Delta_1^{-1})^{k-1} ) \in \underline{E}_2^{1+2k, 6+8l} (C_4/C_4) = \mu (\varsigma \Delta_1^{l} (\varpi\Delta_1^{-1})^k) \in \underline{E}_2^{1+2k,  6+8l}(C_4/C_4) ,\]
which follows from the relation $T_2\eta  = \mu  \varsigma$. The dashed line from $\mackeyunderlinedbullethat$ to $ \mackeyunderlinedodothat$ indicates that
\[ \eta (\varsigma \Delta_1^l (\varpi \Delta_1^{-1})^k)   \in \underline{E}_2^{2+2k,  8+8l}(C_4/C_4) = \mu ( \Delta_1^{l+1}(\varpi \Delta_1^{-1})^{k+1} ) \in \underline{E}_2^{2+2k, 8+8l} (C_4/C_4).  \]

\subsubsection{Description of the $\underline{E}_2^{*,1-\sigma +*}$-term}\label{sec:Emodulesign}
Finally, we pay an earlier computational debt and  turn to the description of
\[
\underline{E}_2^{*,1-\sigma +*}= \underline{H}^*(C_4, \pi_{1-\sigma+*}E).
\]

Recall that \Cref{prop:pMults} described how multiplication by the element 
$\mathfrak{p}=(u_{\lambda}\dfrak)^{-1}$
is an isomorphism. In Mackey functor language, multiplication by it induces isomorphisms:

\begin{align*}
\xymatrix@R=0.7pc{\underline{E}_2^{s,t}(C_4/C_4) \ar[r]_-{\cong}^-{\mathfrak{p}} & \underline{E}_2^{s,t-3-\sigma}(C_4/C_4) \\
\underline{E}_2^{s,t}(C_4/C_2) \ar[r]_-{\cong}^-{\res_2^4(\mathfrak{p})} & \underline{E}_2^{s,t-3-\sigma}(C_4/C_2) \\
\underline{E}_2^{s,t}(C_4/\{e\}) \ar[r]_-{\cong}^-{\res_1^4(\mathfrak{p})} & \underline{E}_2^{s,t-3-\sigma}(C_4/\{e\})
}
\end{align*}
So, using \Cref{lem:reps}, it is straightforward to compute the $\underline{E}_2^{*,1-\sigma+*}$-term. It is depicted in \Cref{figure:additive-E2-sigma}.

The class $u_{\sigma} \in \underline{E}_2^{0,1-\sigma}(C_4/C_2)$ of \Cref{rem:uvsforE} is a permanent cycle by Theorem 11.3 of \cite{HHRC4} and multiplication by $u_{\sigma}$ induces an isomorphism of spectral sequences, and similarly for multiplication by
$\bar{u}_{\sigma} :=\res_1^2(u_{\sigma})$:
\[
\xymatrix@R=0.7pc{ \underline{E}_r^{s,t}(C_4/C_2)  \ar[r]^-{u_{\sigma}}_-{\cong} & \underline{E}_r^{s,1-\sigma + t}(C_4/C_2) \\
 \underline{E}_r^{s,t}(C_4/\{e\})   \ar[r]^-{\bar{u}_{\sigma}}_-{\cong}  & \underline{E}_r^{s,1-\sigma + t}(C_4/\{e\})}\]

\begin{rem}In $\underline{E}_2^{*,1-\sigma+*}(C_4/C_4)$, the classes $T_2  \mathfrak{p}$,
$\eta' = \varsigma \mathfrak{p}$, and 
$a_{\sigma}  = \nu \mathfrak{p}$
are permanent cycles. Here, $a_{\sigma}$ is as in \Cref{notn:aV} and $\eta'$ is as in \Cref{rem:HHRtranslation}.
\end{rem}


\subsection{The differentials and the extensions}
In this section, we describe the differentials in the spectral sequence
\[  \underline{E}_{2}^{s,\star+t} = \underline{H}^s(C_4, \pi_{\star+t}E) \Longrightarrow {\spi}_{\star+t-s}E^{h} \]
for $\star =0$ and $\star = 1-\sigma$. The differentials have the form $d_r \colon \underline{E}_r^{s,\star+t} \to \underline{E}_r^{s+r, \star+t+r-1}$. As was mentioned above, the results in this section follow from the computations of \cite{HHRC4}.

\subsubsection{The $d_3$-differentials and the $\underline{E}_4$-page}

We first describe the $d_3$-differentials.

\begin{prop}\label{prop:d3s}
   In the spectral sequence 
   \[H^s(C_4, E_{\star}) \Longrightarrow \pi_{\star-s} E^{hC_4},\]
   the $d_3$-differential are $\eta$, $\nu$, $\mu$, $\Delta_1$ and $\varpi^2$ linear. 
  \noindent
  In the spectral sequence 
     \[H^s(C_4, E_{t}) \Longrightarrow \pi_{t-s} E^{hC_4},\]
they are determined by
   \begin{align*}
d_3(T_2)&=\eta^3 , & d_3(\varpi) &= \eta\varpi^2\Delta_1^{-1} , & d_3(\varsigma) &=\eta \varsigma \varpi \Delta_1^{-1}=\mu \varpi^2\Delta_1^{-1}.
   \end{align*}
     In the spectral sequence 
     \[H^s(C_4, E_{1-\sigma+t}) \Longrightarrow \pi_{1-\sigma+t-s} E^{hC_4},\]
they are determined by
      \begin{align*}
    d_3(\Delta_1\mathfrak{p})&=\eta \varpi \mathfrak{p} ,  & d_3(\varpi \varsigma \mathfrak{p}) &=\eta \varsigma \varpi^2 \Delta_1^{-1}  \mathfrak{p} & 
   \end{align*} 
     In the spectral sequence
      \[H^s(C_2, E_t) \Longrightarrow \pi_{t-s} E^{hC_2},\]
the $d_3$-differentials are $\mu_0$, $\eta_0$ and $\Sigma_{2,0}^2$-linear and are determined by
\begin{align*}
d_3(\Sigma_{2,0}) &= \mu_0\eta_0^3.
\end{align*}
\end{prop}

\begin{proof}
We refer the reader to the differentials listed in \cite[Table 3]{HHRC4}.

The differentials in the spectral sequence for $C_4$ follow from \Cref{rem:HHRtranslation} and the differential 
\begin{align*}d_{3}(u_{\lambda}) &= \eta a_{\lambda}.
\end{align*} 
The $d_3$ differential for $C_2$ is the \cite{HHRC4} differential
\[d_3(\Sigma_{2,0}) = \eta_0^2(\eta_0 +\eta_1)  =  \eta_0^2(\eta_0 +(1 + \mu_0)\eta_0)  = \mu_0 \eta_0^{3} . \qedhere\]
\end{proof}

Both $\Delta_1$ and $\varpi^2$ are $d_3$ cycles and $(\underline{E}_3,d_3)$ is a module over $\W[\![\mu ]\!][\varpi^2, \Delta_1^{\pm1}]$. Using this module structure, it suffices to describe the following five differentials to determine the $\underline{E}_4^{*,*}$-page as a Mackey functor. The $\underline{E}_4$-page is illustrated in \Cref{fig:E5page}. The relevant exact sequences of Mackey functors are depicted in \Cref{fig:exactsequencesds}. See also \cite[Section 5, 13]{HHRC4}.

\begin{enumerate}
   \item For $d_3 \colon \underline{E}_3^{0, 4} \to \underline{E}_3^{3, 6}$ we have an exact sequence
\begin{align*}
\xymatrix{0 \ar[r] &  \mackeyunderlinedboxleftslashtilde \ar[r] &  \mackeyunderlinedboxdottilde \ar[r]^-{d_3}  & \mackeyunderlinedbullethat \ar[r] &  \mackeyblacktriangledown \ar[r] &  0}
\end{align*}
determined by $d_3(T_2) = \eta^3$ and $d_3(\Sigma_{2,0})=\mu_0\eta_0^3$. This gives the following remaining classes
   \begin{align*}
   \mackeyunderlinedboxleftslashtilde &\subseteq \underline{E}_4^{0, 4} & &\text{generated by } 2T_2  \in \underline{E}_2^{0, 4}(C_4/C_4) \\
& &  &    \text{and by }  2\delta_1 , 2\Sigma_{2,0}  \in \underline{E}_4^{0,  4}(C_4/C_2) \\
& &  &    \text{and by }  r_{1,0}r_{1,1} , r_{1,0}^2 \in \underline{E}_4^{0,  4}(C_4/\{e\}) \\
 \mackeyblacktriangledown &\subseteq \underline{E}_4^{3,  6} & & \text{generated by } \varsigma (\varpi \Delta_1^{-1}) \in \underline{E}_4^{3,  6}(C_4/C_4) \\
& &  & \text{and by } \eta_0 \delta_1  (\eta_0^2 \Sigma_{2,0}^{-1}) \equiv \eta_0^3  \in \underline{E}_4^{3,  6}(C_4/C_2).
   \end{align*}
   The following commutative diagram, with rows and columns exact, may help the reader relate this family of $d_3$-differentials to those of \cite[Section 13]{HHRC4}:
   \[\xymatrix@=1pc{  & 0 \ar[d] & 0\ar[d] & 0 \ar[d] &  & \\
   0 \ar[r] & \mackeyunderlinedboxleftslashhat \ar[r] \ar[d] & \mackeyunderlinedboxhat \ar[r]^-{d_3} \ar[d] & \mackeyunderlinedbullethat \ar[r] \ar[d] & 0 \ar[d] & \\ 
   0 \ar[r] &  \mackeyunderlinedboxleftslashtilde \ar[d] \ar[r] &  \mackeyunderlinedboxdottilde \ar[d] \ar[r]^-{d_3}  & \mackeyunderlinedbullethat \ar[r] \ar[d] &  \mackeyblacktriangledown \ar[d] \ar[r] &  0 \\ 
      0 \ar[r] &  \mackeyboxleftslashoverline \ar[d] \ar[r] &  \mackeyboxoverline \ar[d] \ar[r]^-{d_3}  & \mackeybullethat \ar[d] \ar[r] &  \mackeyblacktriangledown \ar[d] \ar[r] &  0 \\ 
      & 0 & 0 & 0 & 0 &  }\]

\item For $d_3 \colon \underline{E}_3^{1, 6} \to \underline{E}_3^{4, 8}$, we have an exact sequence
\begin{align*}
\xymatrix{0 \ar[r] & \mackeyunderlinedbullethat \ar[r]^-{d_3} & \mackeyunderlinedodothat \ar[r] & \mackeycirc \ar[r] & 0 }
\end{align*}
determined by the differentials $d_3(\varsigma) = \mu \varpi^2\Delta_1^{-1}$ and $d_3(\eta_0\Sigma_{2,0}) = \mu_0\eta_0^4$.
This gives the following remaining classes
\begin{align*}
 \mackeycirc&\subseteq \underline{E}_4^{4,  8} & & \text{generated by } \varpi^2\Delta_1^{-1}  \in \underline{E}_4^{4,  8}(C_4/C_4) \\
& &  & \text{and by } (\eta_0 \eta_1)^2  \equiv \eta_0^4    \in \underline{E}_4^{4,  8}(C_4/C_2).
\end{align*}
   There is a commutative diagram:
   \[\xymatrix@=1pc{ & 0 \ar[d] & 0 \ar[d] &  &  \\
    0 \ar[r] & \mackeyunderlinedbullethat \ar[r]^-{d_3} \ar[d] &  \mackeyunderlinedbullethat \ar[d] \ar[r] & 0 \ar[d] & \\
   0 \ar[r] & \mackeyunderlinedbullethat \ar[r]^-{d_3} \ar[d] & \mackeyunderlinedodothat \ar[r] \ar[d]  & \mackeycirc \ar[r]  \ar[d] & 0 \\
   & 0 \ar[r] &  \mackeycirc \ar[r] \ar[d] & \mackeycirc \ar[r]  \ar[d] & 0 \\
   & & 0 & 0  & 
   }\]

\item For $d_3 \colon \underline{E}_3^{2, 8} \to \underline{E}_3^{5, 10}$, we have an exact sequence
\begin{align*}
\xymatrix{0 \ar[r] & \mackeybullet \ar[r] &\mackeyunderlinedodothat   \ar[r]^-{d_3} &\mackeyunderlinedbullethat  \ar[r] & \mackeybulletoverline  \ar[r] & 0 }
\end{align*}
determined by the differentials $d_3(\varpi) =\eta \varpi^2\Delta_1^{-1}$ and $d_3(\eta_0^2\Sigma_{2,0})= \mu_0\eta_0^5$. This gives the following remaining classes
\begin{align*}
 \mackeybullet&\subseteq  \underline{E}_4^{2, 8}  & & \text{generated by } 2\varpi  \in \underline{E}_4^{2, 8}(C_4/C_4) \\
 \mackeybulletoverline&\subseteq  \underline{E}_4^{5, 10}  & & \text{generated by } \eta_0(\eta_0 \eta_1)^2  \equiv  \eta_0^5    \in \underline{E}_4^{5,  10}(C_4/C_2).
\end{align*}
   There is a commutative diagram:
\[\xymatrix@=1pc{ & & 0  \ar[d]  & 0  \ar[d]  &  &\\
& 0 \ar[d] \ar[r] & \mackeyunderlinedbullethat \ar[r]^-{d_3} \ar[d] & \mackeyunderlinedbullethat \ar[d] \ar[r] & 0 \ar[d] & \\ 
0 \ar[r] & \mackeybullet \ar[r] \ar[d] &\mackeyunderlinedodothat  \ar[d] \ar[r]^-{d_3} &\mackeyunderlinedbullethat  \ar[r]   \ar[d]  & \mackeybulletoverline  \ar[r]  \ar[d]  & 0 \\
0 \ar[r] & \mackeybullet \ar[r]  \ar[d] &\mackeycirc \ar[r]^-{d_3} \ar[d] &\mackeybullethat  \ar[r] \ar[d] & \mackeybulletoverline  \ar[r]  \ar[d] & 0 \\ 
& 0 & 0 & 0 & 0 & }\]

\item For $d_3 \colon \underline{E}_3^{3, 10} \to \underline{E}_3^{6, 12}$, we have an exact sequence
\begin{align*}
\xymatrix{0 \ar[r] & \mackeyunderlinedbullethat   \ar[r]^-{d_3} &\mackeyunderlinedbullettilde \ar[r] & \mackeybulletoverline  \ar[r] & 0 }
\end{align*}
determined by the differentials $d_3(\eta \varpi) =  \eta^2\varpi^2\Delta_1^{-1}$ and $d_3(\eta_0^3 \Sigma_{2,0}) = \eta_0^6$. This gives the following remaining classes
\begin{align*}
 \mackeybulletoverline&\subseteq  \underline{E}_4^{6, 12}  & & \text{generated by }  (\eta_0\eta_1)^3\equiv  \eta_0^6    \in \underline{E}_4^{6,  12}(C_4/C_2).
\end{align*}
   There is a commutative diagram:
   \[\xymatrix@=1pc{ & 0 \ar[d] & 0 \ar[d] &  &  \\
    0 \ar[r] & \mackeyunderlinedbullethat \ar[r]^-{d_3} \ar[d] &  \mackeyunderlinedbullethat \ar[d] \ar[r] & 0 \ar[d] & \\
   0 \ar[r] & \mackeyunderlinedbullethat \ar[r]^-{d_3} \ar[d] & \mackeyunderlinedbullettilde \ar[r] \ar[d]  & \mackeybulletoverline \ar[r]  \ar[d] & 0 \\
   & 0 \ar[r] &  \mackeybulletoverline \ar[r] \ar[d] & \mackeybulletoverline \ar[r]  \ar[d] & 0 \\
   & & 0 & 0  & 
   }\]

\item For $d_3 \colon \underline{E}_3^{4, 12} \to \underline{E}_3^{7, 14}$, we have an exact sequence
\begin{align*}
\xymatrix{0 \ar[r] &  \mackeyunderlinedbullettilde \ar[r]^-{d_3}  & \mackeyunderlinedbullethat \ar[r] &  \mackeyblacktriangledown \ar[r] &  0}
\end{align*}
determined by $d_3(T_2\varpi^2\Delta_1^{-1}) = \eta^3\varpi^2 \Delta_1^{-1}$ and $d_3(\Sigma_{2,0}(\eta_0\eta_1)^2)=\mu_0\eta_0^3(\eta_0\eta_1)^2$. This gives the following remaining classes
   \begin{align*}
 \mackeyblacktriangledown &\subseteq \underline{E}_4^{7,  14} & & \text{generated by } \varsigma (\varpi^3 \Delta_1^{-2}) \in \underline{E}_4^{7,  14}(C_4/C_4) \\
& &  & \text{and by }  \eta_0^3(\eta_0\eta_1)^2\equiv \eta_0^7  \in  \underline{E}_4^{7,  14}(C_4/C_2).
   \end{align*}
   There is a commutative diagram:
   \[\xymatrix@=1pc{ & 0 \ar[d] & 0 \ar[d] &  &  \\
    0 \ar[r] & \mackeyunderlinedbullethat \ar[r]^-{d_3} \ar[d] &  \mackeyunderlinedbullethat \ar[d] \ar[r] & 0 \ar[d] & \\
   0 \ar[r] & \mackeyunderlinedbullettilde \ar[r]^-{d_3} \ar[d] & \mackeyunderlinedbullethat \ar[r] \ar[d]  & \mackeyblacktriangledown \ar[r]  \ar[d] & 0 \\
 0 \ar[r]  & \mackeybulletoverline \ar[r] \ar[d] &  \mackeybullethat \ar[r] \ar[d] & \mackeyblacktriangledown \ar[r]  \ar[d] & 0 \\
   &0  & 0 & 0  & 
   }\]
\end{enumerate}

We give a similar description of the $\underline{E}_4^{*,1-\sigma+*}$. Again, using the $\W[\![\mu ]\!][\varpi^2, \Delta_1^{\pm1}]$-module structure, it suffices to describe the following five differentials to determine all the $d_3$ differentials. The $\underline{E}_4$-page is illustrated in \Cref{fig:E5pagesigma}. See \cite[Section 5]{HHRC4} and  \Cref{fig:exactsequencesds} for various relevant exact sequences of Mackey functors.

\begin{enumerate}
\item For $d_3 \colon \underline{E}_3^{0,5-\sigma} \to  \underline{E}_3^{3,7-\sigma}$, we have an exact sequence
\begin{align*}
\xymatrix{0 \ar[r] &  \mackeyunderlinedhalftopboxhat \ar[r] &  \mackeyunderlinedboxdothat \ar[r]^-{d_3}  & \mackeyunderlinedbullethat \ar[r] &  \mackeybulletoverline \ar[r] &  0}
\end{align*}
determined by $d_3(\Delta_1\mathfrak{p})=\eta \varpi \mathfrak{p} $ and $d_3(\Sigma_{2,0}u_{\sigma})=\mu_0\eta_0^3u_{\sigma}$. This gives the following remaining classes, where $\bar{u}_{\sigma}=\res_1^2(u_{\sigma})$:
   \begin{align*}
   \mackeyunderlinedhalftopboxhat &\subseteq \underline{E}_4^{0, 5-\sigma} & &\text{generated by } 2\Delta_1\mathfrak{p}  \in \underline{E}_2^{0, 5-\sigma}(C_4/C_4) \\
& &  &    \text{and by }  2\delta_1u_{\sigma} , 2\Sigma_{2,0}u_{\sigma}  \in \underline{E}_4^{0,  5-\sigma}(C_4/C_2) \\
& &  &    \text{and by }  r_{1,0}r_{1,1} \bar{u}_{\sigma}, r_{1,0}^2\bar{u}_{\sigma} \in \underline{E}_4^{0,  5-\sigma}(C_4/\{e\}) \\
 \mackeybulletoverline &\subseteq \underline{E}_4^{3,  7-\sigma} & & \text{generated by }  \eta_0 \delta_1  (\eta_0^2 \Sigma_{2,0}^{-1}) u_{\sigma}  \equiv \eta_0^3 u_{\sigma}  \in \underline{E}_4^{3,  7-\sigma}(C_4/C_2).
   \end{align*}
    The following commutative diagram has exact rows and columns:
   \[\xymatrix@=1pc{  & 0 \ar[d] & 0\ar[d] & 0 \ar[d] &  & \\
   0 \ar[r] & \mackeyunderlinedboxleftslashhat \ar[r] \ar[d] & \mackeyunderlinedboxhat \ar[r]^-{d_3} \ar[d] & \mackeyunderlinedbullethat \ar[r] \ar[d] & 0 \ar[d] & \\ 
   0 \ar[r] &  \mackeyunderlinedhalftopboxhat \ar[d] \ar[r] &  \mackeyunderlinedboxdothat \ar[d] \ar[r]^-{d_3}  & \mackeyunderlinedbullethat \ar[r] \ar[d] &  \mackeybulletoverline \ar[d] \ar[r] &  0 \\ 
      0 \ar[r] &  \mackeyhalftopbox \ar[d] \ar[r] &  \mackeybox\ar[d] \ar[r]^-{d_3}  & \mackeybullethat \ar[d] \ar[r] &  \mackeybulletoverline \ar[d] \ar[r] &  0 \\ 
      & 0 & 0 & 0 & 0 &  }\]

\item For $d_3 \colon \underline{E}_3^{1, 7-\sigma} \to \underline{E}_3^{4, 9-\sigma}$, we have an exact sequence
\begin{align*}
\xymatrix{0 \ar[r] & \mackeyunderlinedbullethat \ar[r]^-{d_3} & \mackeyunderlinedbullettilde \ar[r] & \mackeybulletoverline \ar[r] & 0 }
\end{align*}
determined by the differentials $d_3(\eta \Delta_1\mathfrak{p})=\eta^2 \varpi \mathfrak{p}$ and $d_3(\eta_0\Sigma_{2,0}u_{\sigma}) = \mu_0\eta_0^4u_{\sigma}$.
This gives the following remaining classes
\begin{align*}
 \mackeybulletoverline&\subseteq \underline{E}_4^{4,  9-\sigma} & & \text{generated by } (\eta_0 \eta_1)^2u_{\sigma}  \equiv \eta_0^4u_{\sigma}    \in \underline{E}_4^{4,  9-\sigma}(C_4/C_2).
\end{align*}

\item For $d_3 \colon \underline{E}_3^{2, 9-\sigma} \to \underline{E}_3^{5, 11-\sigma}$, we have an exact sequence
\begin{align*}
\xymatrix{0 \ar[r] &  \mackeyunderlinedbullettilde \ar[r]^-{d_3}  & \mackeyunderlinedbullethat \ar[r] &  \mackeyblacktriangledown \ar[r] &  0}
\end{align*}
determined by the differentials $d_3(\eta^2 \Delta_1\mathfrak{p})=\eta^3 \varpi \mathfrak{p}$ and $d_3(\eta_0^2\Sigma_{2,0}u_{\sigma})= \mu_0\eta_0^5u_{\sigma}$.
\begin{align*}
 \mackeyblacktriangledown&\subseteq  \underline{E}_4^{5, 11-\sigma}  & & \text{generated by } \varsigma \varpi^2\Delta_1^{-1}\mathfrak{p}   \in \underline{E}_4^{5,  11-\sigma}(C_4/C_4)\\
& & & \text{and by } \eta_0(\eta_0 \eta_1)^2u_{\sigma}  \equiv  \eta_0^5u_{\sigma}    \in \underline{E}_4^{5,  11-\sigma}(C_4/C_2).
\end{align*}

\item For $d_3 \colon \underline{E}_3^{3, 11-\sigma} \to \underline{E}_3^{6, 13-\sigma}$, we have an exact sequence
\begin{align*}
\xymatrix{0 \ar[r] & \mackeyunderlinedbullethat \ar[r]^-{d_3} & \mackeyunderlinedodothat \ar[r] & \mackeycirc \ar[r] & 0 }
\end{align*}
determined by the differentials $ d_3(\varpi \varsigma \mathfrak{p}) =\eta \varsigma \varpi^2 \Delta_1^{-1}  \mathfrak{p}$ and $d_3(\eta_0^3 \Sigma_{2,0}u_{\sigma}) = \eta_0^6u_{\sigma}$. This gives the following remaining classes
\begin{align*}
 \mackeycirc&\subseteq  \underline{E}_4^{6, 13-\sigma}  & & \text{generated by } \varpi^3\Delta_1^{-1}\mathfrak{p}   \in \underline{E}_4^{6,  13-\sigma}(C_4/C_4) \\
& & & \text{and by }  (\eta_0\eta_1)^3 u_{\sigma}\equiv  \eta_0^6u_{\sigma}    \in \underline{E}_4^{6,  13-\sigma}(C_4/C_2).
\end{align*}

\item For $d_3 \colon \underline{E}_3^{4, 13-\sigma} \to \underline{E}_3^{7, 15-\sigma}$, we have an exact sequence
\begin{align*}
\xymatrix{0 \ar[r] & \mackeybullet \ar[r] &\mackeyunderlinedodothat   \ar[r]^-{d_3} &\mackeyunderlinedbullethat  \ar[r] & \mackeybulletoverline  \ar[r] & 0 }\end{align*}
determined by $ d_3(\varpi^2\mathfrak{p})=\eta \varpi^3 \Delta^{-1} \mathfrak{p}$ and $d_3(\Sigma_{2,0}(\eta_0\eta_1)^2u_{\sigma})=\mu_0\eta_0^3(\eta_0\eta_1)^2u_{\sigma}$. This gives the following remaining classes
   \begin{align*}
   \mackeybullet &\subseteq \underline{E}_3^{4, 13-\sigma} & & \text{generated by } 2\varpi^2 \mathfrak{p} \in \underline{E}_3^{4, 13-\sigma}(C_4/C_4) \\
 \mackeybulletoverline &\subseteq \underline{E}_4^{7,  15-\sigma} & & \text{generated by } \eta_0^3(\eta_0\eta_1)^2 u_{\sigma} \equiv \eta_0^7u_{\sigma}  \in  \underline{E}_4^{7,  15-\sigma}(C_4/C_2).
   \end{align*}

\end{enumerate}

There are several exotic restrictions and transfers that follow from the $d_3$ differentials. Following \cite[Figure 10]{HHRC4} we indicate exotic transfers by solid blue lines and exotic restrictions by dashed green lines.

 \subsubsection{Higher differentials and the $E_{\infty}$-page}
The higher differentials for the homotopy fixed point spectral sequences $H^*(C_4,E_*)$ and $H^*(C_2, E_*)$ are listed below. From the $\underline{E}_4$-term onwards, the homotopy fixed point spectral sequence and the slice spectral sequence are isomorphic in the range $s>2$ and $t-s >s$. Because of the periodicity of the HFPSS, higher differentials are easily deduced from slice differentials. We refer the reader to \cite[Section 14]{HHRC4} for more details. The computation is also illustrated in \Cref{fig:E5page} and \Cref{fig:E5pagesigma}. The exact sequences of Mackey functors involving the differentials are depicted in \Cref{fig:exactsequencesds}.

\begin{rem}
The following classes are permanent cycles:
\begin{align*}
\bar{\kappa} &= \varpi^2\Delta_1 & \epsilon &= \varpi^4\Delta_1^{-2} & \kappa &= 2\varpi\Delta_1
\end{align*}
Therefore, all differentials are linear with respect to multiplication by these classes.
\end{rem}

\begin{rem}
One key difference between our computation and that of \cite{HHRC4} is the behavior of $u_{2\sigma}$. In the $\underline{E}_2(C_4/C_4)$-term of the slice spectral sequence for ${\spi}_{\star}k_{[2]}$, there is a relation $a_{\sigma}^3 u_{\lambda}=0$. However, in the $\underline{E}_2$-term of the homotopy fixed point spectral sequence for $\spi_{\star}E_2^{h}$, the image of $u_{V}$ for every representation $V$ becomes a unit. Therefore, $a_{\sigma}^3=0$ in ${H}^{3}(C_4, \pi_{\star}E)$. In fact, $a_{\sigma}^3u_{\lambda}^3 \dfrak^3=\nu^3$, which is zero on the $\underline{E}_2(C_4/C_4)$-page. This implies that the target of the slice differential $d_5(u_{2\sigma})=a_{\sigma}^3 a_{\lambda} \dfrak $ is trivial. So, in the spectral sequence 
\[{H}^{*}(C_4, \pi_{\star}E) \Rightarrow {\pi}_{\star-*}E^{hC_4}\]
we have
\[d_{5}(u_{2\sigma})=0.\]
We will show in \Cref{prop:pc} that $u_{2\sigma}u_{\lambda}^4$ is a permanent cycle. This, together with $d_7(u_{\lambda}^4)$ from \cite[Theorem 11.13]{HHRC4} implies that $u_{2\sigma}$  instead supports a $d_7$ differential.
\end{rem}

 \begin{prop}\label{thm:diffsaddd5}
The $d_5$ differentials are $\mu $, $\eta$, $\nu$, $\bar{\kappa}$ and $\Delta_1^{2}$-linear.
The differentials $d_5 \colon \underline{E}_5^{s,t}(C_4/C_4) \to \underline{E}_5^{s+5,t+4}(C_4/C_4) $ are determined by
\begin{align*}
d_5(\Delta_1)&=\nu\Delta_1^{-1} \varpi^2, &    d_5(\nu \varpi)&= 2\Delta_1^{-2}\varpi^4
\end{align*}
The differentials $d_5 \colon \underline{E}_5^{s,1-\sigma +t}(C_4/C_4) \to \underline{E}_5^{s+5,1-\sigma+t+4}(C_4/C_4) $
are determined by
\begin{align*}
d_5( \Delta_1  \nu\mathfrak{p}) &=2  \Delta_1^{-1}  \varpi^3\mathfrak{p}, & d_5(\Delta_1 \varpi  \mathfrak{p}) &= \nu \Delta_1^{-1} \varpi^3\mathfrak{p} .
\end{align*}
\end{prop}

\begin{figure}[H]
\begin{center}
\includegraphics[page=1, width=\textwidth]{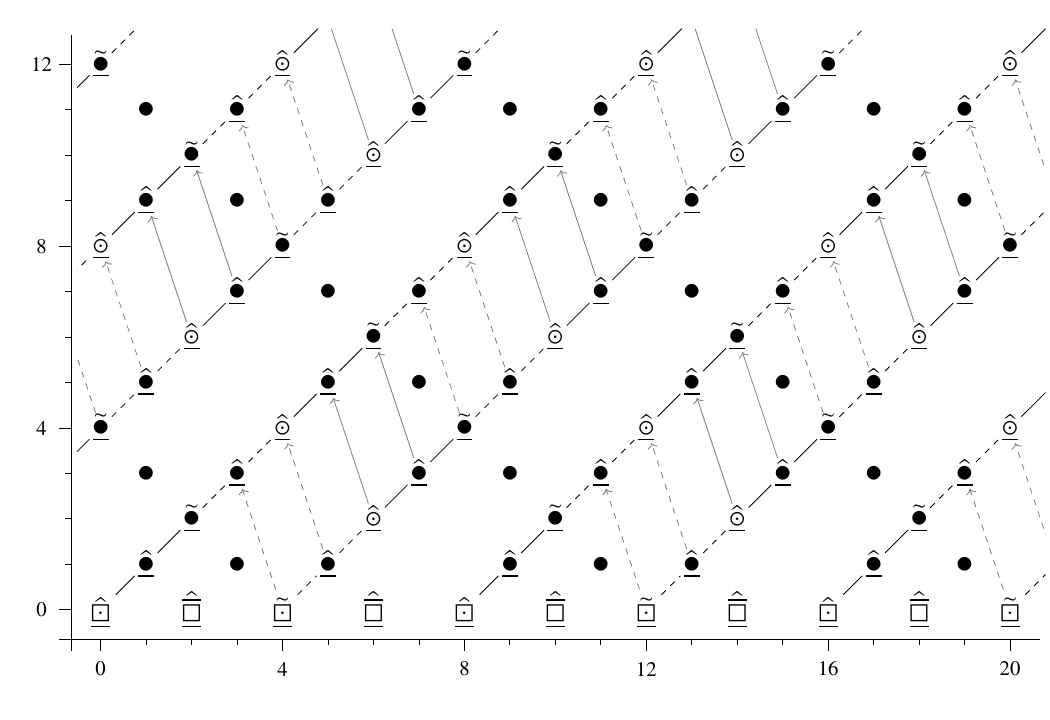}
  \includegraphics[page=1, width=\textwidth]{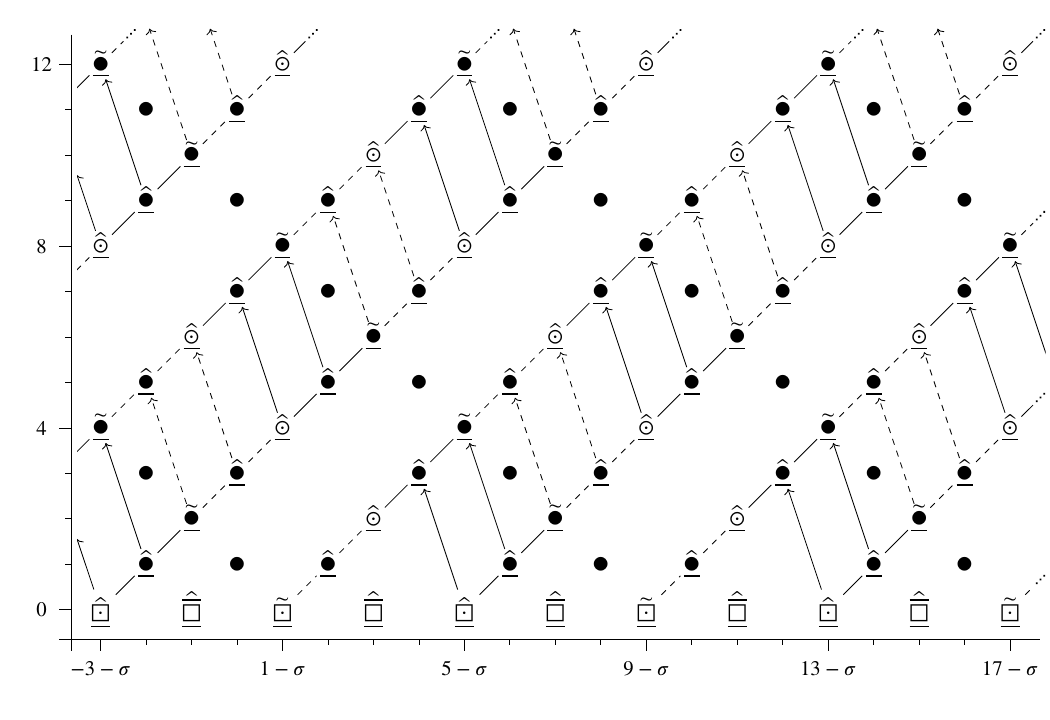}
 \end{center}
 \caption{The $\underline{E}_2^{*,*}$ (top) and $\underline{E}_2^{*,1-\sigma+*}$ (bottom) page of the homotopy fixed point spectral sequence given by $\underline{H}^*(C_4, E_\star)$. Lines of slope $(1,1)$ indicate $\eta$ multiplication on the $E_3^{*,\star}(C_4/C_4)$-term. They are dashed if the generator of the target is not a multiple of $\eta$. The $d_3$-differential patters are also drawn. Again, a $d_3$ is dashed if the target of the differential in $E_3^{*,\star}(C_4/C_4)$ is not the generator.}
 \label{figure:additive-E2} \label{figure:additive-E2-sigma}
 \end{figure}
 
\begin{figure}[H]
\centering
\begin{align*}
&\xymatrix@R=0.5pc{ 0 \ar[r] & \mackeyunderlinedbullethat \ar[r]^-{d_3} & \mackeyunderlinedodothat \ar[r] & \mackeycirc \ar[r] & 0 \\
0 \ar[r] & \mackeyunderlinedbullethat   \ar[r]^-{d_3} &\mackeyunderlinedbullettilde \ar[r] & \mackeybulletoverline  \ar[r] & 0\\
0 \ar[r] &  \mackeyunderlinedbullettilde \ar[r]^-{d_3}  & \mackeyunderlinedbullethat \ar[r] &  \mackeyblacktriangledown \ar[r] &  0  } &
 \xymatrix@R=0.5pc{ 
0 \ar[r] &  \mackeyunderlinedboxleftslashtilde \ar[r] &  \mackeyunderlinedboxdottilde \ar[r]^-{d_3}  & \mackeyunderlinedbullethat \ar[r] &  \mackeyblacktriangledown \ar[r] &  0\\
0 \ar[r] &  \mackeyunderlinedhalftopboxhat \ar[r] &  \mackeyunderlinedboxdothat \ar[r]^-{d_3}  & \mackeyunderlinedbullethat \ar[r] &  \mackeybulletoverline \ar[r] &  0 \\
0 \ar[r] & \mackeybullet \ar[r] &\mackeyunderlinedodothat   \ar[r]^-{d_3} &\mackeyunderlinedbullethat  \ar[r] & \mackeybulletoverline  \ar[r] & 0
}
\end{align*}
\begin{align*}
\xymatrix@R=0.5pc{ 
0 \ar[r] &\mackeyunderlinedboxboxhat  \ar[r] & \mackeyunderlinedboxdothat \ar[r]^-{d_5}  &\mackeybullet  \ar[r]  & 0 \\
0 \ar[r] & \mackeybullet \ar[r]^-{d_5} & \mackeycirc \ar[r] &\mackeyblacktriangle  \ar[r] & 0 \\
0 \ar[r] &\mackeyunderlinedcirccirchat  \ar[r] & \mackeyunderlinedodothat \ar[r]^-{d_5} & \mackeybullet \ar[r] & 0 \\
0 \ar[r] & \mackeyblacktriangledown  \ar[r] & \mackeycirc \ar[r]^-{d_5} &\mackeybullet  \ar[r]  & 0 
}
\end{align*}
\begin{align*}
\xymatrix@R=0.5pc{ 
0 \ar[r] & \mackeyunderlinedboxboxtilde \ar[r] &  \mackeyunderlinedboxdottilde \ar[r]^-{d_7} & \mackeybulletoverline \ar[r] & 0 \\
0 \ar[r] &\mackeybulletoverline \ar[r]^-{d_7} &\mackeyblacktriangle \ar[r] &\mackeybullet \ar[r] & 0 \\
0 \ar[r] & \mackeyunderlinedboxblackboxhat \ar[r] &\mackeyunderlinedboxboxhat  \ar[r]^-{d_7} &\mackeyblacktriangledown  \ar[r] & 0  \\
0 \ar[r] & \mackeyunderlinedblacktrianglehat \ar[r] & \mackeyunderlinedbullethat \ar[r]^-{d_7} & \mackeybulletoverline \ar[r] & 0 \\
0 \ar[r] &\mackeyunderlinedbullethat \ar[r] & \mackeyunderlinedbullettilde \ar[r]^-{d_7} &\mackeybulletoverline \ar[r] & 0 \\
0 \ar[r] &\mackeyunderlinedbullethat \ar[r] & \mackeyunderlinedcirccirchat \ar[r]^-{d_7} & \mackeyblacktriangledown\ar[r] & 0 \\
0 \ar[r] &\mackeybullet \ar[r] &\mackeyblacktriangledown \ar[r]^-{d_7} &\mackeybulletoverline \ar[r] & 0 
 }
\end{align*}
\begin{align*}
\xymatrix@R=0.5pc{ 
0 \ar[r] &\mackeyunderlinedblacktrianglehat  \ar[r] &\mackeyunderlinedbullethat \ar[r]^-{d_7} &\mackeyblacktriangle \ar[r] & \mackeybullet \ar[r] & 0 \\
0 \ar[r] &\mackeybulletoverline  \ar[r]^-{d_7} &\mackeyblacktriangle \ar[r]^-{d_7}  & \mackeyblacktriangledown \ar[r]^-{d_7} &\mackeybulletoverline \ar[r] & 0 \\
0 \ar[r] & \mackeyunderlinedboxboxhat \ar[r] &\mackeyunderlinedboxdothat \ar[r]^-{d_7} &\mackeyblacktriangledown \ar[r]^-{d_7} & \mackeybulletoverline \ar[r] & 0 \\
0 \ar[r] & \mackeyblacktriangledown \ar[r] &\mackeycirc \ar[r]^-{d_7} & \mackeyblacktriangledown \ar[r]^-{d_7} & \mackeybulletoverline \ar[r] & 0 }
\end{align*}
\begin{align*}
\xymatrix@R=0.5pc{ 
0 \ar[r] & \mackeyunderlinedblacktrianglehat \ar[r] & \mackeyunderlinedbullethat \ar[r]^-{d_7}& \mackeyblacktriangle \ar[r]^-{d_7} & \mackeyblacktriangledown \ar[r] & \mackeybulletoverline \ar[r] & 0 }
\end{align*}
\begin{align*}
\xymatrix@R=0.5pc{ 
0 \ar[r] &  \mackeyunderlinedbullettilde \ar[r] & \mackeyunderlinedblacktrianglehat \ar[r]^-{d_{11}} & \mackeybullet \ar[r] & 0 \\
0 \ar[r] & \mackeyunderlinedhalfbottomboxhat  \ar[r] & \mackeyunderlinedhalftopboxhat \ar[r]^-{d_{13}} & \mackeybullet \ar[r] & 0 }
\end{align*}
\caption{The different patterns of $d_3$, $d_5$, $d_7$, $d_{11}$ and $d_{13}$-differentials.}
\label{fig:exactsequencesds}
\end{figure}

\begin{proof}
Again, these come from differentials listed in \cite[Table 3]{HHRC4}. They are deduced using \Cref{rem:HHRtranslation} from 
\begin{align*}
d_5(u_{2\sigma}u_{\lambda}^2) &= u_{2\sigma}\bar{\nu} a_{\lambda}^2\dfrak ,  & 
d_5(u_{\lambda}^2) &=\bar{\nu} a_{\lambda}^2\dfrak. 
\end{align*}
Here, $\bar{\nu} = a_{\sigma} u_{\lambda} \in \underline{E}_2^{1, 3-\sigma-\lambda}(C_4/C_4)$ is a permanent cycle. In fact, $\bar{\nu}= \nu \dfrak^{-1}$. Here, we also use that $\nu^2 = 2\varpi$, which is related to the ``gold relation'' $2u_{2\sigma}a_{\lambda} = a_{\sigma}^2u_{\lambda}$.
\end{proof}

We take a moment here to prove the following result, which settles \Cref{prop:orientableus}.
   \begin{prop}\label{prop:pc}
The classes $u_{\lambda}^8, u_{2\sigma}^2$, $u_{2\sigma}u_{\lambda}^4$ and $\bar{u}_{4\lambda} =\res_2^4(u_{4\lambda})$ are permanent cycles. In particular, the class $u_{8\sigma_2}$ is a permanent cycle  in the spectral sequence for $E^{hC_2}$.
   \end{prop}
   \begin{proof}
That $u_{\lambda}^8$, $u_{2\sigma}^2$ and $\bar{u}_{4\lambda}$ are permanent cycles is part of \cite[Theorem 11.13]{HHRC4}. The claim about $u_{8\sigma_2}$ follows from the fact that the restriction of $\lambda$ to the group $C_2$ is $2\sigma_2$.

We give a quick argument to justify that $u_{2\sigma}u_{\lambda}^4$ is a permanent cycle. Multiplication by $u_{2\sigma}u_{\lambda}^4$ induces an isomorphism 
\begin{equation}
\label{eq:isou2sigmau4lambda} \underline{E}_2^{*,*}(C_4/C_4) \to \underline{E}_2^{*, 10-2\sigma-4\lambda+*}(C_4/C_4).\end{equation}
that sends $1$ to $u_{2\sigma}u_{\lambda}^4$. Using this isomorphism, it is visible that if $u_{2\sigma}u_{\lambda}^4$ is a $d_7$-cycle, then it is a permanent cycle by sparseness.  
Both $u_{2\sigma}$ and $u_{\lambda}^4$ are $d_5$-cycles. 
The only possible non-trivial $d_7$ differential is $d_7(u_{2\sigma}u_{\lambda}^4)= \eta'u_{2\sigma}u_{\lambda}^2 a_{\lambda}^3 \dfrak$. Note that $\bar{\kappa}$ is detected by
   \[\bar{\kappa}= \varpi^2 \Delta_1 = a_{\lambda}^2u_{2\sigma}^2 \dfrak^6( u_{2\sigma}u_{\lambda}^4)\]
So, the differential would imply that 
\[d_7(\varpi^2 \Delta_1) = \eta' a_{\lambda}^5 u_{\lambda}^2 u_{2\sigma}^3 \dfrak^7 = \varsigma \varpi^{5}\Delta_1^{-2}  \]
which is non-trivial on the $E_7$-term, a contradiction. Therefore, $u_{2\sigma}u_{\lambda}^4$ is a $d_7$-cycle. 
\end{proof}
   
   \begin{cor}\label{cor:usefulpc}
The class $\varpi \Delta_1^2 \mathfrak{p}$ is a permanent cycle detecting a non-zero class 
\[\varpi \Delta_1^2 \mathfrak{p} \in {\spi}_{19-\sigma}E(C_4/C_4).\]
   \end{cor}
   \begin{proof}
   This follows from \Cref{prop:pc} since 
   \[\varpi \Delta_1^2 \mathfrak{p}= a_{\lambda} u_{2 \sigma}^2 \dfrak^5 (u_{2\sigma} u_{\lambda}^4)\]
is a product of permanent cycles.
   \end{proof}

 \begin{prop}\label{thm:diffsadd}
The $d_7$ differentials are $\mu $, $\eta$, $\nu$, $\bar{\kappa}$, $\epsilon$ and $\Delta_1^{4}$-linear. The $d_7$-differentials $d_7 \colon \underline{E}_7^{s,t}(C_4/C_4) \to \underline{E}_7^{s+7,t+6}(C_4/C_4) $ are determined by
\begin{align*}
&d_7(2\Delta_1)= \varsigma \varpi^3 \Delta_1^{-2}  & & d_7(\Delta_1^2)= \varsigma \varpi^3 \Delta_1^{-1} & & d_7(\mu\Delta_1  )=0 .
\end{align*}
Further, $\mu\Delta_1 $ is a permanent cycle.

The differentials $d_7 \colon \underline{E}_7^{s,1-\sigma +t}(C_4/C_4) \to \underline{E}_7^{s+7,1-\sigma + t+6}(C_4/C_4) $ are determined by
\begin{align*}
&d_7(2\Delta_1^3 \varpi\mathfrak{p})=  \varsigma \varpi^{4}\mathfrak{p}& & d_7( \varpi\mathfrak{p})=  \varsigma  \varpi^{4}\mathfrak{p}  \Delta_1^{-3}  & & d_7( \mu \varpi \mathfrak{p}\Delta_1^3)=0  .
\end{align*}

The differentials $d_7 \colon \underline{E}_7^{s,t}(C_4/C_2) \to \underline{E}_7^{s+7,t+6}(C_4/C_2) $
are $\eta_0$, $\mu_0$ and $\Sigma_{2,0}^4$-linear. They are determined by
\[d_{7}(\Sigma_{2,0}^2) =\eta_0^7.\]
The spectral sequence $\underline{E}_7^{*,*}(C_4/C_2)$ collapses at $E_7$.
\end{prop}
 \begin{proof}
We use the proof of \cite[Theorem14.3]{HHRC4}. From that discussion, one deduces that
\[d_7(2\Delta_1)  = a_{\lambda}^3 u_{2\sigma}\eta' \dfrak^3 \ \ \ \text{and} \ \ \ d_7(\Delta_1^2)  = \Delta_1d_7(2\Delta_1). \]
 The first two differentials for $ \underline{E}_*^{*,*}(C_4/C_4)$ then follow using \Cref{rem:HHRtranslation}. 
 
The $d_7$ for  $ \underline{E}_*^{*,*}(C_4/C_2)$ is obtained by multiplying the differential from \cite{HHRC4}, \[d_7([\bar{u}_{\lambda}^2])= a_{\sigma_2}^7\bar{r}_{1,0}^3\] with the permanent cycle 
$\bar{r}_{1,0}^4 $.
There is a vanishing line at $s=8$ on the $E_8$-page, and so the spectral sequence collapses.

For $d_7(\Delta_1 \mu )=0 $, note that $\mu \Delta_1 = \tr_2^4(\mu_0 \delta_1^2)$. Since 
 \[d_7(\mu_0 \delta_1^2 ) = d_7(\mu_0(1-\mu_0)^2\Sigma_{2,0}^2) =  \mu_0(1+\mu_0^2)\eta_0^7\]
 and the latter is zero in $\underline{E}_7^{7,14}(C_4/C_2)$, the claim follows.
 
 Now, note that $\mu_0 \delta_1^2$ is a permanent cycle, hence so is $\mu \Delta_1 = \tr_2^4(\mu_0 \delta_1^2)$. Finally, the differentials in $\underline{E}_7^{*,1-\sigma+*}(C_4/C_4)$ then follow using \Cref{cor:usefulpc} and the fact that the $d_7$ differentials are $\Delta_1^4$-linear.
 \end{proof}

 \begin{prop}
 The $d_{11}$ and $d_{13}$ differentials are $\mu $, $\eta$, $\nu$, $\bar{\kappa}$, $\epsilon$ and $\Delta_1^{4}$-linear. The $d_{11}$-differentials $d_{11} \colon \underline{E}_{11}^{s,t}(C_4/C_4) \to \underline{E}_{11}^{s+11,t+10}(C_4/C_4) $ are determined by
 \begin{align*}
 d_{11}(\varsigma \varpi) =2 \Delta_1^{-4}\varpi^7 = \nu^2 \varpi^6\Delta_1^{-4}.
 \end{align*}
  The $d_{11}$-differentials $d_{11} \colon \underline{E}_{11}^{s,1-\sigma+t}(C_4/C_4) \to \underline{E}_{11}^{s+11,1-\sigma + t+10}(C_4/C_4) $ are determined by
 \begin{align*}
 d_{11}(\varsigma \varpi^2 \Delta_1^2 \mathfrak{p}) = \nu^2 \varpi^7\Delta_1^{-2}  \mathfrak{p}.
 \end{align*}
 The $d_{13}$-differentials $d_{13} \colon \underline{E}_{13}^{s,t}(C_4/C_4) \to \underline{E}_{13}^{s+13,t+12}(C_4/C_4) $ are determined by
 \begin{align*}
 & d_{13}(\Delta_1 \nu \varpi)= \Delta_1^{-4}\varpi^8 & & d_{13}(\Delta_1^3 \nu^2) = \Delta_1^{-2}\nu \varpi^7 
 \end{align*}
  The $d_{13}$-differentials $d_{13} \colon \underline{E}_{13}^{s,1-\sigma+t}(C_4/C_4) \to \underline{E}_{13}^{s+13,1-\sigma + t+12}(C_4/C_4) $ are determined by
 \begin{align*}
 & d_{13}(\Delta_1^3 \nu \varpi^2 \mathfrak{p})= \Delta_1^{-2}\varpi^9  \mathfrak{p}& & d_{13}(\Delta_1^5 \nu^2 \varpi \mathfrak{p} ) = \nu \varpi^8  \mathfrak{p}.
 \end{align*}
 \end{prop}
 \begin{proof}
 The $d_{11}$ in $\underline{E}_{*}^{*,*}(C_4/C_4)$ follows from \cite[Theorem 14.2 (iv)]{HHRC4}. The $d_{13}$s come from Theorem \cite[Theorem 14.4]{HHRC4}. Finally, the differentials in $\underline{E}_{*}^{*,1-\sigma +*}(C_4/C_4)$ then follow from those in the integer graded spectral sequence by multiplying by the permanent cycle $\Delta_1^2 \varpi \mathfrak{p}$ of \Cref{cor:usefulpc}.
 \end{proof}

\begin{rem}
There is a typo in the statement of \cite[Proposition 2.3.8]{BO}. It should read $d_{11}(\gamma \xi) = 2\delta^{-4} \xi^7$.
\end{rem}
 
The $\underline{E}_{\infty}$ term is illustrated in \Cref{fig:E5page}. The exact sequences of Mackey functors required to compute it  are listed in \Cref{fig:exactsequencesds}. Finally, \Cref{table:Mackey}, \Cref{table:Mackey-two} and \Cref{table:Mackey-three} contain the definitions of the Mackey functors required to read   \Cref{fig:E5page}.

\begin{figure}[H]
\begin{center}
 \includegraphics[angle=90, page=1, width=\textwidth]{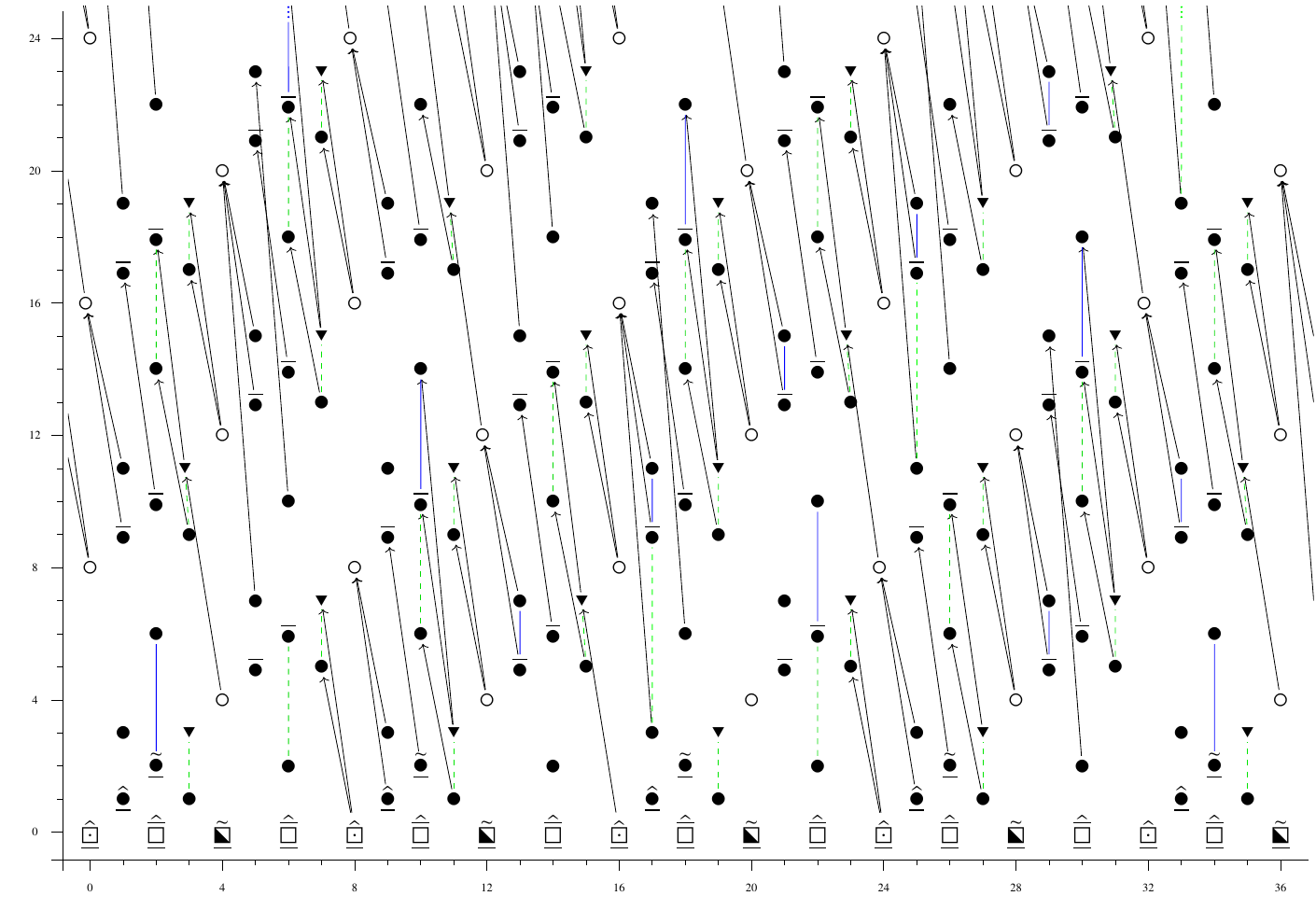}
 \end{center}
 \caption{The $\underline{E}_4^{*,*}$ page onwards of the HFPSS with $d_r$ for $r\geq 5.$}
 \label{fig:E5page}
 \end{figure}
 
 \begin{figure}[H]
\begin{center}
 \includegraphics[angle=90, page=2, width=\textwidth]{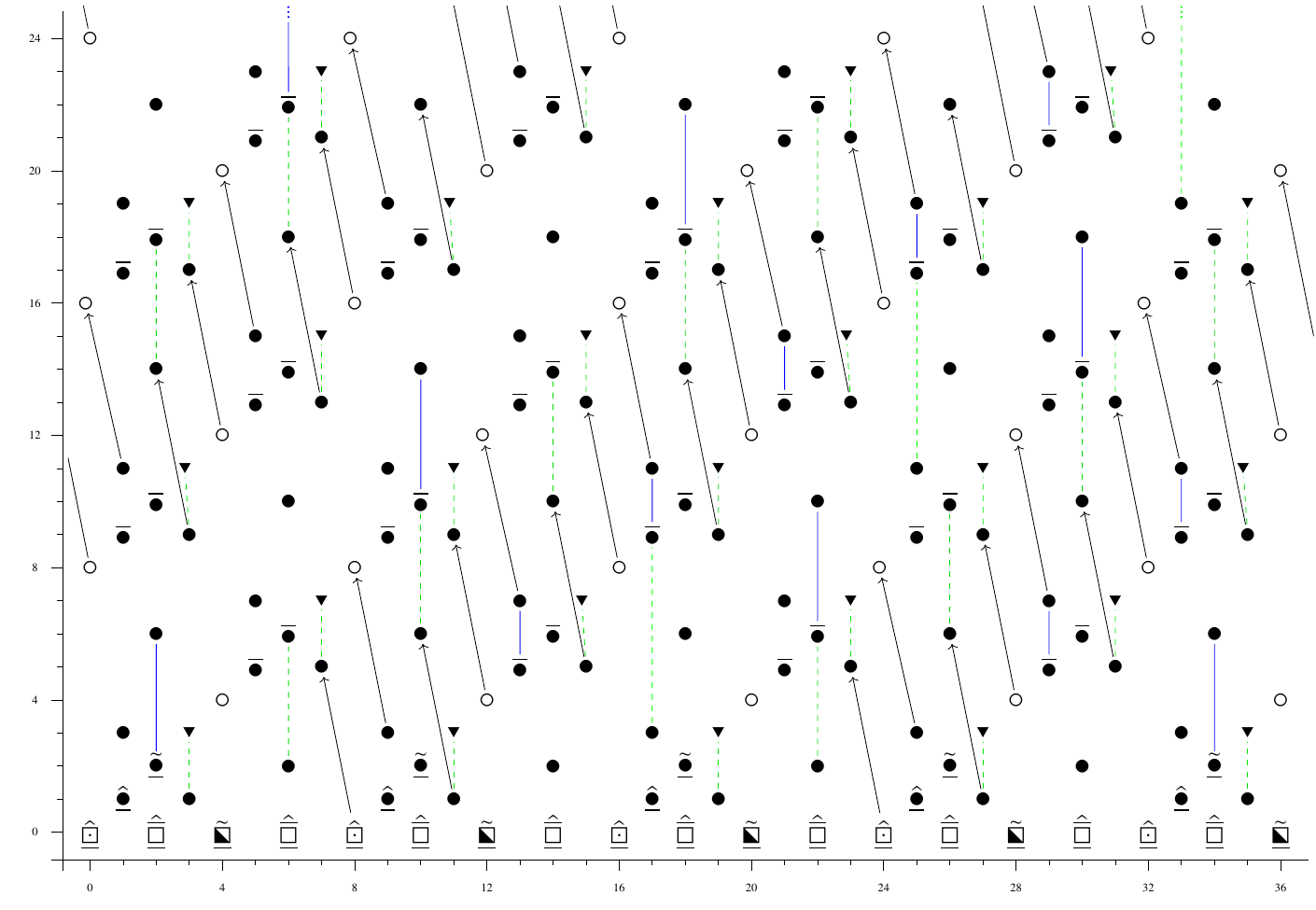}
 \end{center}
 \caption{The $\underline{E}_5^{*,*}$ page of the HFPSS with $d_5$ differentials.}
 \label{fig:E5page2}
 \end{figure}
 
  \begin{figure}[H]
\begin{center}
 \includegraphics[angle=90, page=3, width=\textwidth]{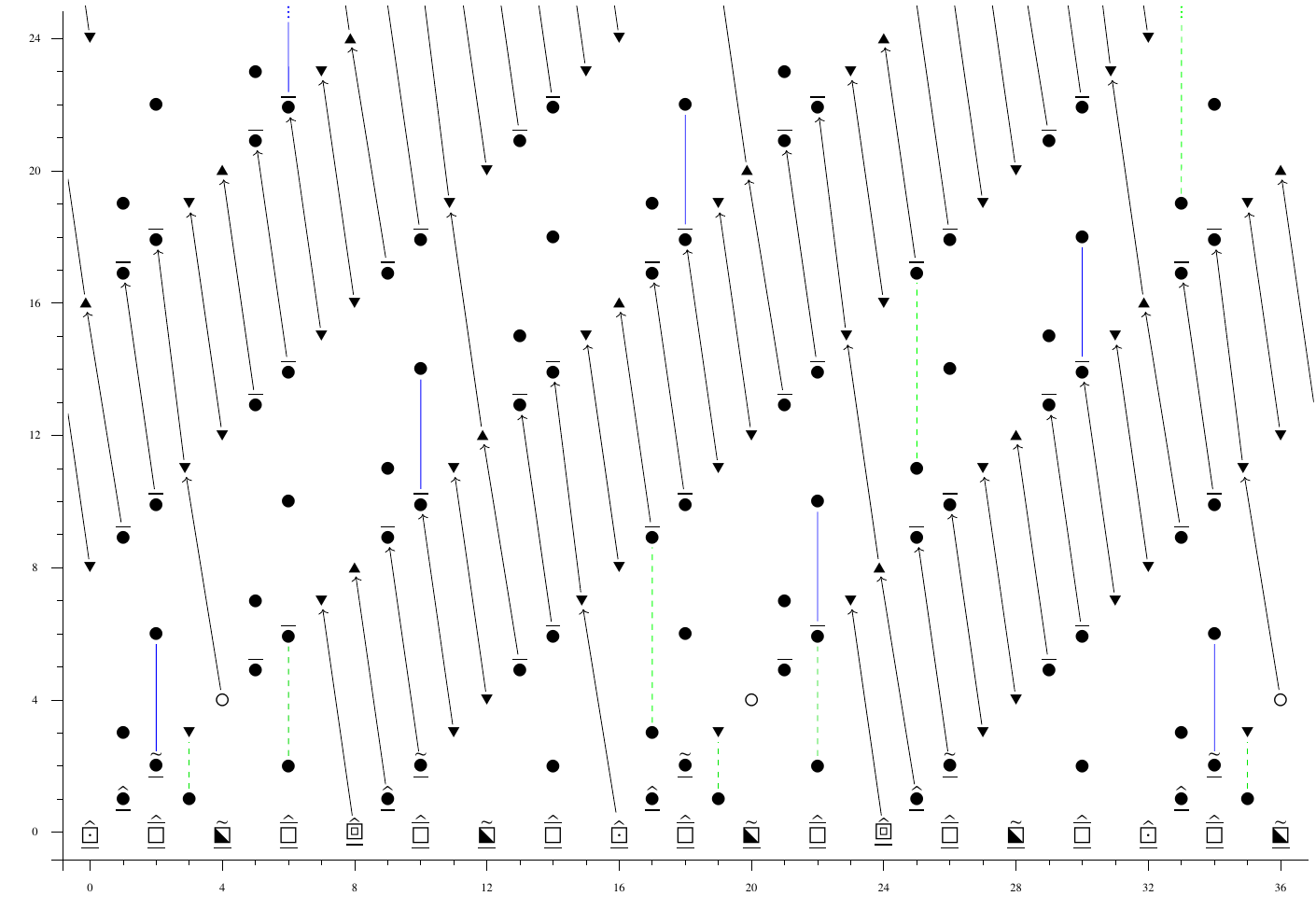}
 \end{center}
 \caption{The $\underline{E}_7^{*,*}$ page of the HFPSS with $d_7$ differentials.}
 \label{fig:E5page3}
 \end{figure}
 
   \begin{figure}[H]
\begin{center}
 \includegraphics[angle=90, page=4, width=\textwidth]{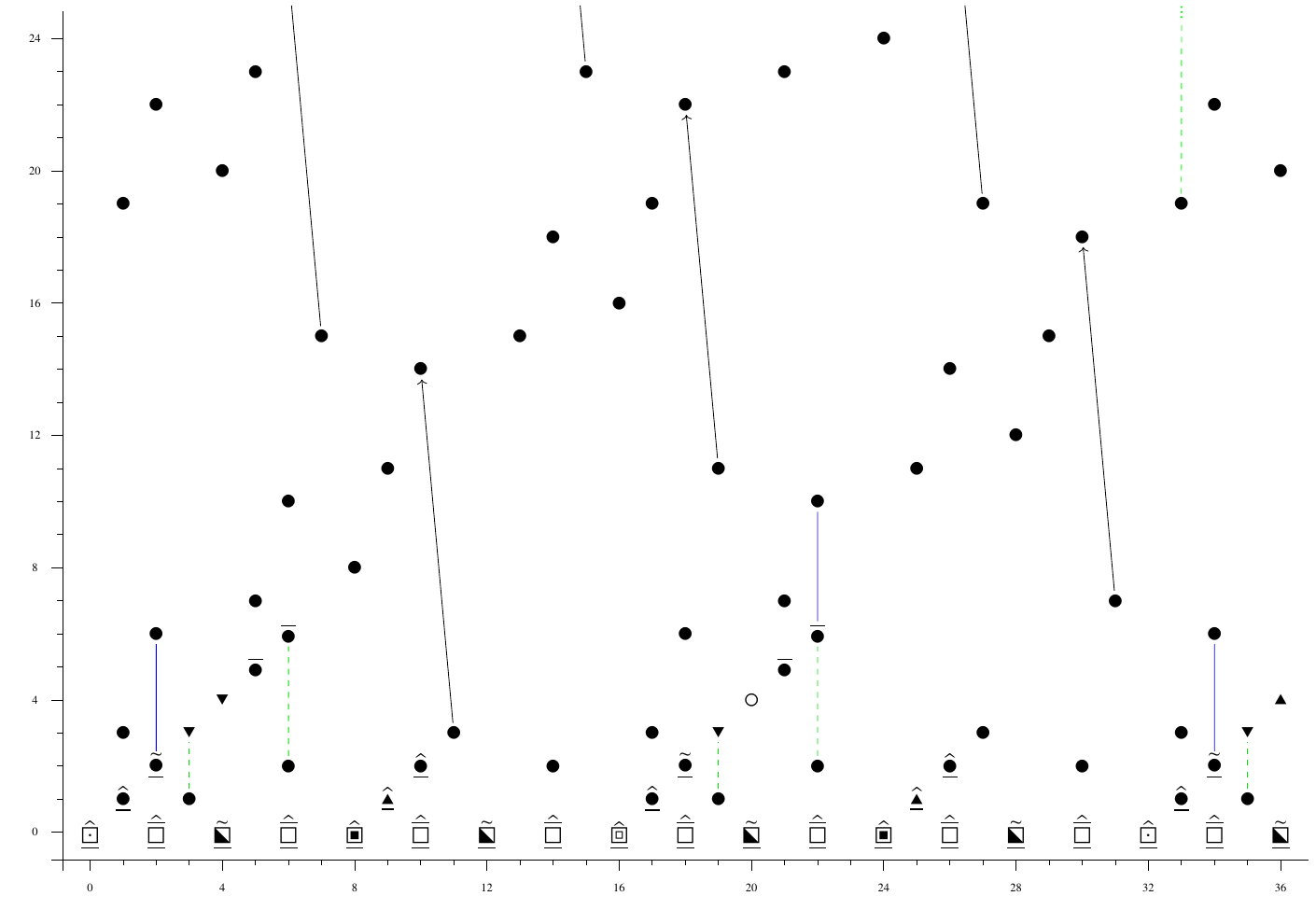}
 \end{center}
 \caption{The $\underline{E}_{11}^{*,*}$ page of the HFPSS with $d_{11}$ differentials.}
 \label{fig:E5page4}
 \end{figure}

    \begin{figure}[H]
\begin{center}
 \includegraphics[angle=90, page=5, width=\textwidth]{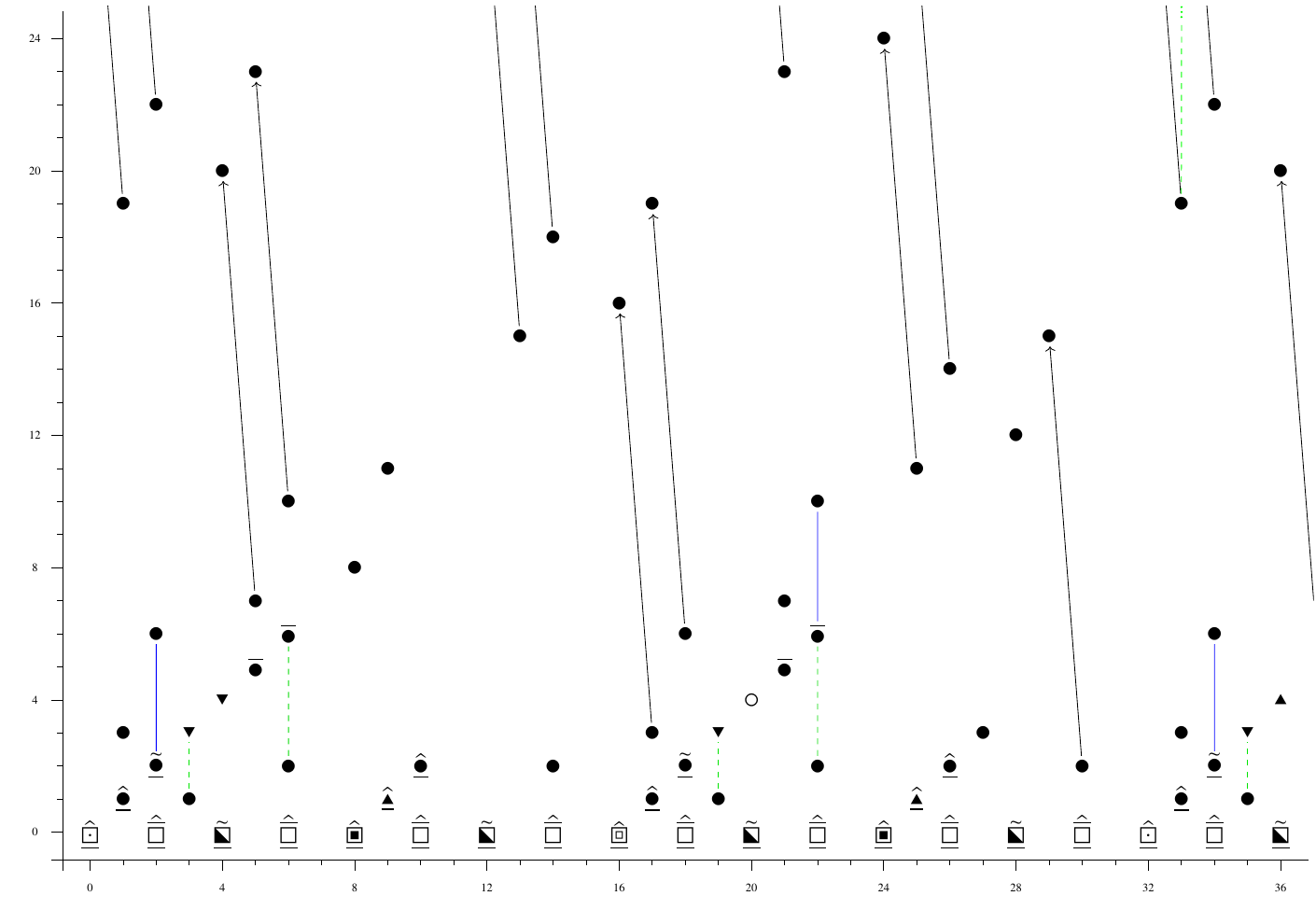}
 \end{center}
 \caption{The $\underline{E}_{13}^{*,*}$ page of the HFPSS with $d_{13}$ differentials.}
 \label{fig:E5page5}
 \end{figure}

 \begin{figure}[H]
\begin{center}
 \includegraphics[angle=90, page=6, width=\textwidth]{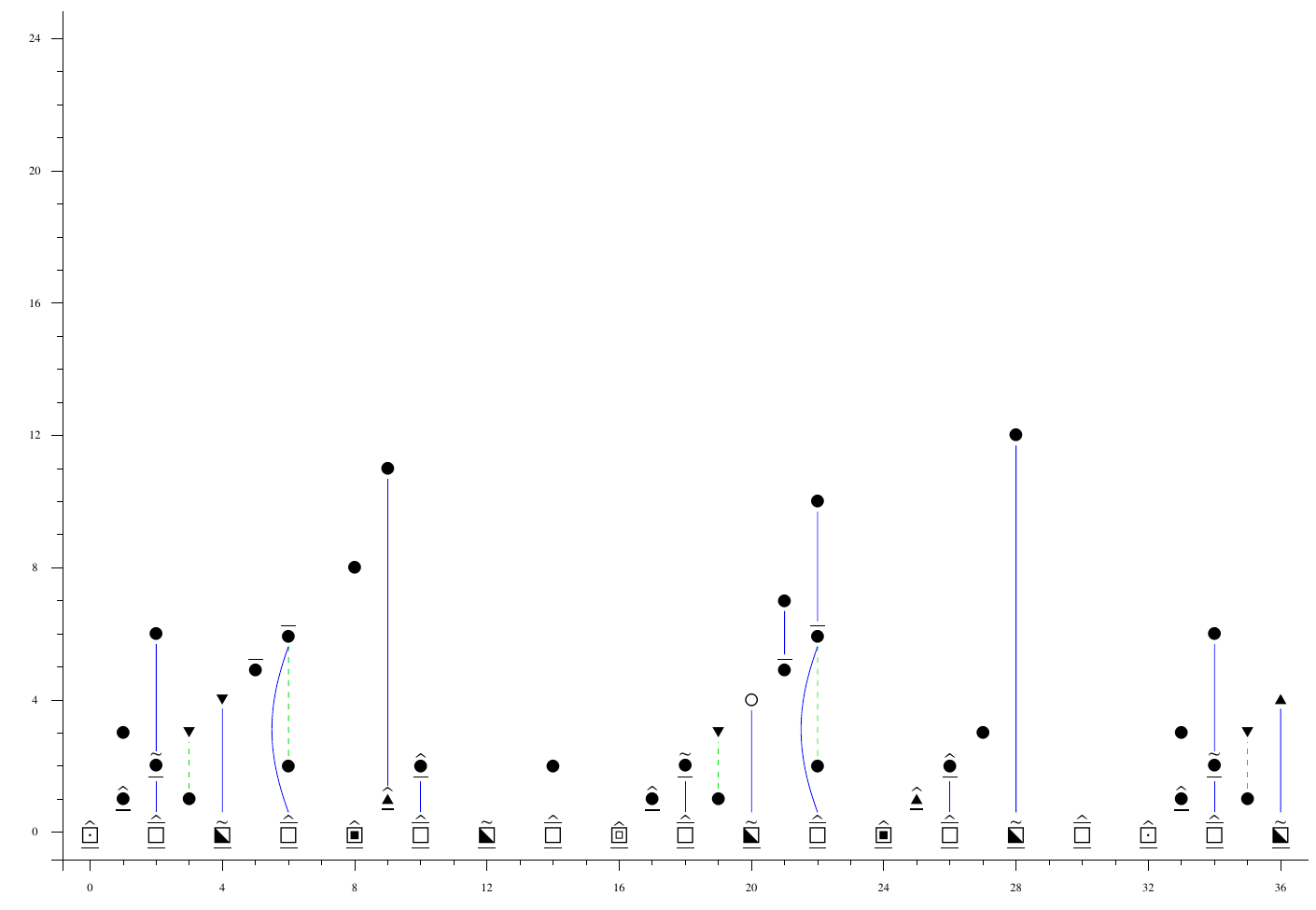}
 \end{center}
 \caption{The $\underline{E}_{\infty}^{*,*}$ page of the HFPSS with exotic extensions.}
 \label{fig:Einf}
 \end{figure}

\begin{figure}[H]
\begin{center}
 \includegraphics[angle=90, page=1, width=\textwidth]{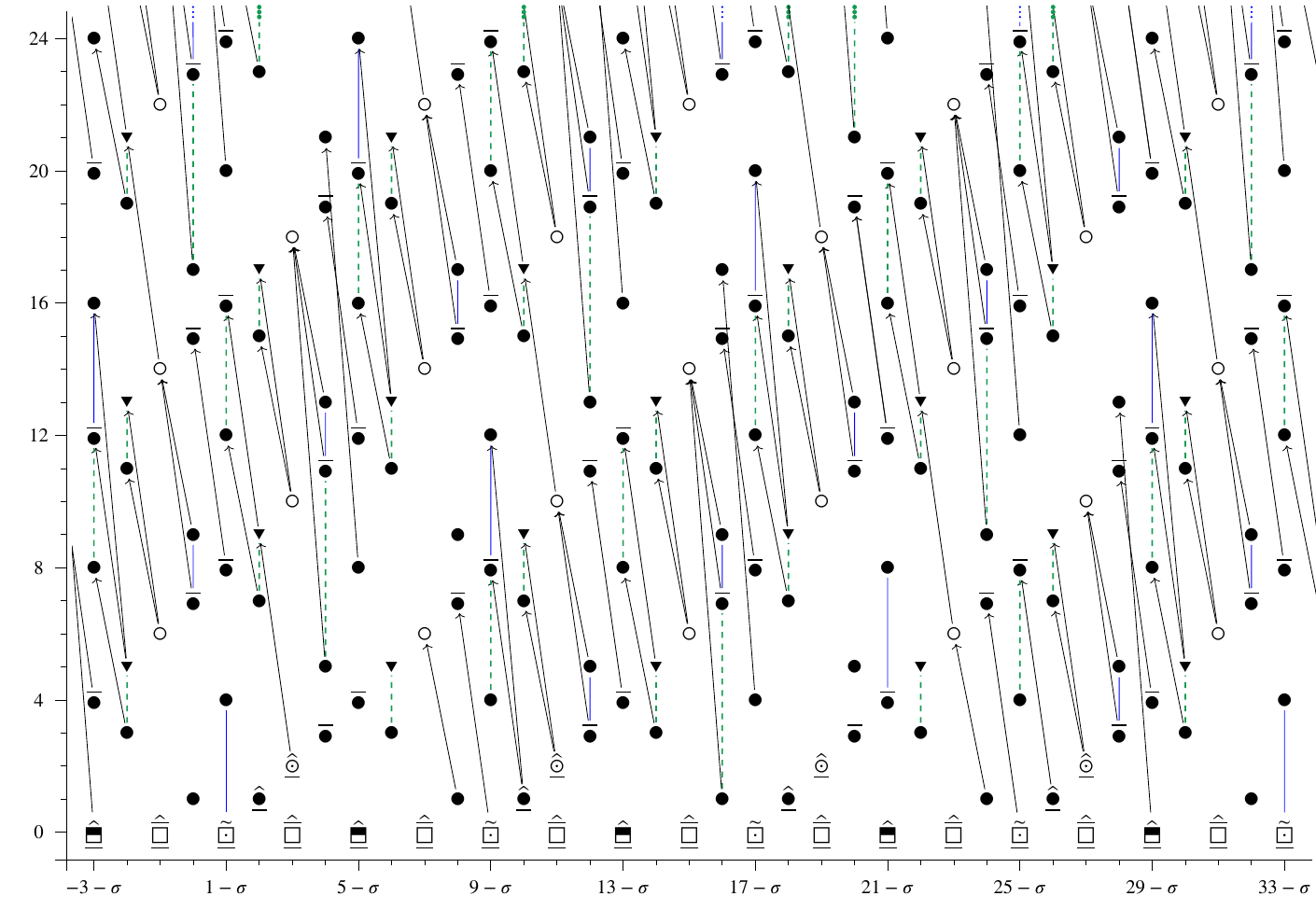}
 \end{center}
 \caption{The $\underline{E}_4^{*,(1-\sigma)+*}$ page onwards of the HFPSS with $d_r$ for $r\geq 5.$}
 \label{fig:E5pagesigma}
 \end{figure}

\begin{figure}[H]
\begin{center}
 \includegraphics[angle=90, page=2, width=\textwidth]{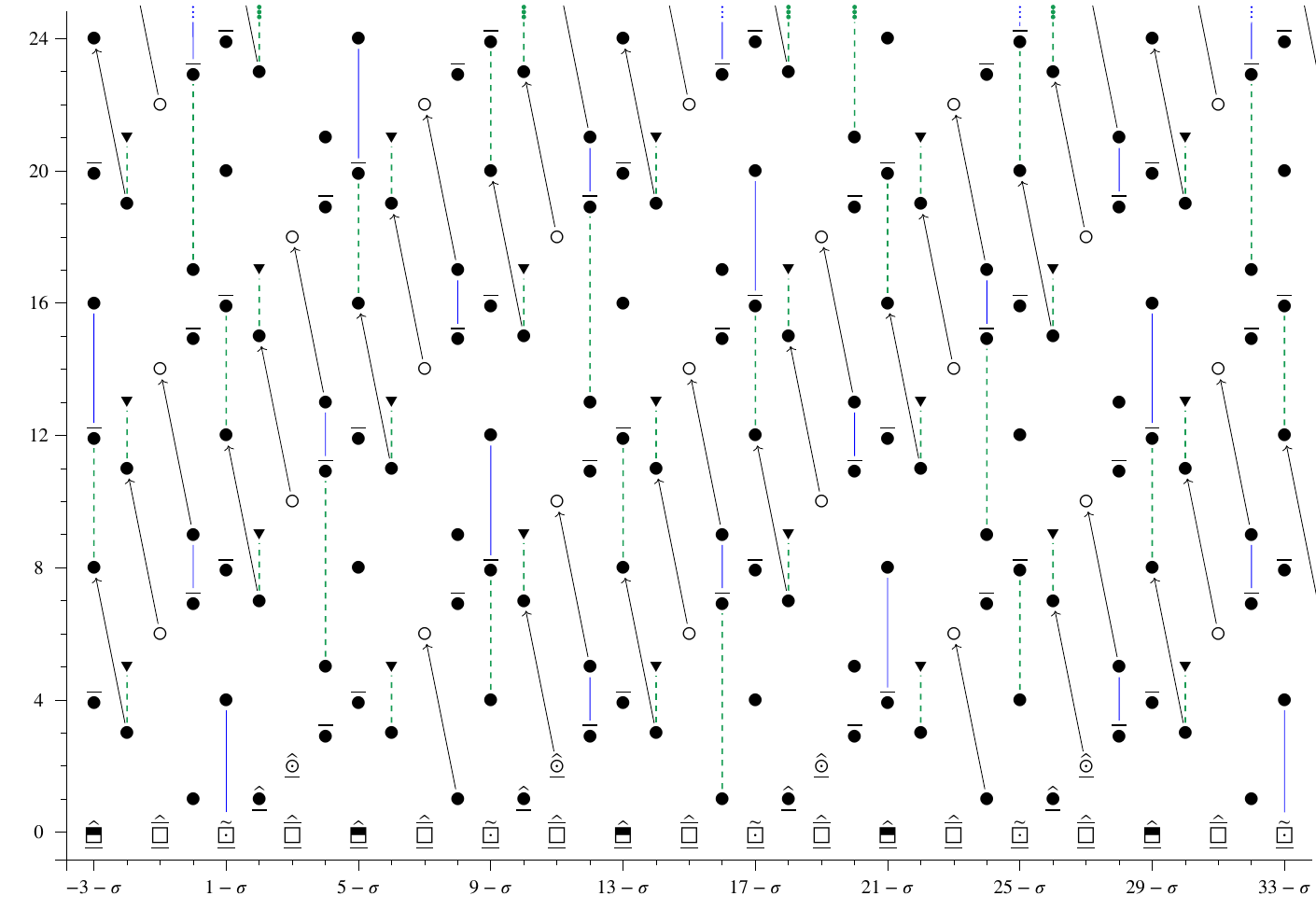}
 \end{center}
 \caption{The $\underline{E}_5^{*,(1-\sigma)+*}$ page of the HFPSS with $d_5$ differentials.}
 \label{fig:E5pagesigma2}
 \end{figure}
 
 \begin{figure}[H]
\begin{center}
 \includegraphics[angle=90, page=3, width=\textwidth]{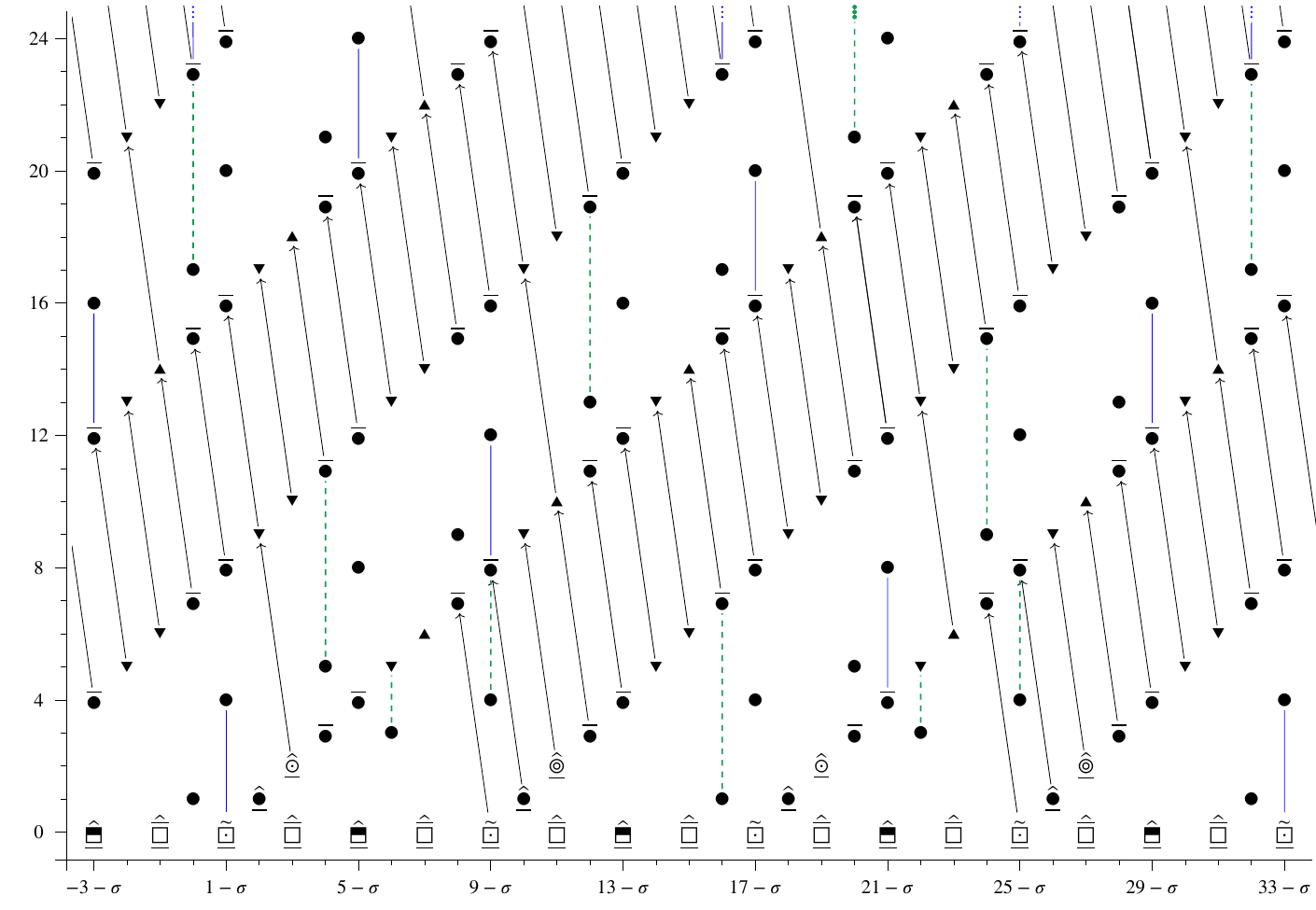}
 \end{center}
 \caption{The $\underline{E}_7^{*,(1-\sigma)+*}$ page of the HFPSS with $d_7$ differentials.}
 \label{fig:E5pagesigma3}
 \end{figure}
 
  \begin{figure}[H]
\begin{center}
 \includegraphics[angle=90, page=4, width=\textwidth]{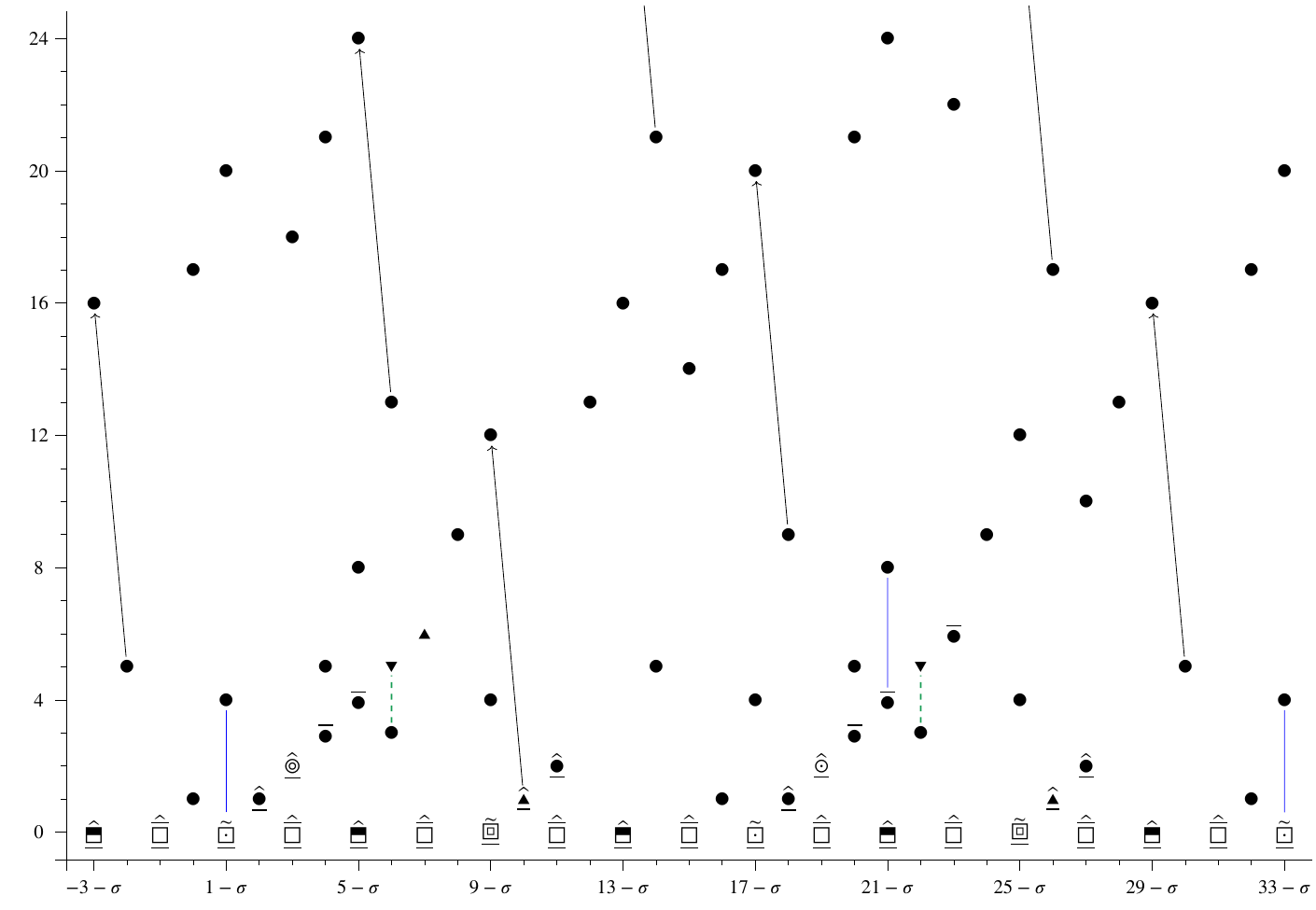}
 \end{center}
 \caption{The $\underline{E}_{11}^{*,(1-\sigma)+*}$ page of the HFPSS with $d_{11}$ differentials.}
 \label{fig:E5pagesigma4}
 \end{figure}

  \begin{figure}[H]
\begin{center}
 \includegraphics[angle=90, page=5, width=\textwidth]{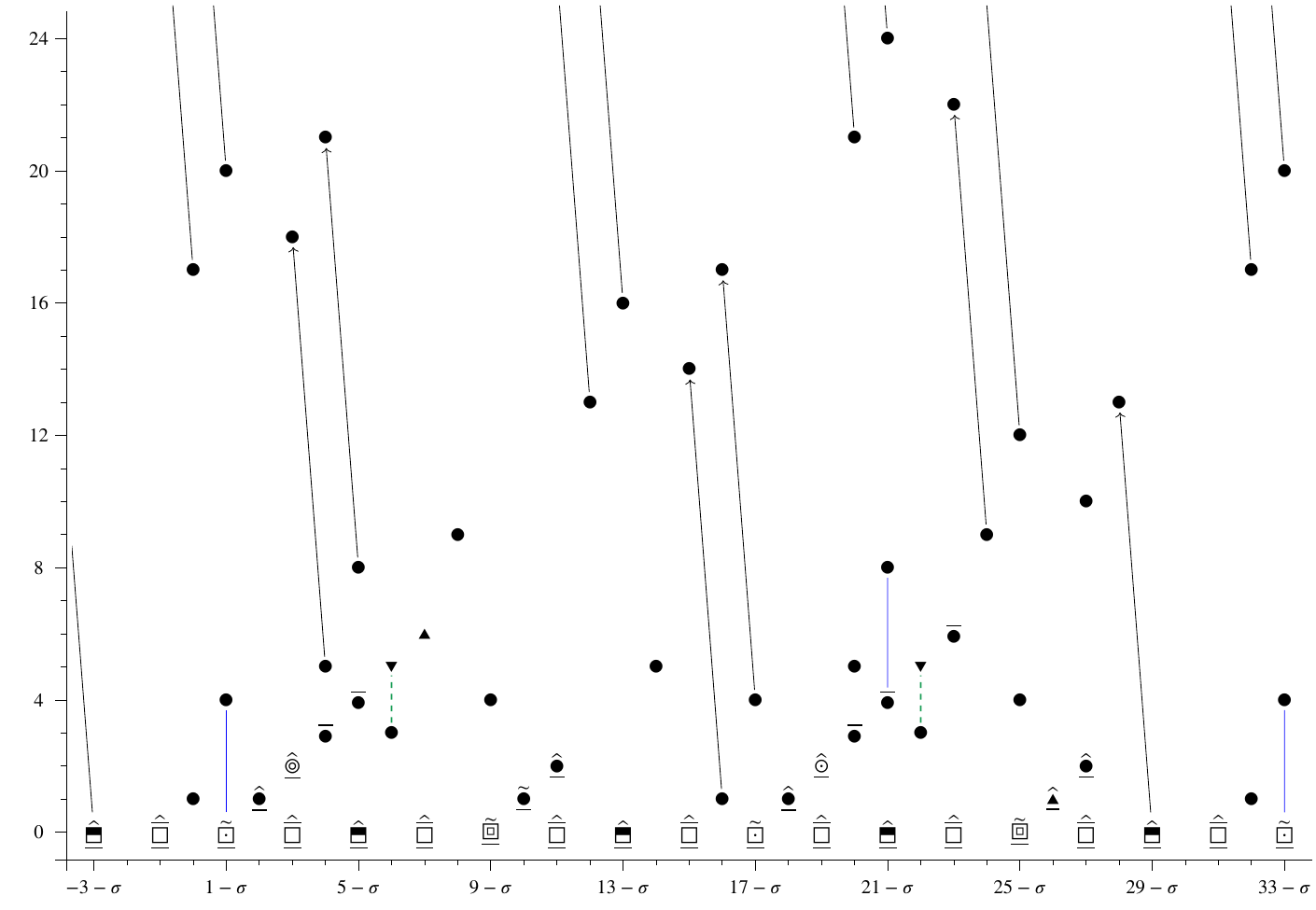}
 \end{center}
 \caption{The $\underline{E}_{13}^{*,(1-\sigma)+*}$ page of the HFPSS with $d_{13}$ differentials.}
 \label{fig:E5pagesigma5}
 \end{figure}

\begin{figure}[H]
\begin{center}
 \includegraphics[angle=90, page=6, width=\textwidth]{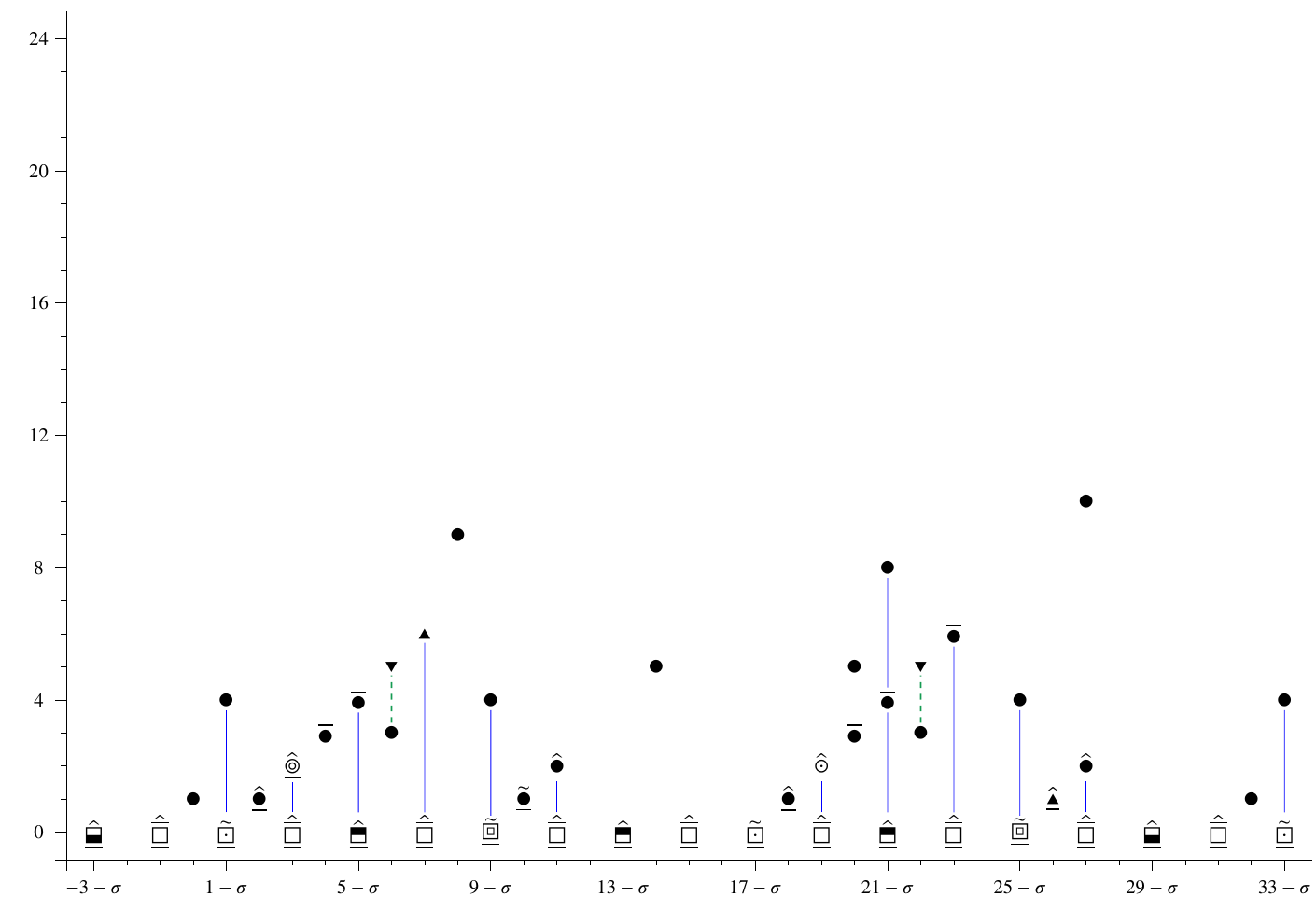}
 \end{center}
 \caption{The $\underline{E}_{\infty}^{*,(1-\sigma)+*}$ page of the HFPSS with exotic extensions.}
 \label{fig:Einfsigma}
 \end{figure}

\newpage
\subsubsection{Exotic restrictions and transfers}
Exotic restrictions and transfers can mostly be deduced from \cite{HHRC4}. In general, they are solved using \cite[Lemma 4.2]{HHRC4} which states that, for a cyclic $2$-group $G$ with finite subgroup $G'$ of index two, if $\sigma$ denotes the sign representation for $G$, then
\begin{itemize}
\item $\im(\tr_{G'}^{G}) = \ker(a_{\sigma})$, and
\item $\ker(\res_{G'}^{G}) = \im(a_{\sigma})$.
\end{itemize}
We have explicitly computed the $\underline{E}_{\infty}^{*,*+\star}$-pages for $\star=0, 1-\sigma$. Further, as we noted in \Cref{rem:allcompsintwo},
 there are isomorphisms
\begin{align*}
\underline{E}_{\infty}^{*,*} &\cong  \underline{E}_{\infty}^{*,*+18-2\sigma} \\
\underline{E}_{\infty}^{*,1-\sigma+*} &\cong  \underline{E}_{\infty}^{*,*+15+\sigma}
\end{align*}
So, it is straightforward to study the image and the kernel of multiplication by $a_{\sigma}$ and $a_{\sigma_2}$ on $\underline{E}_{\infty}^{*,*}$ and $\underline{E}_{\infty}^{*,1-\sigma+*}$. This allows us to use the above observation to deduce exotic transfers and restrictions that are not stated in \cite{HHRC4}.

We first list the extensions for $\spi_*E$, and then turn to the shift by $1-\sigma$. The homotopy groups ${\spi}_{*}E$ and ${\spi}_{1-\sigma+*}E$ are listed in \Cref{table:finalfortrivial}. See also \Cref{fig:Einf} and \Cref{fig:Einfsigma}.
\begin{enumerate}
\item In ${\spi}_{t}$ for $t\equiv 2 \mod (32)$, the exotic transfers resolve in two steps as
\[\xymatrix@R=0.7pc{ 
0 \ar[r] & \mackeybullet \ar[r] & \mackeyunderlinedcirccirchat \ar[r] & \mackeyunderlinedbullettilde \ar[r]& 0  \ ,\\
0 \ar[r] &\mackeyunderlinedcirccirchat   \ar[r] & \mackeyeinftwoethy \ar[r] &  \mackeyunderlinedboxhatoverline   \ar[r] & 0   \ .
 }\] 
 Note that, in \cite[Figure 17]{HHRC4}, $\spi_2K_{[2]}$ is a direct sum of $\mackeypictriangledown$ and $ \mackeypiceinften$, where
 \[\xymatrix@R=0.7pc{ 
0 \ar[r] & \mackeypicbullet \ar[r] & \mackeypictriangledown \ar[r] & \mackeypicbulletoverline \ar[r]& 0  \ ,\\
0 \ar[r] &\mackeypicbullethat  \ar[r] & \mackeypiceinften \ar[r] &  \mackeypicboxhatoverline   \ar[r] & 0   \ .
 }\] 
 and the gray Mackey functors are defined as in \Cref{table:Mackey}, but with $\fk$ replaced by $\Z/2$ and $\W$ replaced by $\Z$. However, $ \spi_2E_{[2]}$ cannot be expressed as a direct sum because of the $\W[\![\mu]\!]$-module structure. This obstruction can be understood in terms of the diagram 
 \[\xymatrix{\mackeypictriangledown(C_4/C_2) \oplus \mackeypiceinften(C_4/C_2) \ar[r] \ar[d]_-{\cong} &    \mackeyeinftwoethy(C_4/C_2)  \ar[d]^-{=}\\
   \Z/2\{\ast\} \oplus  \Z/2[C_4/C_2]  \ar[r]^-{\subset} & \frac{\fk[\![ \mu]\!] [{C_4/C_2}_+]}{\mu\cdot * =\Delta }   }\]
\item In ${\spi}_{t}$ for $t\equiv 3, 19 \mod (32)$, the exotic restriction resolves as
\[\xymatrix{
0 \ar[r] & \mackeyblacktriangledown \ar[r] & \mackeycirc \ar[r] & \mackeybullet \ar[r] & 0 \ .
}\]
\item In ${\spi}_{t}$ for $t\equiv 4 \mod (32)$, the exotic transfer resolves as
\[\xymatrix{
0 \ar[r] &\mackeyblacktriangledown   \ar[r] & \mackeyeinffourethy \ar[r] &  \mackeyunderlinedboxleftslashtilde  \ar[r] & 0  \ .
}\]
\item In ${\spi}_{t}$ for $t\equiv 6 \mod (32)$, the exotic restriction and transfer resolves in two steps as
\[\xymatrix@R=0.7pc{ 
0 \ar[r] & \mackeybulletoverline \ar[r] & \mackeyblacktriangle \ar[r] & \mackeybullet \ar[r]& 0 \ , \\
0 \ar[r] &\mackeyblacktriangle   \ar[r] &  \mackeyeinfsixethy  \ar[r] &  \mackeyunderlinedboxhatoverline   \ar[r] & 0   \ .
 }\] 
\item In ${\spi}_{t}$ for $t\equiv 9 \mod (32)$, the exotic transfer resolves as
\[\xymatrix{
0 \ar[r] &\mackeybullet   \ar[r] & \mackeyunderlinedodothat \ar[r] &  \mackeyunderlinedblacktrianglehat  \ar[r] & 0  \ .
}\]
\item In ${\spi}_{t}$ for $t\equiv 10, 26 \mod (32)$, the exotic transfer resolves as
\[\xymatrix{
 0\ar[r] & \mackeyunderlinedbullethat \ar[r] & \mackeyeinftenethy  \ar[r] &  \mackeyunderlinedboxhatoverline   \ar[r] & 0   \ .
}\]
\item in ${\spi}_{t}$ for $t\equiv 18 \mod (32)$, the exotic transfer resolves as
\[\xymatrix{
0 \ar[r] & \mackeyunderlinedbullettilde \ar[r] &  \mackeyeinfeighteenethy \ar[r] & \mackeyunderlinedboxhatoverline \ar[r] & 0 \ .
} \]
\item In ${\spi}_{t}$ for $t\equiv 20 \mod (32)$, the exotic transfer resolves as
\[\xymatrix{
0 \ar[r] & \mackeycirc \ar[r] & \mackeyeinftwentyethy  \ar[r] & \mackeyunderlinedboxleftslashtilde \ar[r] & 0 \ .
 } \]
\item In ${\spi}_{t}$ for $t\equiv 21 \mod (32)$, the exotic transfer resolves as
\[\xymatrix{
0 \ar[r] & \mackeybullet \ar[r] & \mackeyblacktriangledown  \ar[r] & \mackeybulletoverline \ar[r] & 0 \ .
} \]
\item In ${\spi}_{t}$ for $t\equiv 22 \mod (32)$, the exotic restriction and transfer resolves in three steps as
\[\xymatrix@R=0.7pc{ 
0 \ar[r] & \mackeybullet \ar[r] & \mackeyblacktriangledown  \ar[r] & \mackeybulletoverline \ar[r] & 0 \ , \\
0 \ar[r] &\mackeyblacktriangledown   \ar[r] & \mackeycirc \ar[r] &  \mackeybullet   \ar[r] & 0 \ ,  \\
0 \ar[r] &\mackeycirc   \ar[r] & \mackeyeinftwentytwoethy \ar[r] &  \mackeyunderlinedboxhatoverline   \ar[r] & 0  \ .
 }\] 
 \item In ${\spi}_{t}$ for $t\equiv 28 \mod (32)$, the exotic transfer resolves as
\[\xymatrix{
0 \ar[r] & \mackeybullet \ar[r] & \mackeyeinftwentyeightethy  \ar[r] & \mackeyunderlinedboxleftslashtilde \ar[r] & 0 \ .
} \]
\end{enumerate}

Next, we list the extensions for $\spi_{*+1-\sigma}E$. Most of them can be read off the computation of the homotopy groups and we add a word for those that are not so clear.
\begin{enumerate}
\item In ${\spi}_{t}$ for $t\equiv 1-\sigma \mod (32)$, the exotic transfer resolves as
\[\xymatrix{
0 \ar[r] & \mackeybullet \ar[r] & \mackeysigmaminusoneethy \ar[r] &  \mackeyunderlinedboxdottilde  \ar[r] & 0 \ .
}\]
\item In ${\spi}_{t}$ for $t\equiv 3-\sigma \mod (32)$, the exotic transfer resolves as
\[\xymatrix{
0 \ar[r] &\mackeyunderlinedcirccirchat   \ar[r] & \mackeyeinftwoethy \ar[r] &  \mackeyunderlinedboxhatoverline   \ar[r] & 0   \ .
}\]
\item In ${\spi}_{t}$ for $t\equiv 5-\sigma \mod (32)$, the exotic restriction resolves as
\[\xymatrix{
0 \ar[r] & \mackeybulletoverline \ar[r] & \mackeyeinffoursigmaethy \ar[r] &  \mackeyunderlinedhalftopboxhat \ar[r] & 0 \ .
}\]
\item In ${\spi}_{t}$ for $t\equiv 6-\sigma, 22-\sigma \mod (32)$, the exotic restriction resolves as
\[\xymatrix{
0 \ar[r] & \mackeyblacktriangledown \ar[r] & \mackeycirc \ar[r] &  \mackeybullet \ar[r] & 0 \ .
}\]
\item In ${\spi}_{t}$ for $t\equiv 7-\sigma \mod (32)$, the exotic transfer resolves as
\[
\xymatrix{
0 \ar[r] &\mackeyblacktriangle   \ar[r] &  \mackeyeinfsixethy  \ar[r] &  \mackeyunderlinedboxhatoverline   \ar[r] & 0 \ .
}\]
\item In ${\spi}_{t}$ for $t\equiv 9-\sigma, 25-\sigma \mod (32)$, the exotic transfer resolves as
\[
\xymatrix{
0 \ar[r] &\mackeybullet   \ar[r] &  \mackeysigmaminusnineethy  \ar[r] &  \mackeyunderlinedboxboxtilde   \ar[r] & 0 \ .
}\]
There must be an exotic transfer in this degree since $\spi_{9-2\sigma}E(C_4/C_4)=0$ so the class in $(9-\sigma, 4)$ must be in the image of the transfer. Similarly,  $\spi_{25-2\sigma}E(C_4/C_4)=0$ implies the exotic transfer in stem $25-\sigma$.
\item In ${\spi}_{t}$ for $t\equiv 11-\sigma, 27-\sigma \mod (32)$, the exotic transfer resolves as
\[\xymatrix{
 0\ar[r] & \mackeyunderlinedbullethat \ar[r] & \mackeyeinftenethy  \ar[r] &  \mackeyunderlinedboxhatoverline   \ar[r] & 0   \ .
}\]
\item In ${\spi}_{t}$ for $t\equiv 19-\sigma \mod (32)$, the exotic transfer resolves as
\[\xymatrix{
 0\ar[r] & \mackeyunderlinedodothat \ar[r] & \mackeyunderlinedboxdotcheck  \ar[r] &  \mackeyunderlinedboxhatoverline   \ar[r] & 0   \ .
}\]
\item In ${\spi}_{t}$ for $t\equiv 21-\sigma \mod (32)$, the exotic restriction and transfer resolve in two steps as
\[\xymatrix@R=0.7pc{ 
0 \ar[r] & \mackeybullet \ar[r] & \mackeyblacktriangledown  \ar[r] & \mackeybulletoverline \ar[r] & 0 \ ,
 \\
0 \ar[r] &\mackeyblacktriangledown   \ar[r] & \mackeyunderlinedhalftopboxcheck \ar[r] &  \mackeyunderlinedhalftopboxhat   \ar[r] & 0  \ .
 }\] 
 \item In ${\spi}_{t}$ for $t\equiv 23-\sigma  \mod (32)$, the exotic transfer resolves as
\[\xymatrix{
 0\ar[r] & \mackeybulletoverline \ar[r] & \mackeyunderlinedblackboxwhitecirchatoverline  \ar[r] &  \mackeyunderlinedboxhatoverline   \ar[r] & 0   \ .
}\]
\end{enumerate}
\begin{rem}
There is no exotic transfer to the class detected in degree $(27-\sigma, 10)$ since this class is not in the kernel of multiplication by $a_{\sigma}$. Similarly, there is no exotic transfer in stem $20-\sigma$.
\end{rem}

\begin{cor}\label{cor:notintegershit}
The homotopy groups of $\spi_{*+1-\sigma}E$ are not an integer shift of $\spi_*E$. As was stated in \Cref{prop:oneminussigma}, this implies that there is no $d\in \Z$ such that $E \smsh S^{\sigma-1} \simeq E \smsh S^d$ as $C_4$ equivariant $E$-module spectra.
\end{cor}

\begin{table}[H]
\begin{center}
\begin{tabular}{ |C|C|C|C|C|C|C|C|C|C|C|C|C|C|C|C|C|C|C|C|C|C|C|C|C|C|C|C|C|C|C|C|C|    }
\hline
   \multicolumn{9}{|c|}{Homotopy groups ${\spi}_{*}E $ }\\
 \hline
t \mod 32 &  0 & 1 & 2 & 3 & 4 & 5 & 6 & 7   \\ 
\hline
{\spi}_{t}E & \mackeyunderlinedboxdothat   & \mackeyunderlinedbullethat \oplus \mackeybullet  & \mackeyeinftwoethy  & \mackeycirc & \mackeyeinffourethy  &  \mackeybulletoverline & \mackeyeinfsixethy  &  0  \\ 
 \hline 
\multicolumn{9}{|C|}{} \\
 \hline
t \mod 32 &  8 & 9 & 10 & 11 & 12 & 13 & 14 & 15   \\ 
\hline
{\spi}_{t}E & \mackeyunderlinedboxblackboxhat \oplus   \mackeybullet   & \mackeyunderlinedodothat & \mackeyeinftenethy  & 0   &  \mackeyunderlinedboxleftslashtilde  & 0  & \mackeyunderlinedboxhatoverline \oplus   \mackeybullet & 0 \\ 
 \hline
 \multicolumn{9}{|C|}{} \\
  \hline
t \mod 32 &  16 & 17 & 18 & 19 & 20 & 21 & 22 & 23   \\ 
\hline
{\spi}_{t}E & \mackeyunderlinedboxboxhat    & \mackeyunderlinedbullethat & \mackeyeinfeighteenethy  & \mackeycirc   &  \mackeyeinftwentyethy &\mackeyblacktriangledown  & \mackeyeinftwentytwoethy  & 0  \\ 
 \hline
 \multicolumn{9}{|C|}{} \\
   \hline
t \mod 32 &  24 & 25 & 26 & 27 & 28 & 29 & 30 & 31   \\ 
\hline
{\spi}_{t}E & \mackeyunderlinedboxblackboxhat    & \mackeyunderlinedblacktrianglehat & \mackeyeinftenethy  & \mackeybullet   &  \mackeyeinftwentyeightethy & 0  & \mackeyunderlinedboxhatoverline  & 0  \\ 
 \hline
 \multicolumn{9}{|C|}{} \\
   \hline
   \multicolumn{9}{|c|}{Homotopy groups ${\spi}_{1-\sigma+*}E $ }\\
 \hline
t \mod 32 &  1 & 2 & 3 & 4 & 5 & 6 & 7 & 8   \\ 
\hline
{\spi}_{t-\sigma}E &  \mackeysigmaminusoneethy  &  \mackeyunderlinedbullethat   &  \mackeyeinftwoethy & \mackeybulletoverline & \mackeyeinffoursigmaethy   & \mackeycirc  & \mackeyeinfsixethy  & \mackeybullet   \\ 
 \hline 
\multicolumn{9}{|C|}{} \\
 \hline
t \mod 32 &  9 & 10 & 11 & 12 & 13 & 14 & 15 & 16   \\ 
\hline
{\spi}_{t-\sigma}E & \mackeysigmaminusnineethy   & \mackeyunderlinedbullettilde    &  \mackeyeinftenethy  & 0  &  \mackeyunderlinedhalftopboxhat  &  \mackeybullet  & \mackeyunderlinedboxhatoverline  &  0   \\ 
 \hline
 \multicolumn{9}{|C|}{} \\
  \hline
t \mod 32 &  17 & 18 & 19 & 20 & 21 & 22 & 23 & 24   \\ 
\hline
{\spi}_{t-\sigma}E &  \mackeyunderlinedboxdottilde   &  \mackeyunderlinedbullethat   & \mackeyunderlinedboxdotcheck  & \mackeybulletoverline \oplus \mackeybullet  & \mackeyunderlinedhalftopboxcheck  & \mackeycirc  &  \mackeyunderlinedblackboxwhitecirchatoverline &  0   \\ 
 \hline
 \multicolumn{9}{|C|}{} \\
   \hline
t \mod 32 &  25 & 26 & 27 & 28 & 29 & 30 & 31 & 32   \\ 
\hline
{\spi}_{t-\sigma}E & \mackeysigmaminusnineethy   &   \mackeyunderlinedblacktrianglehat  &  \mackeyeinftenethy \oplus \mackeybullet   & 0  & \mackeyunderlinedhalfbottomboxhat & 0  &  \mackeyunderlinedboxhatoverline & \mackeybullet   \\ 
 \hline
\end{tabular}
\caption{The homotopy groups ${\spi}_{*}E $ (top) and ${\spi}_{1-\sigma+*}E $ (bottom).}
\label{table:finalfortrivial}
\label{table:finalfortrivialsigma}
\end{center}
\end{table}


\section{The Algebraic Picard Group}\label{sec:alg}
In this section, we compute the algebraic part of $\mPic(E)$. First, we recall a few facts from algebra. 

Let $R$ be a commutative ring with an action of a group $G$ by ring automorphisms. We consider the category of $G$-twisted $R$-modules where the objects are $R$-modules $M$ with an action of $G$ compatible with the action of $G$ on $R$. Namely, for $r\in R, m\in M$ and $g\in G$ 
\[g(rm) =g(r)g(m).\] 
This is a symmetric monoidal category and we let its Picard group be denoted by $\PicalgRG{R}{G}$. If $G$ is a finite group and $R$ is a local ring, then a $G$-twisted $R$-module $M$ is invertible if and only if it is free of rank one as an $R$-module. From this, one deduces that there is an isomorphism
\begin{equation}\label{eq:piccohiso}
\Pic_G(R) \cong H^1(G, R^{\times}). \end{equation}
The isomorphism is defined by choosing an $R$ module generator $m$ for $M$ and defining a function $\phi_M \colon G \to R$ by the formula
\[g(m) = \phi_M(g) m.\]

Note further that if we let $H$ vary over the subgroups of $G$, it is clear that the right hand side assembles as a Mackey functor. This corresponds on the left hand side to
the Mackey functor
\[\mPic(R)(G/H)  = \PicalgRG{R}{H}\]
with $\res_H^G(M)$ the module $M$ with $H$ action restricted along the inclusion of $H$ in $G$ and
\[\tr_H^G(M) = N_{H}^G(M),  \]
where $ N_{H}^G(M) = \bigotimes_{G/H}^R M$ is the indexed tensor product of $R$-modules. 

Now, let $R$ be a ring spectrum with an action of $G$ and $R_0 = \pi_0R$. In the Picard spectral sequence, which we introduce in more detail in the next section, we will have 
\[\underline{E}_{\subpic,2}^{1,1}(G/H) \cong H^1(H, R_0^{\times}) \cong \PicalgRG{R_0}{H}.  \] 
Therefore, the computation of $\mPic(R_0)$ is an input for that of $\mPic(R)$.

In this section, we prove the following statement.
\begin{prop}\label{prop:H14}
There are isomorphisms $\PicalgRG{E_0}{C_4} \cong \Z/4$ and $\PicalgRG{E_0}{C_2} \cong \Z/2$.
\end{prop}

Let $G=C_2$ or $C_4$. Since a $G$-$E_0$-module is in $\PicalgRG{E_0}{G}$ if and only if it is free of rank one as an $E_0$-module, for each integer $n$, $E_{2n}$ is an element of $\PicalgRG{E_0}{G}$. Further, the multiplication
\begin{equation} 
E_{2n} \otimes_{E_0} E_{2m} \to  E_{2(n+m)}
\end{equation}
induces an isomorphism. This gives a group homomorphism
\[\varphi^G(-) \colon \Z \to  \PicalgRG{E_0}{G}\]
where $\varphi^G(n) = E_{2n}$.
\begin{lem}
Let $k=0,1,2$. The kernel of $\varphi^{C_{2^k}}$ is the ideal $(2^k)\Z$.
\end{lem}
\begin{proof}
First, let $G=C_4$.
If $f\colon E_0 \to E_{2n}$ is an isomorphism of $C_4$-twisted $E_0$-modules, then $f(1)$ is a unit of degree $2n$ which is invariant modulo the action of $C_4$. Conversely, any such isomorphism is given by multiplication by such an invariant unit. The element $\Delta_1$ defined in \Cref{tab:elementsE} has this property, so multiplication by $\Delta_1$ induces an equivariant isomorphism 
\[\Delta_1 \colon E_{2n} \to E_{2n+8}\]
so that $E_{2*}$ is at least $4$-periodic and $(4) \subseteq \ker(\varphi^{C_4})$. There is no such unit in $E_2$, $E_4$ or $E_6$, so this identifies the kernel.

The argument for $C_2$ is similar, replacing $\Delta_1$ by $\delta_1$ and that for the trivial group is obvious.
\end{proof}

As an immediate consequence, we have:
\begin{cor}\label{cor:incZ4}
Let $k=0,1,2$. There are inclusions $\Z/2^k \subseteq \PicalgRG{E_0}{C_{2^k}}$ where $\Z/2^k$ is generated by the isomorphism class of $E_2$.
\end{cor}
To finish the proof of \Cref{prop:H14}, we show that $H^1(C_{2^k},E_0^{\times})$ has order at most $2^{k}$. 
Once we have shown this, we can assemble the Mackey functor $\mPic(E_0) \cong \underline{H}^{1}(C_4, E_0^{\times})$. The effect of
\[\res_2^4 \colon \PicalgRG{E_0}{C_4} \to  \PicalgRG{E_0}{C_2}  \] 
is obvious since it sends the generator $E_2$ to itself. Since $\tr_2^{4} \circ \res_2^{4}=|C_4/C_2|$ (for example by \eqref{eq:piccohiso}), the transfer must be multiplication by $2$. Therefore, $\mPic(E_0)$ is the following Mackey functor:
\[\Mackey{ \PicalgRG{E_0}{C_4}  \cong \Z/4}{ \PicalgRG{E_0}{C_2} \cong \Z/2}{ \PicalgRG{E_0}{\{e\}} =0 }{0}{0}{2}{1}\]
We denote this Mackey functor by $\mackeypiccirc$.

The remainder of the section is dedicated to proving that $H^1(C_{2^k}, E_0^{\times})$ has order at most $2^k$ for $k=1,2$. This is rather technical and the next section will only appeal to the statement of \Cref{prop:H14}, so the reader may safely skip the remainder of this section.
We treat $C_2$ and $C_4$ separately since proving this for $C_2$ is much easier. 
\begin{lem}
There is an isomorphism 
\[H^1(C_2, E_0^{\times})\cong \Z/2.\]
\end{lem}
\begin{proof}
The group $C_2$ is generated by $\gamma^2$ which acts trivially on $E_0^{\times}$; see \Cref{sec:act}.
Therefore, we have
\[H^1(C_2, E_0^{\times}) = \ker\left(E_0^{\times}  \xra{(-)^2} E_0^{\times}\right),\]
but, $x^2=1$ in $E_0$ if and only if $x=(\pm 1)$.
\end{proof}

To deal with $\PicalgRG{E_0}{C_4} $, we introduce some notation. Let $U_0 = E_0^{\times}$ and, for $n\geq 1$,
\[U_n = \left\{x \in E_0^{\times} : x \equiv 1 \mod (2^n) \right\}.  \]
There are isomorphisms 
\[U_n /U_{n+1} \cong \begin{cases}  E_0^{\times}/2 & n=0 \\
 E_0/2&n\geq 1 .
\end{cases}\]
We will use the long exact sequence on cohomology associated to the short exact sequence
\begin{equation}\label{eq:algkeyseq}
 \xymatrix{0 \ar[r] & U_1 \ar[r] & E_0^{\times} \ar[r]  & E_0^{\times}/2 \ar[r] & 0 .} \end{equation}

\begin{lem}\label{lem:ex}
The maps $E_0^{C_4} \to (E_0/2)^{C_4}$ and $(E_0^{\times})^{C_4} \to (E_0^{\times}/2)^{C_4}$ are surjective.
\end{lem}
\begin{proof}
Since an element of $E_0$ is a unit if and only if it is a unit modulo $2$, it suffices to prove that the map $E_0^{C_4} \to  (E_0/2)^{C_4}$ is a surjection. Since ${H}^1(C_{4}, E_0) =0$ (see, for example, \Cref{figure:additive-E2}) this follows from the long exact sequence associated to the short exact sequence 
 \[\xymatrix{0 \ar[r] & E_0 \ar[r]^-{2}  & E_0 \ar[r] & E_0/2 \ar[r] & 0.}\qedhere\]
 \end{proof}
It follows from \Cref{lem:ex} that there is an exact sequence
 \[\xymatrix{ 0 \ar[r] & {H}^1(C_4, U_1) \ar[r] & {H}^1(C_4, E_0^{\times}) \ar[r]  & {H}^1(C_4,  E_0^{\times}/2) }  \]

 We will show that both $H^1(C_4, U_1)$ and $H^1(C_4,  E_0^{\times}/2) $ have order $2$, which will imply that the order of $ H^1(C_4, E_0^{\times}) $ is at most $4$. Together with \Cref{cor:incZ4}, these results prove \Cref{prop:H14}.
 
 \begin{rem}
One can also prove the claim by filtering $ E_0^{\times}$ by powers of its maximal ideal. That argument is slightly shorter but more technical. We have opted to take the slightly longer, but less steep trail.
 \end{rem}
 
 \begin{prop}\label{prop:H1U1}
There is an isomorphism $H^1(C_4, U_1) \cong \Z/2$.
 \end{prop}
 \begin{proof}
 To compute $H^1(C_4, U_1)$, we use the Bockstein spectral sequence
 \begin{align} \label{eq:bockU}
 E_1^{s,n} = H^s(C_4, U_{n}/U_{n+1}) \Longrightarrow H^{s}(C_4, U_1)
 \end{align}
 with differentials $d_r\colon E_r^{s,n} \to E_r^{s+1,n+r}$ and $n\geq 1$. See \Cref{fig:bockV}.

The differentials are described as follows. First, if $\gamma^2$ acts trivially on a $C_4$-module $M$,
\[\gamma^3(-)\gamma^2(-)\gamma(-)(-) =(\gamma(-)(-))^2  \]
on $M$.
The standard resolution for the group $C_4$ gives a cochain complex
\begin{align} \label{eq:cochain}
\xymatrix{ 0 \ar[r] & P^0(M) \ar[rr]^-{\gamma(-)(-)^{-1}}_-{d} & & P^1(M) \ar[rr]^-{(\gamma(-)(-))^2}_-d  & & P^2(M) \ar[rr]^-{\gamma(-)(-)^{-1}}_-d & & \ldots  } 
\end{align}
with $P^s(M) = M$ for all $s \geq 0$. The cohomology of this complex is $H^*(C_4, M)$.

For an element 
\[\alpha \in E_r^{s,n} \cong H^s(C_4, U_n/U_{n+1})\]
a differential $d_r(\alpha)$ in the spectral sequence \eqref{eq:bockU}
 is the Bockstein associated to the exact sequence
\[0 \to U_{n+1}/U_{n+r+1} \to U_{n}/U_{n+r+1} \to U_{n}/U_{n+1} \to 0.\]
This is computed by choosing a cocycle representative $a \in P^s(U_{n}/U_{n+1})$ for $\alpha$ and lifting it to an element $\widetilde{a} \in P^s(U_{n}/U_{n+r+1})$. Then $d(\widetilde{a})$ will necessarily be a cocycle in 
\[P^{s+1}\left(U_{n+r}/U_{n+r+1}\right) \subseteq  P^{s+1}\left(U_{n+1}/U_{n+r+1} \right)\] 
since $\alpha$ was a $d_{r-1}$-cycle. The cohomology class represented by 
\[d (\widetilde{a}) \in H^{s+1}(C_4, U_{n+r}/U_{n+r+1})\] 
is $d_r(\alpha)$.

Fix an isomorphism
\begin{align*} &   E_0/2 \cong \fk[\![\mu_0]\!] \to U_n/U_{n+1}, 
\\ &  f(\mu_0) \mapsto 1+2^nf(\mu_0) \mod 2^{n+1}\end{align*}
 From \Cref{sec:act}, we have that
 \[ U_n /U_{n+1}   \cong E_0/2 \cong A(+)/2    \]
as $C_4$-modules. Recall that 
\[\mu  = \mu_0 + \gamma(\mu_0)\]
is invariant for the action of $C_4$ (since $\gamma^2(\mu_0)=\mu_0$). Using the standard resolution, we get the following $\W [\![\mu]\!] $-modules
 \begin{align}\label{eq:cohUn}
  E_1^{s,n} \cong H^s(C_4, U_n /U_{n+1} )  \cong \begin{cases} \fk[\![\mu]\!] & s=2t \\
 \fk \oplus \fk[\![\mu]\!] & s=2t+1 .
 \end{cases}\end{align}
with cocycle representatives in $P^s(U_n/U_{n+1})$ given by 
\begin{align*}
f(\mu)  &\leftrightarrow    1+2^{n} f(\mu)  & &s=2t \\
  \alpha  &\leftrightarrow 1+2^n \alpha  & & s=2t+1 \\
f(\mu)  &\leftrightarrow    1+2^n f(\mu)\mu_0 & & s=2t+1 .
 \end{align*}
for $f(\mu) \in \fk[\![\mu]\!]$ and $\alpha \in \fk$.

For $f(\mymu) \in \fk[\![\mymu]\!] = (E_0/2)^{C_4}$, we let $\widetilde{f}(\mymu) \in \W[\![\mymu]\!] = E_0^{C_4}$ be any invariant lift, which exists by \Cref{lem:ex}. Further, given $\alpha \in \fk$, we let $\widetilde{\alpha} \in \W$ be zero if $\alpha=0$ or a Teichm\"uller lift if $\alpha \in \fk^{\times}$.

First, let $s=2t$. For $f(\mymu) \in \fk[\![\mu]\!] = E_1^{2t,n} $
the cocycle representative $1+2^n{f}(\mymu) \in P^{2k}(U_n/U_{n+1})$ lifts to
\[1+2^n\widetilde{f}(\mymu) \in P^{2t}(U_n/U_{n+r+1}). \]
Since $1+2^n\widetilde{f}(\mymu)$ is invariant for the action of $C_4$, we that have
\[d = \gamma(-)(-)^{-1} \colon P^{2t}(U_n/U_{n+r+1}) \to P^{2t+1}(U_n/U_{n+r+1})\]
 is given by $d(1+2^n\widetilde{f}(\mymu)) =1$, which reduces to zero in $U_{n+r}/U_{n+r+1}$ for any $r\geq 1$. So the differentials
 \[d_r \colon E_r^{2t, n} \to E_r^{2t+1, n+r}\]
 are all trivial.

Now let $s=2t+1$ for odd $t\geq 0$. A choice of representative for a class
\[(0, f(\mymu)) \in  \fk \oplus \fk[\![\mymu]\!] \cong E_1^{2t+1,n} \] 
is given by $1+2^n\widetilde{f}(\mymu)\mu_0$ in $P^{2t+1}(U_n/U_{n+1})$, which lifts to
\[1+2^n \widetilde{f}(\mymu)\mu_0 \in P^{2t+1}(U_n/U_{n+r+1}).\]
We compute
\[d = (\gamma(-)(-))^2\colon P^{2t+1}(U_n/U_{n+r+1}) \to P^{2t+2}(U_n/U_{n+r+1}). \]
Using that $\mymu= \gamma(\mu_0)+\mu_0 $, we have
\begin{align*}
d(1+2^n \widetilde{f}(\mymu)\mu_0) &=  \gamma(1+2^n \widetilde{f}(\mymu)\mu_0)^2(1+2^n \widetilde{f}(\mymu)\mu_0 )^{2}  \\
&=  ((1+2^n \widetilde{f}(\mymu)\gamma(\mu_0))(1+2^n \widetilde{f}(\mymu)\mu_0) )^{2}  \\
&\equiv  1+2^{n+1} \mymu\widetilde{f}(\mymu) + 2^{2n} \mymu^2\widetilde{f}(\mymu)^2 \mod 2^{n+2}  \\
\end{align*}
Therefore, on $\fk[\![\mymu]\!] \subseteq  \fk \oplus \fk[\![\mymu]\!] \cong E_1^{2t+1,n}$, we have differentials 
\[d_1 \colon  E_1^{2t+1,n} \to E_1^{2t+2,n} \cong \fk[\![\mymu]\!]\] 
given by
\[d_1(f(\mymu))=\begin{cases} \mu f(\mymu)+\mymu^2 f(\mymu)^2 & n=1 \\
\mu f(\mymu) & n>1
\end{cases}\]
which in both cases are isomorphisms onto $\mymu\fk[\![\mymu]\!]$. 

A choice of representative for a class
\[(\alpha, 0) \in  \fk \oplus \fk[\![\mymu]\!] \cong E_1^{2t+1,n} \] 
is given by $1+2^n\alpha$ in $P^{2t+1}(U_n/U_{n+1})$,
which lifts to
\[1+2^n\widetilde{\alpha} \in P^{2t+1}(U_n/U_{n+r+1}),\]
Since this class is invariant for the action of $C_4$, we have that
\begin{align*}
d(1+2^n\widetilde{\alpha} ) &=(1+2^n\widetilde{\alpha} )^4   \\
&\equiv 1+2^{n+2}\widetilde{\alpha} +2^{2n+1} \widetilde{\alpha}^2 \mod 2^{n+3} 
\end{align*}
Therefore, on $\fk \subseteq  \fk \oplus \fk[\![\mymu]\!] \cong E_1^{2t+1,n}$, we have differentials 
\[d_1 \colon  E_1^{2t+1,n} \to E_1^{2t+2,n} \cong \fk[\![\mymu]\!]\] 
given by
\[d_1(\alpha)=\begin{cases} \alpha+\alpha^2 & n=1 \\
\alpha & n>1.
\end{cases}\]
These are isomorphisms if $n\geq 2$, and if $n=1$, the kernel is $\F_2 = \fk^{\Gal(\fk/\F_2)}$. 

In particular, it follows that
\[E_{\infty}^{1,n}  = \begin{cases} \F_2 & n=1 \\
0 &\text{otherwise}
\end{cases}\]
which implies the claim.
\end{proof}

\begin{center}
\begin{figure}[H]
\center
\includegraphics[height=0.2\textheight]{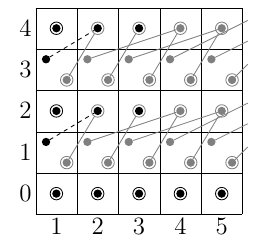}
\includegraphics[height=0.2\textheight]{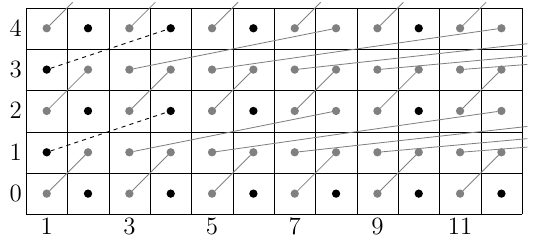}
\captionsetup{width=\textwidth}
\caption{The spectral sequence \eqref{eq:bockU} (top) and the spectral sequence \eqref{eq:bockV} (bottom), drawn in the $(n,s)$-plane. A $\mdsmblkcircle$ denotes a copy of $\fk$.  
A $\circledbullet$ denotes a copy of $\fk[\![\mymu]\!]$, where the inner dot stands for $\fk\{1\}$ and the outer rim stands for $\mymu\fk[\![\mymu]\!]$ (this allows us to draw the differentials more precisely). The dashed line indicates that $\alpha \mapsto \alpha +\alpha^2$ for $\alpha \in \fk$.}
\label{fig:bockV}
\label{fig:bockU}
\end{figure}
\end{center}

\begin{prop}
There is an isomorphism $H^1(C_4, E_0^{\times}/2) \cong \Z/2$.
\end{prop}
\begin{proof}
To compute this, we will use another Bockstein spectral sequence associated to the following filtration. Let
\[V_n=\{x \in E_0^{\times}/2 : x\equiv 1 \mod \mu_0^n\} \]
There is an exact sequence
 \[ 1 \to V_1 \to E_0^{\times}/2 \to \fk^{\times} \to 1. \]
 Since $\fk^{\times}$ has order prime to $2$ and $H^0(C_4, (E_0/2)^{\times}) \to H^0(C_4, \fk^{\times})$ is surjective, 
  there is an isomorphism
 \[ H^1(C_4, E_0^{\times}/2)  \cong H^1(C_4, V_1 ).\]
 To compute  $H^1(C_4, V_1)$, we use the Bockstein spectral sequence
 \begin{equation}\label{eq:bockV}
 E_1^{s,n} = H^s(C_4, V_n/V_{n+1} ) \Longrightarrow  H^s(C_4, V_1 ) \end{equation}
 with differentials $d_r \colon E_r^{s,n}  \to E_r^{s+1,n+r} $. See \Cref{fig:bockV}. The maps
 \begin{align*} 
 \phi_n &\colon \fk \to  V_n/V_{n+1}, &  \phi_n(\alpha) = 1+\mu_0^n\alpha
   \end{align*}
are isomorphisms for $n\geq 1$. We proceed as in the proof of \Cref{prop:H1U1}, using the resolution \eqref{eq:cochain} to compute the differentials.
We have isomorphisms
\[ H^s(C_4, V_n/V_{n+1}) \cong  \fk  \]
where a representative for $\alpha \in \fk$ in $P^s(V_n/V_{n+1})$ is given by the residue class of $1+\mu_0^n \alpha$, which is also a choice of lift in $P^s(V_n/V_{n+r+1})$.

Let $s=2t$ and $n=2^km$ for $m$ odd and $k\geq 0$. For $d \colon P^{2t}(V_n) \to P^{2t+1}(V_n) $, we have
\begin{align*}
d( 1+\mu_0^n  \alpha)  &=  (1+\gamma(\mu_0)^n \alpha)( 1+\mu_0^n  \alpha)^{-1}  \\
&= \left( 1+\frac{\mu_0^n}{(\mu_0+1)^n}\alpha\right)\left(\frac{1}{1+\mu_0^n\alpha}\right) \\
&= \left( 1+\alpha \mu_0^n\frac{1}{(\mu_0^{2^k}+1)^m}\right)\left(\frac{1}{1+\mu_0^n\alpha}\right) \\
&\equiv \left( 1+\alpha \mu_0^n+ \alpha \mu_0^{n+2^k} \right)\left(1+\alpha\mu_0^n + \alpha^2\mu_0^{2n} \right) \\
&\equiv  1+\alpha \mu_0^{n+2^k} \mod \mu_0^{n+2^k+1}.
\end{align*}
Therefore, for $n=2^km$, the first possible non-zero differential $d_r \colon E_r^{2t, n} \to E_r^{2t+1, n+1}$ is $d_{2^k}$. Further, the differentials
\[d_1 \colon E_1^{2t, n} \to E_1^{2t+1, n+1}\]
are isomorphisms for $n$ odd.

Now let $s=2t+1$ and $n=2^km$ where $m$ is odd and $k\geq 0$. 
For $d \colon P^{2t+1}(V_n) \to P^{2t+2}(V_n) $,  
we have
\begin{align*}
d( 1+\mu_0^n  \alpha) &= ( 1+\gamma(\mu_0)^n  \alpha)^2 ( 1+\mu_0^n  \alpha)^2  \\
&=  \left( 1+\alpha^2 \mu_0^{2n} \frac{1}{(\mu_0^{2^{k+1}}+1)^m}\right) ( 1+ \alpha^2 \mu_0^{2n}  ) \\
&=  \left( 1+\alpha^2 \mu_0^{2n}( 1+\mu_0^{2^{k+1}}) \right) ( 1+ \alpha^2 \mu_0^{2n}  )  \mod \mu_0^{2n+2^{k+1}+1}\\
&= \begin{cases} 1+ (\alpha+\alpha^2) \mu_0^{2^{k+2}} &n=2^k \\
1+ \alpha^2 \mu_0^{ 2n + 2^{k+1}} &n=2^km, \ m\neq 1 . 
\end{cases}
\end{align*}
In particular, if $n$ is odd, we get differentials $d_{n+2} \colon E_{n+2}^{2t+1, n} \to E_{n+2}^{2t+1, 2n+2} $ given by
\[d_{n+2}(\alpha) = \begin{cases} \alpha+\alpha^2 & n=1 \\
\alpha^2 & n>1.
\end{cases}\]
These are isomorphisms if $n\neq 1$. If $n=1$, then $\F_2 = \fk^{\Gal} \subseteq E_{3}^{2t+1, 1}$ is the kernel. 

Combining these results, we conclude that
\[E_{\infty}^{1,n}  = \begin{cases} \F_2 & n=1 \\
0 &\text{otherwise}
\end{cases}\]
which implies the claim.
\end{proof}


\section{The Picard Spectral Sequence}\label{sec:pic}

In this section, we first establish notation and a few results about the Picard spectral sequence for $R$ an even periodic ring spectrum in the category of genuine $G$ spectra. Then we turn to the analysis we need to prove \Cref{thm:fullpicmackey} in the case when $R=E$ and $G=C_4$.

\subsection{Generalities on the Picard Spectral Sequence}\label{sec:picss}
Let $R$ be an even periodic cofree commutative ring spectrum in the category of genuine $G$-spectra. Suppose that $R^{hG} \to R$ is a faithful $G$-Galois extension. We recall the tools provided by \cite{HeMaSt} and \cite{MathewStojanoska} to compute $\Pic(R^{hG})$, the Picard group of $R^{hG}$-module spectra. Let $\pic(R)$ denote the Picard spectrum of the ring spectrum $R$. Note that $\pic(R)$ is a spectrum with a $G$-action, and that $\pic(R^{hG}) = \pic(R)^{hG}_{ \geq 0}$. In particular 
\[\pi_0 \pic(R)^{hG} \cong \Pic(R^{hG}).\]

It follows that the group $\Pic(R^{hH})$ for any subgroup $H$ of $G$ can be computed by studying the spectral sequence 
\begin{equation}\label{eq:picss}
 \underline{E}_{\subpic,2}^{s,t}= \underline{H}^s(G, \pi_{t} \pic(R)) \Longrightarrow {\spi}_{t-s} (\pic(R))^{h}
\end{equation}
with differentials $d_{\subpic,r}^{s,t}\colon \underline{E}_{\subpic,r}^{s,t} \to \underline{E}_{\subpic,r}^{s+r, t+r-1}$; to obtain a Mackey-valued spectral sequence, we have taken the genuine cofree $G$-spectrum corresponding to $\pic(R)$.
We will be comparing this to the analogous homotopy fixed point spectral sequences $\underline{E}_\times$ (for the units of the ring $R$,  see \eqref{eq:MultSS}) and $\underline{E}_+$ (for the ring $R$ itself, see \eqref{eq:AddSS}).

Note that $\pic(R)$ is a connective spectrum with the property that $\left(\Omega \pic(R)\right)_{ \geq 0} \simeq \mathrm{gl}_1(R)$. Further, as spaces, $\Omega^{\infty}\mathrm{gl}_1(R) \simeq GL_1(R)$ and there is a map
\[ GL_1(R) \to \Omega^{\infty}R\]
which is an inclusion of those components lying over $(\pi_0R)^{\times}$. These equivalences respect the $G$ action, so both of these spaces inherit a $G$ action from $R$ and
\[ \pi_t \pic(R) \cong \begin{cases} \Z/2 & t=0 \\
 R_0^{\times} & t=1 \\
 R_{t-1} & t\geq 2 
 \end{cases}\]
as $G$-modules. It follows that
\[ \underline{E}_{\subpic,2}^{s,t} \cong \begin{cases}  \underline{H}^s(G, \Z/2) & t=0 \\
\underline{H}^s(G,  R_0^{\times}) & t=1 \\
\underline{E}_{+,2}^{s,t-1}  & t\geq 2
 \end{cases}\]
 as Mackey functors. Here, $\underline{E}_{+,*}^{*,*}$ denotes the Mackey functor homotopy fixed point spectral sequence
 \begin{align}\label{eq:AddSS}
 \underline{E}_2^{s,t} = \underline{E}_{+,2}^{s,t}  =\underline{H}^s(G,\pi_{t} R) \Rightarrow {\spi}_{t-s}R^{h}
 \end{align}
with
\[ d_{+,r}^{s,t} \colon \underline{E}_{r,+}^{s,t} \to \underline{E}_{r,+}^{s+r,t+r-1}. \] 
  We also let 
 \begin{align}\label{eq:MultSS}
 \underline{E}_{\times,2}^{s,t} =  \underline{H}^s(G, \pi_{t}  \mathrm{gl}_1(R)) \Rightarrow {\spi}_{t-s}(\mathrm{gl}_1(R))^{h}
 \end{align}
 with differentials $d_{\times,r}^{s,t}\colon \underline{E}_{\times,r}^{s,t} \to \underline{E}_{\times,r}^{s+r, t+r-1}$ and note that 
 \[ \underline{E}_{\times,2}^{s,t} \cong \begin{cases} \underline{H}^s(G, R_0^{\times})  & t=0 \\
\underline{E}_{+,2}^{s,t}  & t\geq 1.
 \end{cases}\]
 
 In  \cite[5.2.4]{MathewStojanoska}, Mathew-Stojanoska identify a range where the differentials $d_{\subpic}$ and $d_+$ are related.  Given a class $x\in \underline{E}_{+,r}^{s,t-1}(G/H)$ where $t\geq 2$, we let $x^{\times}$ and $x^{\subpic}$
 be the corresponding elements in $ \underline{E}_{\times,r}^{s,t-1}(G/H)$ and $ \underline{E}_{\subpic,r}^{s,t}(G/H)$ respectively. 
\begin{thm}[Mathew-Stojanoska]\label{thm:diffpic}
Let $x\in \underline{E}_{+, r}^{s,t-1}(G/H)$ and let ${y} \in \underline{E}_{+, r}^{s+r,t+r-2}(G/H)$.
Let ${x}^{\subpic} \in \underline{E}_{\subpic, r}^{s,t}(G/H)$ and ${y}^{\subpic} \in \underline{E}_{\subpic, r}^{s+r,t+r-1}(G/H)$ be the corresponding classes.
\begin{enumerate}
\item\label{diffpic1} If both $x^{\subpic}$ and ${y^{\subpic}}$ lie in the region of the $(t-s,s)$-plane where $2\leq t$ and $0\leq t-s$, then 
$d_{+,r}(x) = y$ if and only if $d_{\subpic, r}(x^{\subpic}) =y^{\subpic}$.
\item\label{diffpic2} If $2\leq t$ and $2\leq r \leq t-1$, then $d_{+,r}(x) = y$ if and only if $d_{\subpic,r}(x^{\subpic}) = y^{\subpic}$.
\item\label{diffpic3} If $s =t=r$ and $d_{+,r}(x) = y$, then
\[ d_{\subpic, t}(x^{\subpic}) = ({d_{+,t}(x) + x^2})^{\subpic}.\]
\end{enumerate}
\end{thm}

As in \cite{MathewStojanoska}, we will call the first two families of differentials \emph{stable} and the third family \emph{unstable}.

For our arguments below, we also need to know how the transfers and restrictions of these spectral sequences are related. We will need the following result but postpone its proof to the end of the section.

\begin{lem}\label{boringboundary}
Suppose $X$ is a spectrum with a $G$-action, and consider its Postnikov decomposition
\[ X_{\geq t} \to X \to X_{<t}\]
for some $t$. Let $\delta$ be the connecting map $X_{<t} \to \Sigma X_{\geq t}$. 

If $\alpha$ is a permanent cycle in the homotopy fixed point spectral sequence for $X_{<t}$, then 
\[\delta(\alpha) = \beta \in \pi_* (X_{\geq t})^{hG}\]
if and only if there is a differential $d_r(\alpha)=\beta$ of suitable length in the homotopy fixed point spectral sequence for $X$.
\end{lem}

In the following result and its proof,
we adopt the convention that, for $a<b$,
\[\underline{E}^{s,t}_{\times, [a,b], 2}(G/H)
= \underline{H}^s(G, \pi_t\mathrm{gl}_1R_{[a,b]})
\Longrightarrow \spi_{t-s} (\mathrm{gl}_1R_{[a,b]})^h  \]
and
\[\underline{E}^{s,t}_{\subpic, [a,b], 2}(G/H)
= \underline{H}^s(G, \pi_t\pic(R)_{[a,b]})
\Longrightarrow  \spi_{t-s} (\pic(R)_{[a,b]})^h  .\]

\begin{prop}\label{prop:GLtransfers}
Suppose $0<a<b<c$ are integers such that $c\leq 2b-1$, and let $H$ be a subgroup of $G$. Assume we are given classes $x \in \underline{E}^{s,m+s}_{+,r}(G/H)$ and $y \in \underline{E}^{s+p,m+s+p}_{+,r}(G/G)$, such that the following conditions are satisfied:
\begin{enumerate}
\item\label{firstass} The integers $m+s, m+s+p$ are in the interval $[b,c]$; the classes $x$ and $y$ are permanent cycles in the spectral sequence for $R_{[b,c]}$, and, in the Mackey functor $\spi_m (R_{[b,c]})$, we have $tr_H^G(x) = y$.
\item\label{secondasspic} The class $x^\subpic$ corresponding to $x$ is an $r$-cycle in $\underline{E}^{s,m+s+1}_{\subpic, [a+1,c+1], r}(G/H)$, which is hit by a differential, say $d_{r'}(z^\subpic) = x^\subpic$, with $r' \geq r$ and $r' > m+s+1-b$.
\end{enumerate}
Then in the spectral sequence $\underline{E}_{\subpic, [a+1,c+1],*}^{*,*}(G/G)$, there is a differential of suitable length 
\[d_{r'+p}(tr_H^G(z^\subpic)) = y^\subpic.\] 
\end{prop}

\begin{proof}
Note that 
\begin{equation}\label{eq:eqfortransferpic}\Omega (\pic(R)_{[a+1,c+1]}) \simeq \mathrm{gl}_1(R)_{[a,c]},\end{equation}
so the claims are going to follow from their (suitably shifted) counterparts for $\mathrm{gl}_1(R)$ and the spectral sequence $\underline{E}^{*,*}_\times$.

The assumption \eqref{secondasspic} implies
\begin{enumerate}[(2')]
\item\label{secondass} The class $x^\times$ corresponding to $x$ is an $r$-cycle in $\underline{E}^{s,m+s}_{\times, [a,c], r}(G/H)$, which is hit by a differential, say $d_{r'}(z^\times) = x^\times$, with $r' \geq r$ and $r' > m+s+1-b$.
\end{enumerate}
From this, we will deduce that, in the spectral sequence $\underline{E}_{\times, [a,c],*}^{*,*}(G/G)$, there is a differential of suitable length 
\[d_{r'+p}(tr_H^G(z^\times)) = y^\times.\] 
The result for $\underline{E}_{\subpic,[a+1,c+1]}^{*,*}$ then follows from \eqref{eq:eqfortransferpic}.

The assumption $c\leq 2b-1$ gives an equivalence $\mathrm{gl}_1R_{[b,c]} \simeq R_{[b,c]}$ which respects the $G$-action \cite[Corollary 5.2.3]{MathewStojanoska}. This equivalence gives an isomorphism of the respective spectral sequences and isomorphisms $\spi_*\mathrm{gl}_1R_{[b,c]} \cong \spi_* R_{[b,c]} $ of  Mackey functors. In particular, \eqref{firstass} gives that the classes $x^\times$ and $y^\times$ corresponding to $x$ and $y$ are permanent cycles in the spectral sequence for $\mathrm{gl}_1R_{[b,c]}$, and we have
\[ tr_H^G(x^\times_{[b,c]}) = y^\times_{[b,c]}\] 
in $\spi_m \mathrm{gl}_1R_{[b,c]}.$

Now consider the fiber sequence
\[ \mathrm{gl}_1 R_{[b,c]} \to \mathrm{gl}_1 R_{[a,c]} \to  \mathrm{gl}_1 R_{[a,b-1]},\]
which is of the form needed in \Cref{boringboundary} with $X= \mathrm{gl}_1R_{[a,c]}$ and $t=b$. 

The class $z^\times $ is in $\underline{E}_{\times, [a,c], r'}^{s-r', m+s-r'+1}$. The condition on $r'$ ensures that there is a corresponding class $z^\times_{[a,b-1]}$ in $\underline{E}_{\times, [a,b-1], r'}^{s-r', m+s-r'+1}$, which must be a permanent cycle.

First we apply \Cref{boringboundary} to $\alpha = z^\times_{[a,b-1]}$ to conclude that $\delta(z^\times_{[a,b-1]})= x^\times_{[b,c]}$. But $\delta$ is a $G$-map of spectra, so it commutes with transfers, giving that
\[\delta(tr_H^G(z^\times_{[a,b-1]})) = tr_H^G x^\times_{[b,c]} = y^\times_{[b,c]}. \]

Using \Cref{boringboundary} again in the other direction, we conclude there is a differential of suitable length in the HFPSS for $\mathrm{gl}_1R_{[a,c]}$
taking the class corresponding to $tr_H^G(z^\times_{[a,b-1]})$ to the class corresponding to $y^\times$. This differential must be the one we claim.
\end{proof}

\begin{proof}[Proof of \Cref{boringboundary}.]
The key point is that the homotopy fixed point spectral sequence, usually obtained by filtering $EG$ by skeleta, is isomorphic (at the $E_2$-page and beyond) to the spectral sequence obtained from the Postnikov tower of $X$. For a reference, see \cite[Theorem 10.6]{GreenleesMay}. So, in the argument that follows, we are using the model of the HFPSS obtained from the Postnikov tower of $X$, i.e. the Bousfield-Kan spectral sequence for $\holim_q F_G(EG_+, X_{\leq q}) \simeq X^{hG}$.

Suppose indeed that $\alpha$ is a permanent cycle in the HFPSS for $X_{<t}$; by suitably suspending if necessary, this means that $\alpha$ defines an equivariant map 
\[\alpha \colon EG_+ \to X_{<t}.\] 
The diagram
\[ \xymatrix{ & EG_+\ar[d]_\alpha \\
X \ar[r] & X_{<t} \ar[r]^{\delta} &  \Sigma X_{\geq t} }\]
makes it clear that $\delta(\alpha)$ is the obstruction to lifting $\alpha$ to $X$. 

All that is needed is to be more specific about where the obstruction occurs when it is non-zero. So, suppose $\alpha$ can be lifted to a map $EG_+ \to X_{<t+m}$  (which we also denote by $\alpha$), but not further, so that the differential $d(\alpha)=:\beta$ is the composite
\[ EG_+ \xrightarrow{\alpha} X_{<t+m} \to \Sigma^{t+m+1} H\pi_{t+m}X. \]
To show that $\delta(\alpha)$ is detected by $\beta$, it suffices to check that the diagram 
\[ \xymatrix{
EG_+ \ar[r]^-{\beta} \ar[d]_{\alpha} & \Sigma^{t+m+1} H\pi_{t+m} X \ar[d]^c\\
X_{<t} \ar[r]_-b & \Sigma X_{[t,t+m+1)}
}\]
commutes, where the bottom horizontal map is the boundary associated to a Postnikov stage of $X_{<t+m+1}$, and the right-hand vertical map is the connective cover. Now we will use the Postnikov towers for the sequence 
\[ \xymatrix{ 
X \ar[r] & X_{<t} \ar[r]^{\delta} &  \Sigma X_{\geq t} }\]
to obtain the desired conclusion.

Since $\beta$ factors as $\alpha \colon EG_+ \to X_{<t+m}$ composed with the top horizontal arrow below, it suffices to show that the diagram
\[ \xymatrix{
X_{<t+m} \ar[r] \ar[d]_-{\tau_{<t}}& \Sigma^{t+m+1} H\pi_{t+m} X \ar[d]^-c\\
 X_{<t} \ar[r]_-b & \Sigma X_{[t,t+m+1)}
}\]
commutes.
Note that the vertical maps are connective covers, and the horizontal maps are boundary maps in some Postnikov decompositions. So this diagram commutes because if we back up these Postnikov decomposition sequences, the diagrams involved will commute:
\[ \xymatrix{
\Sigma^{t+m}H\pi_{t+m}X \ar[r]^{\tau_{\geq t+m}} \ar[d]_c& X_{<t+m+1} \ar[r]^{\tau_{<t+m}} \ar@{=}[d] &X_{<t+m} \ar[r] \ar[d]_{\tau_{<t}}& \Sigma^{t+m+1} H\pi_{t+m} X \ar[d]^-{c}\\
X_{[t,t+m+1)} \ar[r]^{\tau_{\geq t}} & X_{< t+m+1} \ar[r]^{\tau_{<t}} &X_{<t} \ar[r]_-b & \Sigma X_{[t,t+m+1)} \ .
}\]

For the converse, the same ingredients go in the argument, just in the opposite order. 
\end{proof}

\subsection{The Picard Spectral Sequence for $E$ and $C_4$}
In this section, we study the Picard spectral sequence 
\[ \underline{E}_{\subpic,2}^{s,t}= \underline{H}^s(C_4, \pi_{t} \pic(E)) \Longrightarrow {\spi}_{t-s} (\pic(E))^{h}\]
with differentials $d_{\subpic,r} \colon \underline{E}_{\subpic,r}^{s,t} \to \underline{E}_{\subpic,r}^{s+r, t+r-1}$. 
In this section, we let
\[ \underline{E}_{+,2}^{s,t}  =\underline{H}^s(C_4,\pi_{t} E_0) \Rightarrow {\spi}_{t-s}E^{h}\]
with
$d_{+,r} \colon \underline{E}_{r,+}^{s,t} \to \underline{E}_{r,+}^{s+r,t+r-1}.$
From \Cref{sec:picss}, we have that
\[ \underline{E}_{\subpic,2}^{s,t} \cong \begin{cases}  \underline{H}^s(C_4 ,  \Z/2) & t=0 \\
\underline{H}^s(C_4 , E_0^{\times}) & t=1 \\
\underline{E}_{+,2}^{s,t-1}  & t\geq 2.
 \end{cases}\]

We prove the following result, which is illustrated in \Cref{PicE2C4}. 
\begin{prop}\label{prop:orderbound}
The order of $\Pic(E^{hC_4})$ is at most $64$ and that of $\Pic(E^{hC_2})$ is at most $16$. 
\end{prop}

\begin{figure}[t]
\includegraphics[width=0.8\textwidth]{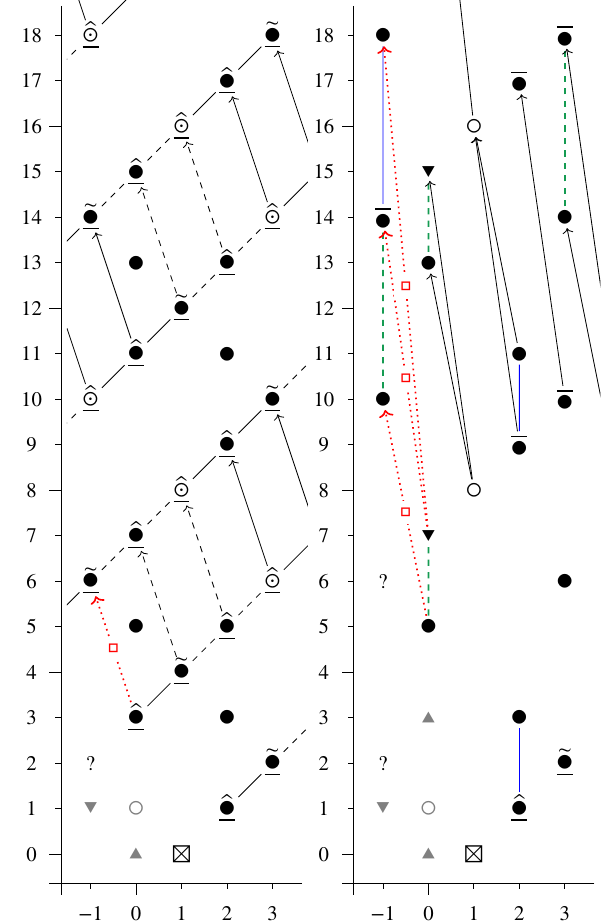} 
\caption{The homotopy fixed point spectral sequence computing $\spi_{*}\pic(E)^{h}$. See \Cref{table:Mackey-pic} and \Cref{table:Mackey-two} for the definitions of the various Mackey functors.
The dotted red $d_3$, $d_5$ and $d_7$ differentials come from \eqref{diffpic3} of \Cref{thm:diffpic}. The dotted red $d_{11}$ is a consequence of the $d_7$ and \Cref{prop:GLtransfers}. Dashed green lines are exotic restrictions and solid blue line exotic transfers.}
\label{PicE2C4}
\end{figure}

\begin{proof}
We prove this by giving an upper bound on classes which survive in stem $t-s=0$ in the spectral sequence $ \underline{E}_{\subpic,2}^{s,t}$. The range of interest is $-2 \leq t-s \leq 1$. The Mackey functors we use here are defined either in \Cref{table:Mackey-pic}  or in \Cref{table:Mackey}, \Cref{table:Mackey-two}, and \Cref{table:Mackey-three}.

From \eqref{diffpic1} and \eqref{diffpic2} of \Cref{thm:diffpic}, it follows that many differentials are forced by those in $\underline{E}_{*,+}^{*,*} $. We do not discuss these and focus on those differentials that follow from \eqref{diffpic3} and \Cref{prop:GLtransfers}.

The first interesting differential is $d_{\subpic,3}: \underline{E}_{\subpic, 3}^{3,3} \to \underline{E}_{\subpic, 3}^{6,5}$, whose data is represented as the short exact sequence
\[\xymatrix{0 \ar[r] & \mackeypictriangle \ar[r] &  \mackeyunderlinedbullethat \ar[r]^-{d_{\subpic,3}} & \mackeyunderlinedbullettilde  }. \]
To justify this, note that the source of this differential is exactly that for which \eqref{diffpic3} of \Cref{thm:diffpic} applies so that
\[d_{\subpic,3}(x^{\subpic}) = (d_{+,3}(x) +x^2)^{\subpic}. \]
We have
\begin{align*}
\xymatrix@R=0.7pc{\underline{E}_{+,3}^{3,2}(C_4/C_4) = \fk[\![\mu]\!]\{ \eta \varpi \Delta_1^{-1}\}  \ar[r]^-{d_{+,3}} & \underline{E}_{+,3}^{6,4}(C_4/C_4) = \fk[\![\mu]\!]\{ \eta^2 \varpi^2 \Delta_1^{-2}\}   \\
\underline{E}_{+,3}^{3,2}(C_4/C_2) = \fk[\![\mu_0]\!]\{ \eta_0^3 \Sigma_{2,0}^{-1}\} \ar[r]^-{d_{+,3}} & \underline{E}_{+,3}^{6,4}(C_4/C_2)   =\fk[\![\mu_0]\!]\{ \eta_0^6 \Sigma_{2,0}^{-2}\}  .}
\end{align*}
See \Cref{sec:genE2}. From \Cref{prop:d3s}, we have
\begin{align*}
d_{+,3}(f(\mu) \eta \varpi \Delta_1^{-1}) &= f(\mu) \eta^2 \varpi^2 \Delta_1^{-2} \\
d_{+,3}(f(\mu_0) \eta_0^3 \Sigma_{2,0}^{-1}) &= \mu_0f(\mu_0) \eta_0^6 \Sigma_{2,0}^{-2}  .
\end{align*}
So, for $f(\mu) \in \fk [\![\mu]\!]$ and $g(\mu_0) \in \fk [\![\mu_0]\!]$
\begin{align*}
d_{\subpic,3}((f(\mu) \eta \varpi \Delta_1^{-1})^{\subpic}) &= ((f(\mu)+f(\mu)^2) \eta^2 \varpi^2 \Delta_1^{-2}   )^{\subpic} \\
d_{\subpic,3}((g(\mu_0) \eta_0^3 \Sigma_{2,0}^{-1})^{\subpic}) &=( (\mu_0g(\mu_0)+g(\mu_0)^2) \eta_0^6 \Sigma_{2,0}^{-2} )^{\subpic}  .
\end{align*}
The first differential is zero if and only if $f(\mu) \in \F_2$. The second is zero if and only if $g(\mu_0) \in \F_2\{\mu_0\}$. In both cases, the kernel is isomorphic to $\Z/2$. It assembles into the Mackey functor $\mackeypictriangle$ depicted in \Cref{table:Mackey-pic}.

The $d_5$ differential of interest is $d_{\subpic, 5}: \underline{E}_{\subpic, 5}^{5,5} \to \underline{E}_{\subpic, 5}^{10,9}$
\[
\xymatrix{ 0  \ar[r] & \mackeypicbullet \ar[r] & \mackeybullet  \ar[r]^-{d_{\subpic,5}} & \mackeybullet  }   \]
These Mackey functors are only non-zero when evaluated at $C_4/C_4$.
From \Cref{prop:d3s}, we get
\begin{align*}
\xymatrix@R=0.7pc{\underline{E}_{+,5}^{5,4}(C_4/C_4) =\fk\{\nu \varpi^2\Delta_1^{-2}\}  \ar[r]^-{d_{+,5}} & \underline{E}_{+,5}^{10,8}(C_4/C_4)=\fk\{\nu^2 \varpi^4\Delta_1^{-4}\}    }
\end{align*}
given by
\begin{align*}
d_{+,5}(\alpha \nu \varpi^2\Delta_1^{-2} ) &= \alpha \nu^2 \varpi^4\Delta_1^{-4}
\end{align*}
 for $\alpha \in \fk$. So, 
\begin{align*}
d_{\subpic,5}((\alpha \nu \varpi^2\Delta_1^{-2})^{\subpic}) &= ( (\alpha +\alpha^2)  \nu^2 \varpi^4\Delta_1^{-4} )^{\subpic}.
\end{align*}
This is zero if and only if $\alpha \in \F_2$ and so the kernel is isomorphic to $\Z/2$. 

Next, we turn to the $d_7$ and closely related $d_{11}$ differentials.
We have
\begin{align*}
\xymatrix@R=0.7pc{
\underline{E}_{+,7}^{7,6}(C_4/C_2) =\fk\{ \eta_0^7\Sigma_{2,0}^{-2} \}  \ar[r]^-{d_{+,7}}  & \underline{E}_{+,7}^{14,12}(C_4/C_2)=\fk\{ \eta_0^{14}\Sigma_{2,0}^{-4}\}   \\
\underline{E}_{+,11}^{7,6}(C_4/C_4) =\fk\{ \varsigma \varpi^3 \Delta_1^{-3}\}  \ar[r]^-{d_{+,11}}  & \underline{E}_{+,11}^{18,16}(C_4/C_4)= \fk\{\nu^2\varpi^8 \Delta_1^{-7}\}   
  }
\end{align*}
given by
\begin{align*}
d_{+,7}(\alpha \eta_0^7\Sigma_{2,0}^{-2}) &= \alpha \eta_0^{14}\Sigma_{2,0}^{-4} \\
d_{+,11}(\alpha\varsigma \varpi^3 \Delta_1^{-3}) &= \alpha \nu^2\varpi^8 \Delta_1^{-7}
\end{align*}
for $\alpha \in \fk$. (The $d_{+,11}$ follows from $d_{+,11}(\varsigma \varpi) = \nu^2\varpi^6 \Delta_1^{-4}$ by linearity with respect to $\Delta_1^{-4}$ and $\bar{\kappa}  = \varpi^2 \Delta_1$.)
This gives 
\begin{align*}
d_{\subpic,7}((\alpha \eta_0^7\Sigma_{2,0}^{-2})^{\subpic}) &= ( (\alpha +\alpha^2) \eta_0^{14}\Sigma_{2,0}^{-4} )^{\subpic}
\end{align*}
whose kernel is isomorphic to $\Z/2$.

Next, we apply \Cref{prop:GLtransfers} to the following setup. There are 
transfers
\[ \tr_2^{4}(  \alpha  \eta_0^7\Sigma_{2,0}^{-2}) = \alpha  \varsigma \varpi^3 \Delta_1^{-3} \]
and
 exotic transfers
\[ \tr_2^{4}(  \alpha \eta_0^{14}\Sigma_{2,0}^{-4}) =\alpha\nu^2\varpi^8 \Delta_1^{-7} \]
for $\alpha\in \fk$ which raise filtration by $s=4$. These ``combine'' 
$\underline{E}_{+, 6}^{14,12}(C_4/C_2)$ and $ \underline{E}_{+,6}^{18,16}(C_4/C_4)$ (for example, in the homotopy groups of $\spi_*E_{[10,18]}$)
as:
\begin{align*}
\xymatrix{0 \ar[r] & \mackeybullet \ar[r] & \mackeyblacktriangledown \ar[r] & \mackeybulletoverline \ar[r] & 0 }.
\end{align*}
In \Cref{prop:GLtransfers}, let $s=14$, $m=-2$, $p=4$, $a=4$, $b=12$, $c=16$, $H=C_2$ and $G=C_4$. Let $r=6$ and $r'=7$. Let $x=  \alpha \eta_0^{14}\Sigma_{2,0}^{-4} \in \underline{E}_{+,6}^{14,12}(C_4/C_2)$, $y=\alpha\nu^2\varpi^8 \Delta_1^{-7} \in  \underline{E}_{+,6}^{18,16}(C_4/C_4)$ and $z= \alpha  \eta_0^7\Sigma_{2,0}^{-2} \in \underline{E}_{+,7}^{7,6}(C_4/C_2)$. The conditions of \Cref{prop:GLtransfers} are then satisfied and it follows that
in the spectral sequence $\underline{E}_{\subpic, [5,17], *}^{*,*}(C_4/C_4)$ 
we have a differential
\begin{align*}
 d_{\subpic,11}( (\alpha  \varsigma \varpi^3 \Delta_1^{-3})^{\subpic})  &= d_{\subpic,11}( \tr_2^{4}((\alpha \eta_0^7\Sigma_{2,0}^{-2})^{\subpic})) \\
&=  \tr_2^{4}(d_{\subpic,7}((\alpha \eta_0^7\Sigma_{2,0}^{-2})^{\subpic}) )\\
&= \tr_2^{4}( ( (\alpha +\alpha^2) \eta_0^{14}\Sigma_{2,0}^{-4} )^{\subpic} ) \\
&= ((\alpha+\alpha^2)\nu^2\varpi^8 \Delta_1^{-7})^{\subpic} .
\end{align*}
From this, we deduce the same differentials in the spectral sequence $\underline{E}_{\subpic,*}^{*,*}$.
Again, such a differential is zero if and only if $\alpha \in \F_2$. So the kernel is isomorphic to $\Z/2$. 
Combining the $d_{\subpic, 7}$ and $d_{\subpic, 11}$ differentials gives an exact sequence
\[
\xymatrix{ 0  \ar[r] &\mackeypictriangledown \ar[r] & \mackeyblacktriangledown  \ar[rr]^-{d_{\subpic,7}/d_{\subpic,11}} & & \mackeyblacktriangledown \ .  }   \]

For $t \geq 8$, $\underline{E}_{+, \infty}^{t,t-1}=0$ and all differentials necessary to make this true are in the range where
 \eqref{diffpic1} and \eqref{diffpic2} of \Cref{thm:diffpic} apply. Therefore, $\underline{E}_{\subpic, \infty}^{t,t}=0$ if $t\geq 8$. So, the order of $\mPic(E)$ at $C_4/C_4$ and $C_4/C_2$ is bounded by the order of the direct sum of the following Mackey functors:
\begin{align*}
 \underline{E}_{\subpic, 2}^{0,0} &=\mackeypictriangle  &
 \underline{E}_{\subpic, 2}^{1,1} &=\mackeypiccirc &
  \underline{E}_{\subpic, 4}^{3,3} &=  \mackeypictriangle &
   \underline{E}_{\subpic, 6}^{5,5} &= \mackeypicbullet &
      \underline{E}_{\subpic, 12}^{7,7} &= \mackeypictriangledown .
\end{align*}
For $C_2$, this bounds the order of $\PicRG{E}{C_2}$ by $16$ and for $C_4$, it bounds the order of $\PicRG{E}{C_4}$ by $64$. 
\end{proof}

\begin{table}[H]
\begin{center}
{\footnotesize
\centering
\begin{tabular}{ |C|C|C|C|C|C|C|C| }
\hline
  \mackeypiccirc & \mackeypicalgfixpoints  & \mackeypictriangledown  &  \mackeypictriangle &  \mackeypicbullet &  \mackeypicbulletoverline    \\ 
 \hline
  \Mackey{\mathbb{Z}/4}{\Z/2}{0}{}{}{2}{1} 
  &\ \ \  \ \ \ \Mackey{\W[\![\mu ]\!]^{\times}}{\W[\![\mu_0 ]\!]^{\times}}{\W[\![\mu_0 ]\!]^{\times}}{(-)^2}{1}{\gamma(-)(-)}{1}
 &  \Mackey{\Z/2}{\Z/2}{0}{}{}{1}{0} 
 &\Mackey{\Z/2}{\Z/2}{0}{}{}{0}{1} 
 &    \Mackey{\Z/2}{0}{0}{}{}{}{} 
 &    \Mackey{0}{\Z/2}{0}{}{}{}{} 
 \\
 \hline
 \end{tabular}}
\end{center}
\caption{Mackey functors in the Picard Spectral Sequence.}
\label{table:Mackey-pic}
\end{table}

Combining \Cref{prop:orderbound} and \Cref{prop:image} gives \Cref{thm:fullpicmackey}. The transfers and restrictions in $\mPic(E)$ are computed using the formula
\[
{}^{E}N_{H}^{C_4}(i_{H}^{\ast}E\smsh S^{W}) \cong
E\smsh S^{\Ind_{H}^{C_4}W},
\]
where $H=e$ or $H=C_2$ and $W \in RO(H)$.


\end{document}